\documentclass[11pt]{article}
\usepackage{fullpage}
\RequirePackage{amsthm,amsmath,amsfonts,amssymb}
\RequirePackage[round]{natbib}
\usepackage{amsmath, amssymb, amsthm}
\usepackage{bm}
\usepackage{graphicx}
\usepackage{subfig}
\usepackage{algorithm}
\usepackage{algorithmic}
\graphicspath{{figs/}}
\usepackage{booktabs}

\usepackage[svgnames]{xcolor}
\usepackage{listings}
\usepackage{mathrsfs} 

\usepackage[citecolor=blue,urlcolor=black, linkcolor=black]{hyperref}
\usepackage{cleveref}

\usepackage{url}            
\usepackage{booktabs}       
\usepackage{amsfonts}       
\usepackage{nicefrac}       
\usepackage{microtype}      
\usepackage[shortlabels]{enumitem}
\usepackage{relsize}

\usepackage{tikz}

\numberwithin{equation}{section}

\DeclareMathOperator*{\plim}{plim}

\DeclareMathOperator*{\PP}{\mathbb{P}}
\DeclareMathOperator*{\E}{\mathbb{E}}
\DeclareMathOperator{\TPP}{\textup{TPP}}
\DeclareMathOperator{\FDP}{\textup{FDP}}

\DeclareMathOperator*{\fdpinf}{\textup{FDP}^{\infty}}
\DeclareMathOperator*{\tppinf}{\textup{TPP}^{\infty}}
\DeclareMathOperator{\prox}{\textup{prox}}
\DeclareMathOperator{\AVE}{\textup{AVE}}
\DeclareMathOperator{\supp}{supp}
\DeclareMathOperator{\sgn}{sign}

\newcommand{\R}{\mathbb{R}}
\newcommand{\Align}[1]{\begin{align}#1\end{align}}
\newcommand{\be}{\begin{equation}}
	\newcommand{\ee}{\end{equation}}
\newcommand{\ben}{\begin{equation*}}
	\newcommand{\een}{\end{equation*}}

\newcommand{\norm}[1]{\lVert#1\rVert}
\renewcommand{\mathbf}{\bm}
\newcommand\goto{\rightarrow}
\newcommand{\bet}{\boldsymbol{\beta}}
\newcommand{\p}{p}
\newcommand{\n}{n}

\newcommand{\Z}{\mathbf{Z}}
\newcommand{\X}{\mathbf{X}}
\newcommand{\w}{\mathbf{w}}
\newcommand{\y}{\mathbf{y}}

\newcommand{\blam}{\bm\lambda}
\newcommand{\thet}{\boldsymbol \theta}
\newcommand{\z}{\mathbf{z}}

\newcommand{\bfalph}{\boldsymbol \alpha}

\newcommand{\h}[1]{\mathlarger{\eta}_{{\textstyle\mathstrut}#1}}
\newcommand{\hprime}[1]{\mathlarger{\eta}'_{{\textstyle\mathstrut}#1}}
\newcommand{\N}{\mathcal{N}}
\renewcommand{\P}{\mathbb{P}}
\newcommand{\tppmax}{u^\star_{\textnormal{\tiny DT}}}
\renewcommand{\omega}{w}
\newcommand{\Mobius}{M\"obius }
\renewcommand{\AA}{\mathrm{A}_{\textnormal{eff}}}
\newcommand{\A}{\mathrm{A}}
\renewcommand{\o}{o}

\definecolor{dgreen}{rgb}{0.1,0.5,0.1}

\newcommand{\GAM}{\xi}

\newtheorem{theorem}{Theorem}
\crefname{theorem}{Theorem}{Theorem}
\newtheorem{othertheorem}{othertheorem}[section]
\newtheorem{lemma}[othertheorem]{Lemma}
\crefname{lemma}{Lemma}{Lemma}

\newtheorem{corollary}[othertheorem]{Corollary}
\crefname{corollary}{Corollary}{Corollary}
\newtheorem{proposition}[othertheorem]{Proposition}
\crefname{proposition}{Proposition}{Proposition}

\newtheorem{definition}[othertheorem]{Definition}
\crefname{definition}{Definition}{Definition}
\newtheorem{fact}[othertheorem]{Fact}
\crefname{fact}{Fact}{Fact}

\theoremstyle{remark}
\newtheorem{rem}[othertheorem]{Remark}

\title{
	Characterizing the SLOPE Trade-off: A Variational Perspective and the Donoho--Tanner Limit
}

\author{Zhiqi Bu\thanks{Author names are listed alphabetically.} \thanks{Graduate  Group  in  Applied  Mathematics  and  Computational  Science, University of Pennsylvania. Email: {\tt zbu@sas.upenn.edu}}
	\and Jason M. Klusowski\thanks{Department of Operations Research and Financial Engineering, Princeton University. Email: {\tt jason.klusowski@princeton.edu} } 
	\and Cynthia Rush\thanks{Department of Statistics, Columbia University. Email: {\tt cynthia.rush@columbia.edu} }
	\and Weijie J. Su\thanks{Department   of   Statistics and Data Science,   University   of   Pennsylvania. Email: {\tt suw@wharton.upenn.edu} }}

\begin{document}
	\maketitle
	
	\begin{abstract}
Sorted $\ell_1$ regularization has been incorporated into many methods for solving high-dimensional statistical estimation problems, including the SLOPE estimator in linear regression. In this paper, we study how this relatively new regularization technique improves variable selection by characterizing the optimal SLOPE trade-off between the false discovery proportion (FDP) and true positive proportion (TPP) or, equivalently, between measures of type I error and power. Assuming a regime of linear sparsity and working under Gaussian random designs, we obtain an upper bound on the optimal trade-off for SLOPE, showing its capability of breaking the Donoho--Tanner power limit. To put it into perspective, this limit is the highest possible power that the Lasso, which is perhaps the most popular $\ell_1$-based method, can achieve even with arbitrarily strong effect sizes. Next, we derive a tight lower bound that delineates the fundamental limit of sorted $\ell_1$ regularization in optimally trading the
FDP off for the TPP. Finally, we show that on any problem instance, SLOPE with a certain regularization sequence outperforms the Lasso, in the sense of having a smaller FDP, larger TPP and smaller $\ell_2$ estimation risk simultaneously. Our proofs are based on a novel technique that reduces a calculus of variations problem to a class of infinite-dimensional convex optimization problems and a very recent result from approximate message passing theory.
		
	\end{abstract}

	
	\section{Introduction}

Reconstructing the signal from noisy linear measurements is vital in many disciplines, including statistical learning, signal processing, and biomedical imaging. In many modern applications where the number of explanatory variables often exceeds the number of measurements, the signal is often believed---or, wished---to be sparse in the sense that most of its entries are zero or approximately zero. Put differently, this means that a majority of the explanatory variables are simply irrelevant to the response of interest. 


Accordingly, a host of methods have been developed to tackle these problems by leveraging the sparsity of signals in high-dimensional linear regression. These methods often rely on, among others, the concept of \textit{regularization} to constrain the search space of the unknown signals. Perhaps the most influential instantiation of this concept is $\ell_1$ regularization, which gives rise to the Lasso method~\citep{lasso_paper}. The optimal amount of regularization, however, hinges on the sparsity level of the signal. Intuitively speaking, if the sparsity level is low, then more regularization should be imposed, and vice versa (see, for example, \citet{abramovich2006adapting}).


This intuition necessitates the development of a regularization technique that is adaptive to the sparsity level of signals, which is typically unknown in practical problems. To achieve this desired adaptivity, \cite{SLOPE1} introduced \textit{sorted $\ell_1$ regularization}. This new regularization technique turns into a method called SLOPE in the setting of a linear regression model
\begin{equation}\label{eq:linear_model}
\y = \X\bet + \w,
\end{equation}
where $\X$ is the $\n\times \p$ design matrix, $\bet \in \R^p$ are the regression coefficients, $\y \in \R^n$ is the response, and $\w \in \R^n $ is the noise term. Explicitly, SLOPE estimates the coefficients by solving the convex programming problem
\begin{equation}\label{eq:SLOPE_cost}
\arg \min_{\bm b} \frac{1}{2}\|\y - \X\bm{b}\|^2 + \sum_{i=1}^p \lambda_i|b|_{(i)},
\end{equation}
where $|b|_{(1)} \ge \cdots \ge |b|_{(p)}$ are the order statistics in absolute value of $\bm b = (b_1, \ldots, b_p)$ and $\lambda_1\geq \cdots \geq \lambda_p \geq 0$ (with at least one strict inequality) are the regularization parameters. The sorted $\ell_1$ penalty, $\sum_{i=1}^p \lambda_i|b|_{(i)}$, is a norm, and the optimization problem for SLOPE is, therefore, convex (see also \citet{figueiredo2016ordered}). As an important feature, the sorted $\ell_1$ norm penalizes larger entries more heavily than smaller ones. Indeed, this regularization technique is shown to be adaptive to the degree of sparsity level and enables SLOPE to obtain optimal estimation performance for certain problems~\citep{SLOPE2}. Notably, in the special case $\lambda_1=\cdots = \lambda_p$, the sorted $\ell_1$ norm reduces to the usual $\ell_1$ norm. Thus, the Lasso can be regarded as a special instance of SLOPE.



A fundamental question, yet to be better addressed, is how to quantitatively characterize the benefits of using the sorted $\ell_1$ regularization. To explore this question, Figure~\ref{fig:simulation_overcome_ques} compares the model selection performance of SLOPE and the Lasso in terms of the \textit{false discovery proportion} (FDP) and \textit{true positive proportion} (TPP) or, equivalently, between measures of type I error and power. Needless to say, a model is preferred if its FDP is small while its TPP is large. As the first impression conveyed by this figure, both methods seem to undergo a trade-off between the FDP and TPP when the TPP is below a certain limit. More interestingly, while \textit{nowhere} on the
Lasso path is the TPP above a limit, which is about 0.5707 in the left plot of Figure~\ref{fig:simulation_overcome_ques} and 0.4343 in the right, SLOPE is able to pass the limit toward achieving full power. To be sure, these contrasting patterns persist even for an arbitrarily large signal-to-noise ratio. This distinction must be attributed to the flexibility of the SLOPE regularization sequence $(\lambda_1, \ldots, \lambda_p)$ compared to a single value as in the Lasso case. Recognizing this message, we are tempted to ask (1) \textit{why} the use of sorted $\ell_1$ regularization brings a significant benefit over $\ell_1$ regularization in the high TPP regime and, equally importantly, (2) \textit{why} SLOPE exhibits a trade-off between the FDP and TPP just as the Lasso does in the low TPP regime.


\begin{figure}[!htp]
	\centering
	\includegraphics[width=6.5cm,height=5cm]{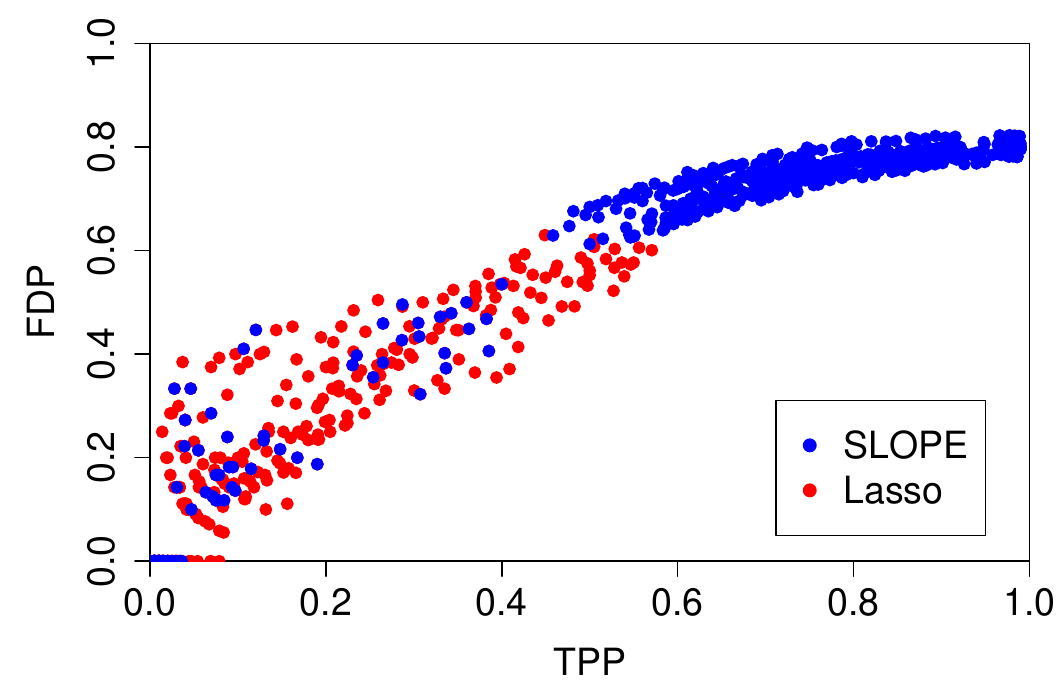}
	\includegraphics[width=6.5cm,height=5cm]{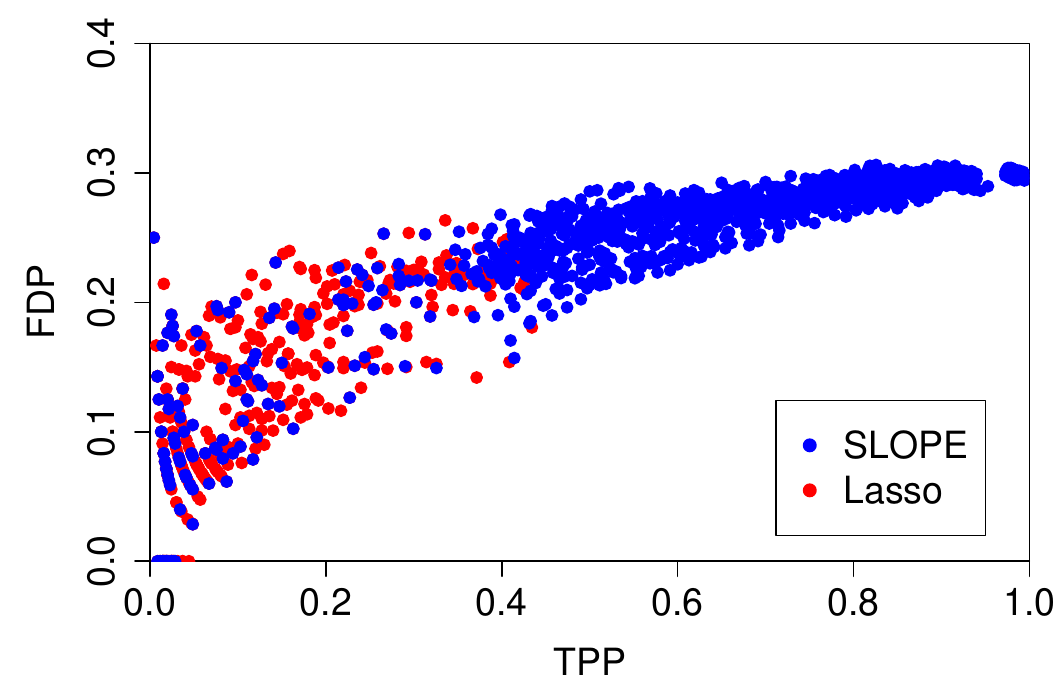}
	\caption{Comparison between SLOPE and the Lasso in terms of the TPP--FDP trade-off. Given an estimate $\widehat{\bet}$, define its $\FDP = \frac{|\{j:\beta_j = 0 ~\text{and}~ \widehat{\beta}_j \ne 0\}|}{|\{j: \widehat{\beta}_j \ne 0\}|}$ and $\TPP = \frac{|\{j:\beta_j \ne 0 ~\text{and}~ \widehat{\beta}_j \ne 0\}|}{|\{j: {\beta}_j\ne 0\}|}$. The SLOPE regularization sequence $\blam_{\lambda,r\lambda,w}$ is defined in \eqref{eq:theta_func}, with varying $0 < r < 1$ and $\lambda > 0$, and $w = 0.2$ in the left plot and $w = 0.3$ in the right plot. The results of the Lasso are taken over its entire solution path, and its highest TPP is about $0.5707$ in the left plot and $0.4343$ in the right plot. Left: $(n,p)=(300,1000), |\{j:\beta_j \neq 0\}|/p=0.2$, and $\w=\bm{0} $ (noiseless); right: $(n,p)=(400,1000), |\{j:\beta_j \neq 0\}|/p=0.7$, and $\w=\bm{0} $. On both plots, non-zero entries of $\bet$ are i.i.d.~draws from the standard normal distribution. More specifications of the setup are detailed in Section~\ref{sec:extend-donoho-tann}. The result presents 10 independent trials.}
\label{fig:simulation_overcome_ques}
\end{figure}


\subsection{A peek at our results}

To address these two questions, in this paper we characterize the optimal trade-off of SLOPE between the TPP and FDP, uncovering several intriguing findings of sorted $\ell_1$ regularization. Assuming $\TPP \approx u$ for $0 \le u \le 1$, loosely speaking, the trade-off curve gives the smallest possible value of the FDP of SLOPE using any regularization sequence in the large system limit. To prepare for a rough description of our contributions, in brief, we work in the setting where the design has i.i.d.~Gaussian entries and the regression coefficients $\beta_1, \ldots, \beta_p$ are i.i.d.~draws from a distribution that takes non-zero values with a certain probability. Notably, it is generally nontrivial to define false discoveries in high dimensions~\citep{g2013false}, which is not an issue however in the case of independent regressors. The assumption on the signal prior corresponds to the \textit{linear sparsity} regime. In addition, we assume that both $n, p
\goto \infty$ and the sampling ratio $n/p$ converges to a constant (see more detailed assumptions in Section~\ref{sec:extend-donoho-tann}). From a technical viewpoint, these assumptions allow us to make use of tools from approximate message passing (AMP) theory \citep{amp1,amp2}.

\paragraph*{Breaking the Donoho--Tanner power limit} To explain the contrasting results presented in Figure~\ref{fig:simulation_overcome_ques}, we prove that under the aforementioned assumptions, SLOPE can achieve an arbitrarily high TPP. Moving from sorted $\ell_1$ regularization to $\ell_1$ regularization, in stark contrast, the Lasso exhibits the Donoho--Tanner (DT) power limit when $n < p$ and the sparsity is above a certain threshold~\citep{donoho2006high,donoho2005neighborly}. Informally, the DT power limit is the largest possible power that any estimate along the Lasso path can achieve in the large system limit. For example, in the setting of Figure~\ref{fig:simulation_overcome_ques} this power limit is about 0.5676 in the left plot and 0.4401 in the right plot. For SLOPE and a certain choice of the regularization sequence, interestingly, we show that the asymptotic TPP-FDP trade-off of SLOPE beyond the DT power limit is given by a simple \Mobius transformation, which is shown by the blue curve in Figure~\ref{fig:intro_result}. This \Mobius transformation naturally serves as an upper bound on the
(optimal) SLOPE trade-off curve above the DT power limit.


\paragraph*{Lower bound via convex optimization} 
Next, we address the second question by lower bounding the optimal trade-off for SLOPE, followed by a comparison between the trade-offs for the two methods in the low TPP regime. To put it into perspective, the Lasso trade-off obtained by~\citet{lassopath} is plotted as the green solid curve in Figure~\ref{fig:intro_result}. Apart from the simple fact that the SLOPE trade-off is better than or equal to the Lasso counterpart, however, it requires new tools to take into account the structure of sorted $\ell_1$ regularization. To this end, we develop a technique based on a class of infinite-dimensional convex optimization problems. The resulting lower bound is shown in red in Figure~\ref{fig:intro_result}. It is worth noting that the development of this technique presents several novel ideas that might be of independent interest for other regularization schemes.

\paragraph*{Instance superiority of SLOPE} The results illustrated so far are taken from an optimal-case viewpoint. Moving to a more practical standpoint, we are interested in comparing the two methods on a specific problem instance and, in particular, wish to find a SLOPE regularization sequence that allows SLOPE to outperform the Lasso with any given penalty parameter in terms of, for example, the TPP, the FDP, or the $\ell_2$ estimation risk. Surprisingly, we prove that on any problem instance, SLOPE can dominate the Lasso according to these three indicators simultaneously. This comparison conveys the message that the flexibility of the sorted $\ell_1$ regularization can turn into appreciable benefits. This result is formally stated in \Cref{thm:instance better}.


\begin{figure}[!htp]
	\centering
	\includegraphics[width=6.5cm,height=5cm]{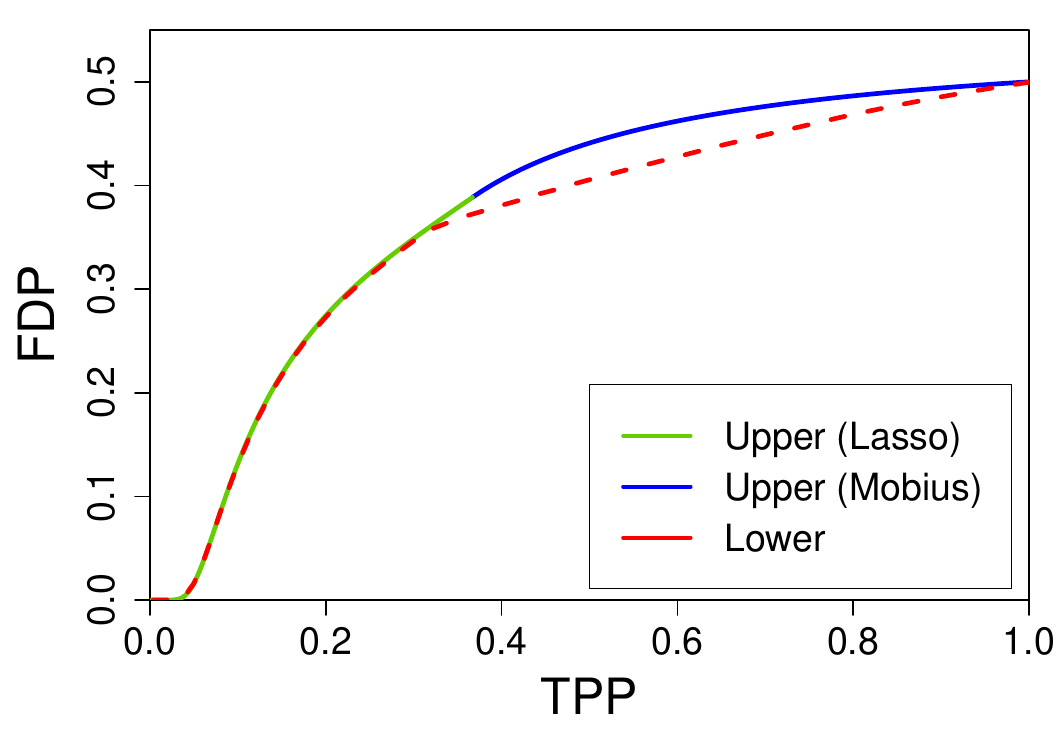}
	\includegraphics[width=6.5cm,height=5cm]{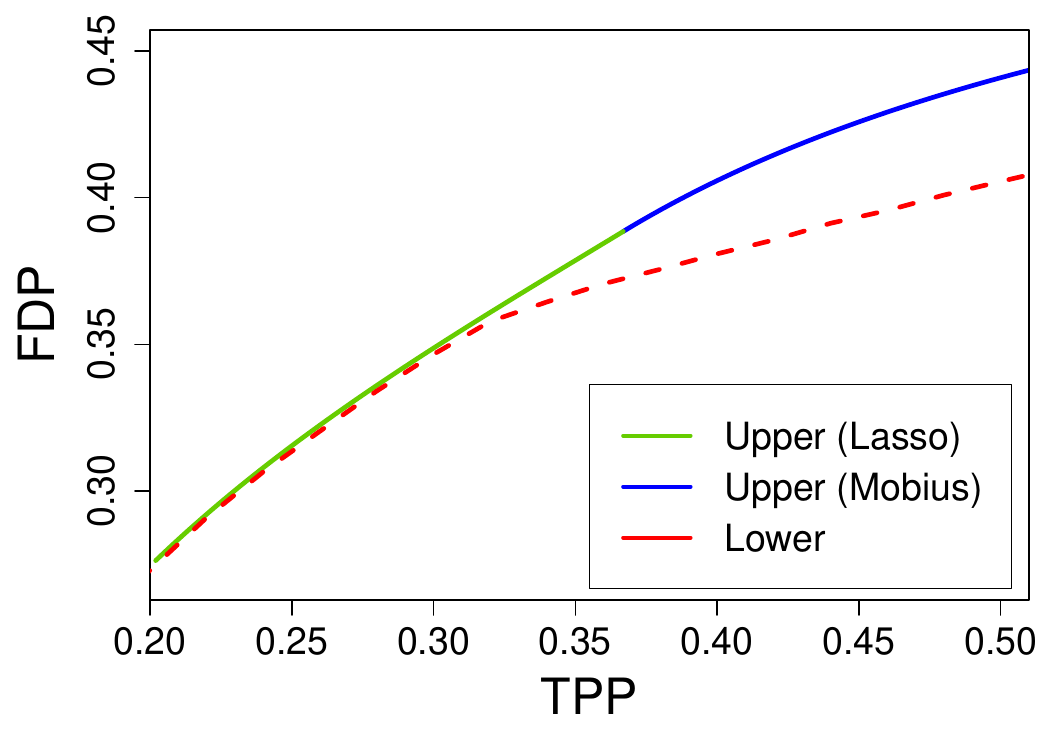}
	\caption{Illustration of the upper bound $q^\star$ and lower bound $q_\star$ for the SLOPE TPP--FDP trade-off. The right plot is the zoom-in of the left. Here $n/p = 0.3$ and $|\{j:\beta_j \neq 0\}|/p=0.5$ (see more details in the working assumptions in \Cref{sec:extend-donoho-tann}). The Lasso trade-off curve shown in green is truncated at the DT power limit about 0.3669~\citep{lassopath}. The optimal SLOPE trade-off curve must lie between the two curves. Notably, the two bounds agree at $\TPP = 1$.}
	\label{fig:intro_result}
\end{figure}

\subsection{Organization}

The remainder of this paper is structured as follows. In \Cref{sec:extend-donoho-tann}, we present the main results of this paper. Next, \Cref{sec:prelim} introduces the AMP machinery at a minimal level as a preparation for the proofs of our main results. In \Cref{sec:lower bound}, we detail the derivation of the lower bound based on variational calculus and infinite-dimensional convex optimization. In \Cref{sec:mobius}, we specify the upper bound, especially the part given by a \Mobius transformation above the DT power limit. We conclude this paper in \Cref{sec:discussion} by proposing several future research directions. Omitted proofs are relegated to the appendix.



\section{Main results} \label{sec:extend-donoho-tann}

Throughout this paper, we make the following working assumptions to specify the design matrix $\X \in \R^{n \times p}$, regression coefficients $\bet \in \R^p$, and noise $\w \in \R^n$ in the linear model \eqref{eq:linear_model}, as well as the SLOPE regularization sequence $\blam = (\lambda_1, \ldots, \lambda_p)$. To obviate any ambiguity, we consider a sequence of problems indexed by $(n, p)$ with both $n, p$ tending to infinity.

\begin{itemize}
	\item[\indent(\textnormal{A1})] The matrix $\X$ has i.i.d.~$\N(0, 1/n)$ entries. The sampling ratio $n/p$ converges to a constant $\delta > 0$. 
	\item[\indent(\textnormal{A2})]  The entries of $\bet$ are i.i.d.~copies of a random variable $\Pi$ satisfying $\P(\Pi \ne 0) = \epsilon$ for a constant $0 < \epsilon < 1$ and $\E(\Pi^2\max\{0,\log \Pi\}) < \infty$. The noise vector $\bm w$ consists of i.i.d.~copies of a random variable $W$ with bounded second moment $\sigma^2 := \E(W^2)  < \infty$.	%
	\item[\indent(\textnormal{A3})]
	
	The SLOPE regularization sequence $\blam(\p) = (\lambda_1, \ldots, \lambda_p)$ is the order statistics of $p$  i.i.d.~realizations of a (nontrivial) non-negative random variable $\Lambda$.
	
\end{itemize}

Moreover, we assume that $\X, \bet$, and $\bm w$ are independent. Notice that the sparsity level of $\bet$ is about $\epsilon p$ and that each column of $\X$ has approximately a unit $\ell_2$ norm. The noise variance $\sigma^2$ can equal $0$, meaning that our results apply to both noisy and noiseless settings. In (A3), by ``nontrivial'' we mean that $\Lambda$ is not always equal to 0. As an aside, SLOPE is reduced to the Lasso if the distribution of $\Lambda$ is a unit probability mass at some positive value.

The working assumptions are mainly driven by their necessity in AMP theory~\citep{amp1,amp2}, which enables the use of the recent analysis of an AMP algorithm when applied to solve SLOPE~\citep{ourAMP} (similar analysis is given in~\citet{SLOPEasymptotic} and requires similar assumptions). 
Regarding (A2), the condition $\PP(\Pi\neq 0)=\epsilon$, which implies linear sparsity of the regression coefficients, is not required for AMP theory. Rather, this condition is only made so that the TPP and FDP are well-defined. Besides, the merit of the linear sparsity regime has been increasingly recognized in the high-dimensional literature~\citep{mousavi2018consistent,weng2018overcoming,su2018first, sur2019likelihood,wang2019does}.

\subsection{Bounds on the SLOPE trade-off}
\label{sec:main-results}

Our main result is the characterization of a trade-off curve that teases apart asymptotically achievable TPP and FDP pairs from the asymptotically unachievable pairs for SLOPE\footnote{R code to reproduce the results, e.g., to calculate $q_\star$ and $q^\star$, is available at \url{https://github.com/woodyx218/SLOPE_AMP}.}. For any estimate $\widehat\bet$, recall that its FDP and TPP are defined as
\begin{align}
	\label{eq:original_defs}
	\FDP = \frac{|\{j:\beta_j = 0 ~\text{and}~ \widehat{\beta}_j \ne 0\}|}{|\{j: \widehat{\beta}_j \ne 0\}|}, \quad \TPP = \frac{|\{j:\beta_j \ne 0 ~\text{and}~ \widehat{\beta}_j \ne 0\}|}{|\{j: {\beta}_j\ne 0\}|},
\end{align}
with the convention $0/0 = 0$. When it comes to the SLOPE estimator, we use $\TPP(\bet,\blam)$ and $\FDP(\bet,\blam)$ to denote its TPP and FDP, respectively. 

Likewise, we define the thresholded FDP and TPP, namely, 
	\begin{align}
		\label{eq:new_defs}
		\FDP_\GAM = \frac{|\{j:\beta_j = 0 ~\text{and}~ |\widehat{\beta}_j|>\GAM\}|}{|\{j: |\widehat{\beta}_j|>\GAM\}|}, \quad \TPP_\GAM = \frac{|\{j:\beta_j \ne 0 ~\text{and}~ |\widehat{\beta}_j|>\GAM\}|}{|\{j: {\beta}_j\ne 0\}|},
	\end{align}
	which reduce to FDP and TPP when $\GAM=0$.
	These thresholded versions of FDP and TPP are introduced purely for technical reasons, and have been used in previous work on penalized estimators like SLOPE including in \citet{wang2017bridge}. Specifically, the SLOPE estimator is known to possibly have many elements that are very close to zero, but not strictly equal to zero, causing the direct asymptotic analysis of the FDP and TPP defined in \eqref{eq:original_defs} to be difficult. We refer interested readers to \citet[Example 3 and Figure 3]{SLOPEasymptotic} for a concrete example that illustrates such a phenomenon.  Instead, we analyze asymptotic (in $p$) properties of $\FDP_{\GAM}$ and $\TPP_{\GAM}$ in \eqref{eq:new_defs} and then allow $\GAM \rightarrow 0$ to recover asymptotic properties of FDP and TPP defined in \eqref{eq:original_defs}.

Our main results are stated in the following two theorems, which give lower and upper bounds on the optimal SLOPE trade-off. Taken together, they demonstrate a fundamental separation between asymptotically achievable TPP--FDP pairs and the unachievable pairs over all signal priors $\Pi$ and SLOPE regularization sequences $\blam$. Note that both the upper bound $q^\star$ and lower bound $q_\star$ are defined on $[0, 1]$ and completely determined by $\epsilon$ and $\delta$. The expression for $q^\star$ is given in \eqref{cor:whole tradeoff}, while $q_\star$ is detailed in \Cref{sec:lower bound}.


\begin{theorem}[Lower bound]\label{thm:lower trade-off-curve-no-q-form}
	Under the working assumptions, namely (A1), (A2), and (A3), 
	the following inequality holds with probability tending to one:
		\begin{align*}
			\FDP_{\GAM}(\bet,\blam) \geq q_\star \left( \TPP_{\GAM}(\bet,\blam); \delta,\epsilon \right) - c_\GAM,
		\end{align*}
		for some positive constant $c_\GAM$ which tends to 0 as $\GAM\to 0$.
\end{theorem}

\begin{theorem}[Upper bound]\label{thm:upper trade-off-curve-no-q-form}
	Under the working assumptions, namely (A1), (A2), and (A3), for any $0 \le u \le 1$, 	there exist a signal prior $\Pi$ and a SLOPE regularization prior $\Lambda$ such that the following inequalities hold with probability tending to one:
		\begin{align*}
			\FDP_{\GAM}(\bet,\blam) \leq q^\star \left( \TPP_{\GAM}(\bet,\blam); \delta,\epsilon \right) + c_\GAM(\Pi,\Lambda)  \quad  \text{ and } \quad |\TPP_{\GAM}(\bet,\blam) - u| \le c_\GAM(\Pi,\Lambda),
		\end{align*}
		for some positive constant $c_\GAM$ which tends to 0 as $\GAM\to 0$.
\end{theorem}

\begin{rem}
	The probability is taken with respect to the randomness in the design matrix, regression coefficients, noise, and SLOPE regularization sequence in the large system limit $\n, \p \rightarrow \infty$. In relating to the assumptions made previously, this theorem holds even for $\sigma^2 = 0 $, the noiseless case.
\end{rem}

The proofs of \Cref{thm:lower trade-off-curve-no-q-form} and \Cref{thm:upper trade-off-curve-no-q-form} are given in \Cref{sec:lower bound} and \Cref{sec:mobius}, respectively. Most notably, our proof of \Cref{thm:lower trade-off-curve-no-q-form} starts by formulating the problem of finding a tight lower bound as a calculus of variations problem. Relying on several novel elements, we further reduce this problem to a class of infinite-dimensional convex programs.


On the one hand, \Cref{thm:lower trade-off-curve-no-q-form} says that it is impossible to achieve high power and a low FDP simultaneously using any sorted $\ell_1$ regularization sequences, and this trade-off is specified by $q_\star$. On the other hand, Theorem~\ref{thm:upper trade-off-curve-no-q-form} demonstrates that SLOPE can achieve at least the same trade-off as that given by $q^\star$ by specifying a prior $\Pi$ and a regularization sequence $\blam$. Indeed, the proof of this theorem is constructive in that we will show that SLOPE can come arbitrarily close to any point on the curve $q^\star$ (see \Cref{sec:mobius}). Another important observation from \Cref{thm:upper trade-off-curve-no-q-form} is that SLOPE can achieve any power levels, which is not necessarily the case for $\ell_1$ regularization-based methods, as we show in \Cref{sec:comp-with-lasso-1}.

Informally, let $q_{_{\textnormal{\tiny SLOPE}}}$ denote the optimal SLOPE trade-off curve. That is, $q_{_{\textnormal{\tiny SLOPE}}}(u)$ is asymptotically the minimum possible value of the FDP under the constraint that the TPP is about $u$, over all possible SLOPE regularization sequences (see formal definition in \Cref{sec:prelim}). Combining the two theorems above, we readily see that the optimal SLOPE trade-off must be sandwiched between $q^\star$ and $q_\star$:
\[
q_\star(u) \leq q_{_{\textnormal{\tiny SLOPE}}}(u) \leq q^\star(u),
\]
for all $0 \le u \le 1$. Consequently, the sharpness of the approximation to the SLOPE trade-off rests on the gap between the two curves, and throughout the paper, we refer to the gap as the function $u\mapsto q^\star(u)-q_\star(u)$. Figure~\ref{fig:SLOPEtradeoffExamples} illustrates several examples of the two curves for various pairs of $\epsilon, \delta$. Importantly, the plots show that the two bounds are very close to each other, thereby demonstrating tightness of our bounds. In fact, the gap between $q_\star$ and $q^\star$ is an upper bound of the gap between the analytical $q^\star$ and the true trade-off $q_\textnormal{\tiny SLOPE}$. Furthermore, a closer look at the plots reveals that the two curves seem to coincide exactly when the TPP is below a certain value. In this regard, the SLOPE trade-off might have been uncovered exactly in this regime of TPP. Future investigation is required to obtain a fine-grained comparison between the two curves.

Looking at \Cref{fig:SLOPEtradeoffExamples}, the reader may initially find the non-monotonicity of the trade-off curves  in $\epsilon$ as surprising. We argue that this is due to the DT phase transition: in the case of the Lasso, for fixed $\delta$, it can be shown that the trade-off curves are monotonically increasing in $\epsilon$; in other words, $q_\text{\tiny Lasso}(u;\delta,\epsilon_1)>q_\text{\tiny Lasso}(u;\delta,\epsilon_2)$ whenever $\epsilon_1>\epsilon_2$. However, in some settings, we empirically observe that $\TPP=1$ is achieved with a dense SLOPE estimator. When this occurs, $q_\text{\tiny  SLOPE}(1)=1-\epsilon$ and thus $q_\text{\tiny  SLOPE}(1;\delta,\epsilon_1)<q_\text{\tiny  SLOPE}(1;\delta,\epsilon_2)$. In words, the SLOPE trade-off at $\TPP=1$ is monotonically \emph{decreasing} in $\epsilon$. Therefore, the patterns may not be monotone between the TPP upper limit $u_{\text{DT}}^\star$ and 1, shifting from increasing in $\epsilon$ to decreasing in $\epsilon$ at the extreme. In short, the regime beyond DT phase limit is different for SLOPE and when SLOPE enters this regime, breaking the monotonicity in $\epsilon$ may occur.

To be complete, we remark that the message conveyed by these two theorems does not contradict earlier results established for FDR control of SLOPE~\citep{bogdan2013statistical,SLOPE1,groupslope,kos2020asymptotic}. The crucial difference between the two sides arises from the linear sparsity assumed in the present paper, which is a clear departure from the much lower sparsity level considered in the literature. In this regard, our results complement the literature by extending our understanding of the inferential properties of the SLOPE method to an unchartered regime. 


\begin{figure}
	\centering
	\includegraphics[width=6cm,height=4.3cm]{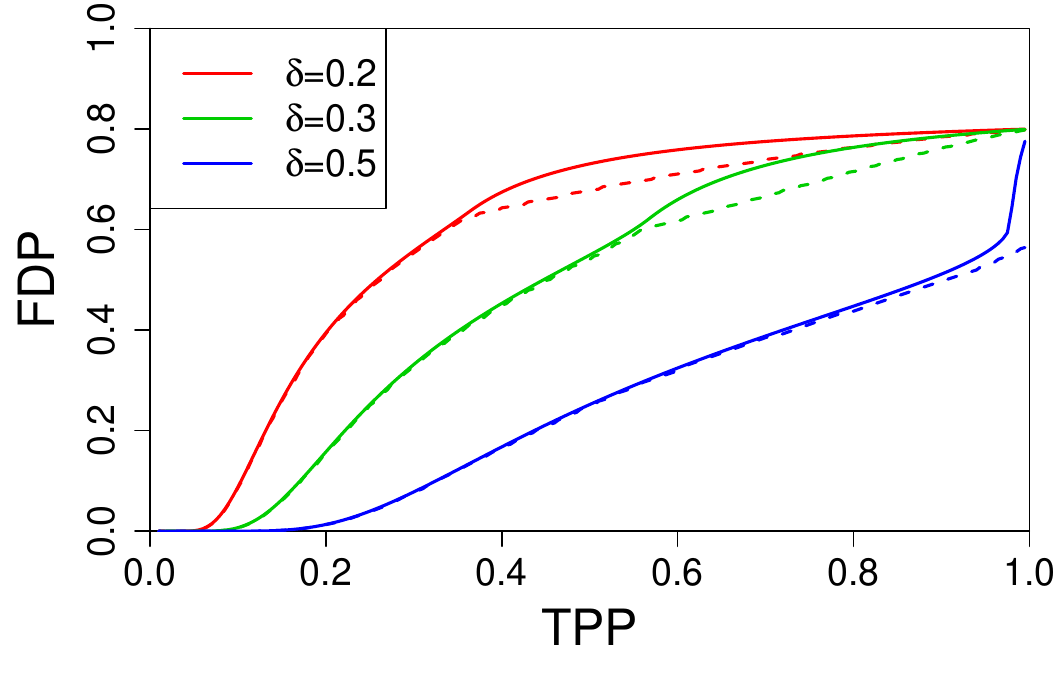}
	\includegraphics[width=6cm,height=4.3cm]{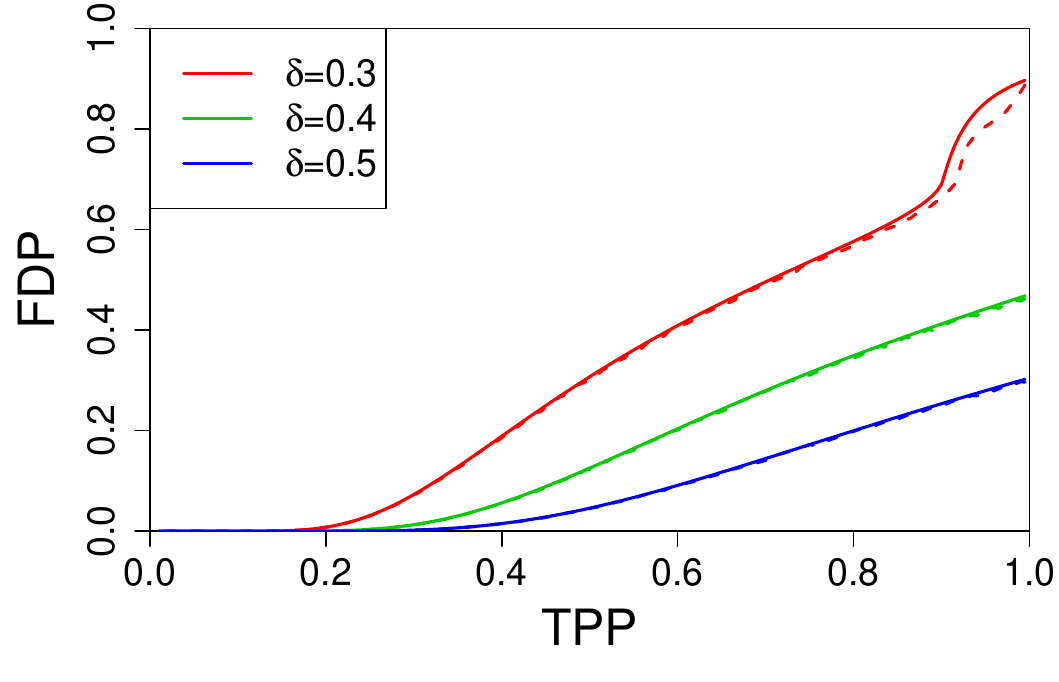}
	\includegraphics[width=6cm,height=4.3cm]{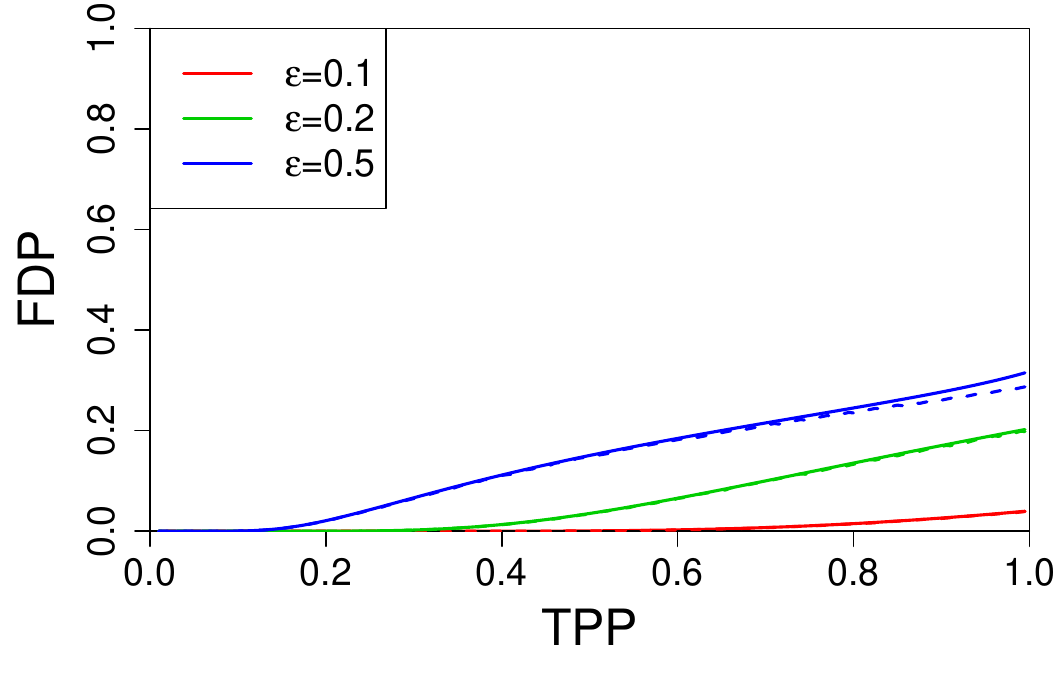}
	\includegraphics[width=6cm,height=4.3cm]{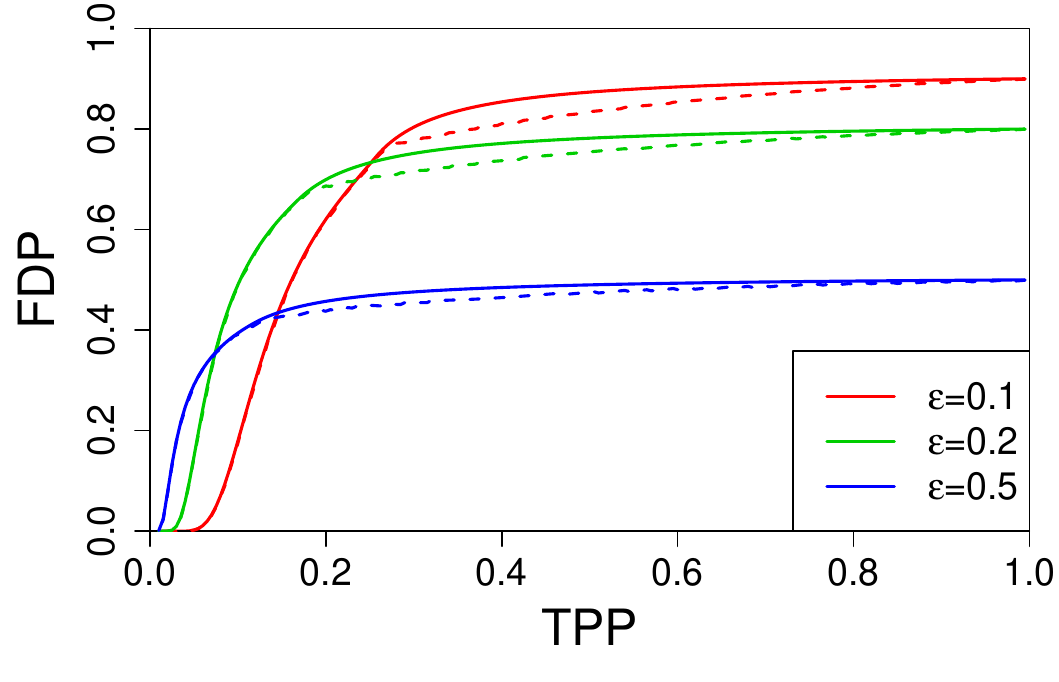}
	\caption{Examples of the SLOPE trade-off bounds $q^{\star}$ and $q_{\star}$ for different $(\delta,\epsilon)$ pairs. Top-left: $\epsilon=0.2$; top-right: $\epsilon=0.1$; bottom-left: $\delta=0.9$; bottom-right: $\delta=0.1$. For a given $\delta$, note that the trade-off for SLOPE is not monotone with respect to $\epsilon$, which is a departure from the Lasso counterpart (see \citet[Figure 4]{lassopath}). Numerically, the upper and lower bounds seem to coincide when the TPP is below a threshold (see more details in \Cref{fig:lasso_phase}). To give more details, in one regime with $\delta=0.1, \epsilon=0.5$, the maximum gap between the upper and lower bounds $\max_{u}[q^\star(u)-q_\star(u)]$ is less than 0.0235; whereas in another regime with $\delta=0.5,\epsilon=0.1$, the maximum gap is always less than 0.0056.}
	\label{fig:SLOPEtradeoffExamples}
\end{figure}

\subsection{Breaking the Donoho--Tanner power limit}
\label{sec:comp-with-lasso-1}

To better appreciate the trade-off results presented in \Cref{thm:upper trade-off-curve-no-q-form} for SLOPE, it is instructive to compare them with the TPP and FDP trade-off for the Lasso, which is arguably the most popular method leveraging $\ell_1$ regularization.

To put it into perspective, first recall some results concerning the optimal trade-off between the TPP and FDP for the Lasso. A surprising fact is that under the working assumptions,\footnote{Note that, in the case of the Lasso, (A3) is replaced by the assumption that $\lambda > 0$ is a constant.} the Lasso cannot achieve full power even with an arbitrarily large signal-to-noise ratio when $\delta < 1$ (that is, $\X$ is ``fat'') and the sparsity ratio $\epsilon$ is above a threshold, which we denote by $\epsilon^\star(\delta)$. The dependence of this value on $\delta$ is specified by the parametric equations
\begin{equation}\label{eq:eps_star}
	\delta = \frac{2\phi(s)}{2\phi(s) + s(2\Phi(s) - 1)}, \qquad \epsilon^\star = \frac{2\phi(s) - 2s \Phi(-s) }{2\phi(s) + s(2\Phi(s) - 1)},
\end{equation}
for $s > 0$.\footnote{In the compressed sensing literature, $\epsilon^\star$ corresponds to the sparsity level where the Donoho--Tanner phase transition occurs~\citep{donoho2009observed,donoho2009counting}.}  For simplicity, henceforth $(\delta,\epsilon)$ is said to be in the \textit{supercritical} regime if $\delta < 1, \epsilon > \epsilon^\star(\delta)$. Otherwise, it is in the \textit{subcritical} regime when $\delta <1, \epsilon \le \epsilon^\star(\delta)$, or $\delta \ge 1$ (that is, $\X$ is ``thin''). In the supercritical regime, \citet{lassopath} proved that the highest achievable TPP of the Lasso, denoted $\tppmax$, takes the form
\be
\label{eq:DT_power_limit}
\tppmax(\delta, \epsilon) := 1 - \frac{(1-\delta)(\epsilon-\epsilon^\star)}{\epsilon(1-\epsilon^\star)}  < 1.
\ee

Throughout the paper, $\tppmax$ is referred to as the \textit{DT power limit}. For completeness, in the subcritical regime the Lasso can achieve any power level. As such, we formally set $\tppmax(\delta, \epsilon) = 1$ when $\delta <1, \epsilon \le \epsilon^\star(\delta)$, or $\delta \ge 1$.

This existing result, in conjunction with \Cref{thm:upper trade-off-curve-no-q-form}, immediately gives the following contrasting result concerning the Lasso and SLOPE. We use $\TPP_{\textnormal{\tiny Lasso}}(\bet,\lambda)$ and $\FDP_{\textnormal{\tiny Lasso}}(\bet,\lambda)$ to denote, respectively, the TPP and FDP of the Lasso with penalty parameter $\lambda$. Likewise, we use $\TPP_{\textnormal{\tiny SLOPE}}(\bet,\bm\lambda)$ and $\FDP_{\textnormal{\tiny SLOPE}}(\bet, \bm\lambda)$ to denote those of SLOPE as $\GAM\to 0$. 

\begin{corollary}[SLOPE breaks the DT power limit]\label{cor:achievability}
	In the supercritical regime, the following conclusions hold under the working assumptions:
	\begin{enumerate}
		\item[(a)]
		The power of the Lasso satisfies $\TPP_{\textnormal{\tiny Lasso}}(\bet,\lambda) < \tppmax$ with probability tending to one.
		
		
		\item[(b)] 
		For any $0 \le u < 1$, there exists a SLOPE regularization prior $\Lambda$ and a signal prior $\Pi$ such that $\TPP_{\textnormal{\tiny SLOPE}}(\bet,\bm\lambda) > u$ with probability tending to one. 
	\end{enumerate}
\end{corollary}


For illustration, \Cref{fig:simulation_overcome_ques} in the introduction reflects this distinction between SLOPE and the Lasso with $\tppmax(0.4, 0.7) = 0.4401$ in the right plot.
Another illustration is the left plot of \Cref{fig:simulation_overcome_ques} and \Cref{sharp_pic}, which is vertically truncated at $\tppmax(0.3, 0.2) = 0.5676$. Notice that SLOPE breaks the DT power limit, i.e.\ there are $(\bet,\bm\lambda)$ pairs for which $\tppmax<\TPP_{\textnormal{\tiny SLOPE}}<1$, while still preserving non-trivial FDP, i.e.\ $\FDP_{\textnormal{\tiny SLOPE}}<1-\epsilon$, where $1-\epsilon$ would be the FDP associated with the trivial procedure that selects all predictors.

\Cref{cor:achievability} highlights the benefit of using sorted $\ell_1$ regularization over the less flexible $\ell_1$ regularization in terms of power. This sharp distinction persists no matter how large the effect sizes are and, therefore, it must be attributed to the flexibility of the SLOPE regularization sequence. As is well-known, the Lasso selects no more than $n$ variables. Worse, a significant proportion of false variables are always interspersed on the Lasso path in the linear sparsity regime and, therefore, even though the Lasso can select up to $n > k$ variables, it would always miss a fraction of true variables, thereby imposing a limit on the power. In contrast, SLOPE does not bear the constraint that $\|\widehat{\bm\beta}\|_0 \le n$ owing to the flexibility of its regularization sequence. In fact, the corresponding constraint for SLOPE is that the number of \textit{unique} non-zero entries is no more than $n$~\citep{SLOPE2}. This flexibility allows SLOPE to have
arbitrarily high power regardless of the regime that $(\delta, \epsilon)$ belongs to.

\begin{figure}[!htp]
	\centering
	\includegraphics[width=8cm,height=5.5cm]{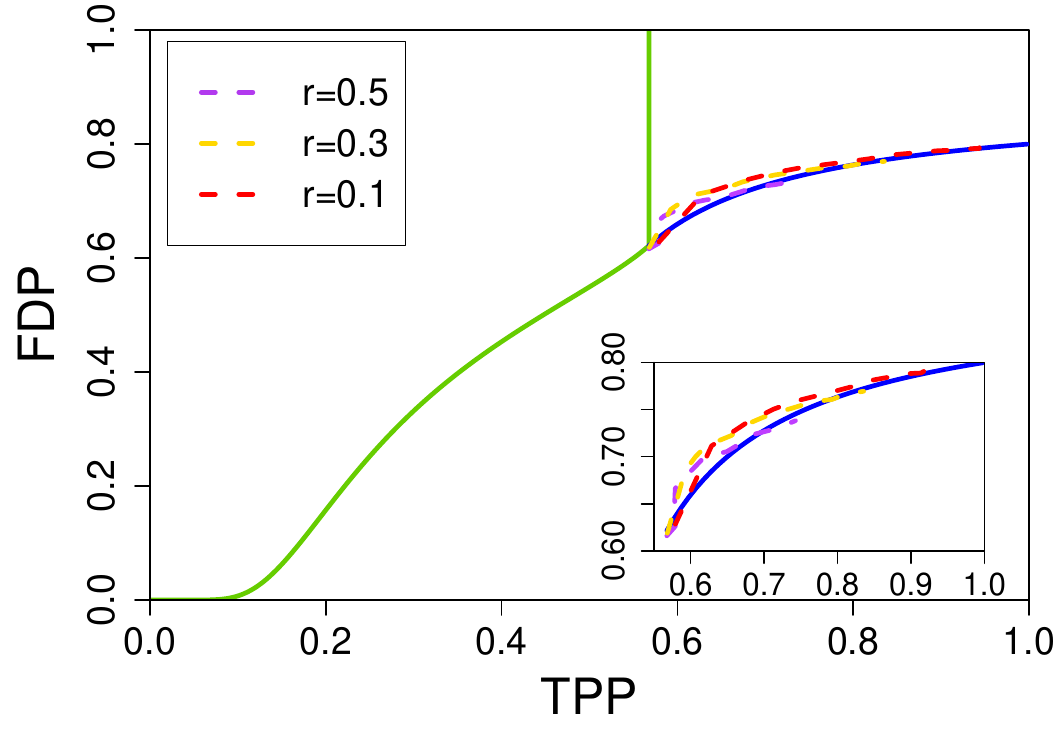}
	\caption{The \Mobius part of the SLOPE trade-off upper bound $q^\star$. The solid curve denotes the upper bound specified by $(\delta,\epsilon)=(0.3,0.2)$. The green line is the Lasso part of $q^\star$ and the blue line is the \Mobius part. The numerical pairs of the TPP and FDP are obtained from experiments that are specified by the following parameters: $n = 300, p=1000, \sigma^2= 0$, signal prior $\Pi_M(\epsilon^\star/\epsilon)$ with $M=10000$ in \eqref{eq:sharp_prior} (note that $\epsilon^\star(0.3) = 0.087$), and regularization prior $\blam_{\sqrt{M}, r\sqrt{M}, w}$ in \eqref{eq:theta_func} with varying $w$. Each pair is averaged over 50 independent trials. }
	\label{sharp_pic}
\end{figure}


Moving forward, we ask which regularization prior $\Lambda$ and signal prior $\Pi$ are ``flexible'' enough to enable SLOPE to break the DT power limit. To achieve desired flexibility, interestingly, it only requires a simple two-level regularization sequence for SLOPE. Consider the following \textit{two-level} SLOPE regularization prior: given constants $a > b \ge 0$ and $0 < \omega < 1$, let $\Lambda_{a, b, \omega} = a$ with probability $\omega$ and otherwise $\Lambda_{a, b, \omega} = b$. The SLOPE regularization sequence drawn from this prior takes the form
\be
\blam_{a, b, \omega} : = \big( \underbrace{a, a, \cdots, a}_{\text{around}~\omega p}, \underbrace{b, b, \cdots, b}_{\text{around}~(1-\omega)p} \big).
\label{eq:theta_func}
\ee

Next, for any $M> 0$ and $0\leq\epsilon' \leq 1$, define the following signal prior:
\begin{align}
	\Pi_M(\epsilon') :=
	\begin{cases}
		M, & \text{w.p.} \quad \epsilon\epsilon',	\\
		M^{-1}, & \text{w.p.} \quad \epsilon - \epsilon\epsilon'	,\\
		0, & \text{w.p.} \quad 1 - \epsilon.
	\end{cases}
	\label{eq:sharp_prior} 
\end{align}

Henceforth in this paper, denote by $\bet_M(\epsilon')$ the regression coefficients sampled from $\Pi_M(\epsilon')$. In the following result, we take $M \goto \infty$, rendering the nonzero entries of $\bet_M(\epsilon')$ either very large or small.

Now we are ready to state the following result, which shows that SLOPE with the two-level regularization sequence can approach any point on the \Mobius transformation \eqref{eq:mob_part} arbitrarily closely. This result also specifies the upper bound $q^\star$ in \Cref{thm:upper trade-off-curve-no-q-form} in the supercritical regime:
\begin{equation}\label{eq:mob_part}
	q^{\star}(u;\delta,\epsilon) = \dfrac{\epsilon(1-\epsilon) u-\epsilon^{\star}(1-\epsilon)}{\epsilon(1-\epsilon^{\star}) u-\epsilon^{\star}(1-\epsilon)},
\end{equation}
for $\tppmax \le u \le 1$ (above the DT power limit). Note that this function takes the form of a \textit{\Mobius transformation}. Notably, taking $u = 1$ gives $q^{\star}(1; \delta,\epsilon) = \dfrac{(\epsilon -\epsilon^{\star})(1-\epsilon)}{\epsilon(1-\epsilon^{\star}) -\epsilon^{\star}(1-\epsilon)}   = 1 - \epsilon$, which is the FDP achieved by the trivial procedure that simply selects all predictors.

\begin{proposition}\label{prop:whole_sharp}
	For any $\tppmax \le u \le 1$ in the supercritical regime, there exists a $w$ such that $\blam_{a, b, \omega}$ and $\bet_M(\epsilon^\star/\epsilon)$ (defined via the prior in \eqref{eq:sharp_prior}) make SLOPE approach the point $(u,q^{\star}(u))$ in the sense that
		\[
		\lim_{M\to\infty} \,\lim_{\GAM\to 0}\, \lim_{n,p\to\infty}\left( \TPP_\GAM(\bet_M(\epsilon^\star/\epsilon),\blam_{a, b, w}), \FDP_\GAM(\bet_M(\epsilon^\star/\epsilon),\blam_{a, b, w}) \right) \rightarrow (u,q^{\star}(u)),
		\]
	where $a = \sqrt{M}, b = r\sqrt{M}$ for a certain value $0 \leq r \leq 1$.
\end{proposition}

\Cref{sharp_pic} provides a numerical example that corroborates this proposition.

This result in fact implies \Cref{thm:upper trade-off-curve-no-q-form} for $\tppmax \le u \le 1$ in the supercritical regime. Note that the first limit $\lim_{n,p\to\infty}$ is taken in the sense of convergence in probability. See more details in its proof in \Cref{sec:approach Mobius}. It is worthwhile to mention that the three-component mixture \eqref{eq:sharp_prior} is considered in \cite{lassopath} for the construction of favorable priors under sparsity constraint (see a generalization in \citet{wang2020price}).  
This mixture prior is used to ensure that the effect sizes are either very strong or very weak. In particular, \Cref{prop:whole_sharp} remains true if $M$ and $1/M$ are replaced by any value diverging to infinity and any value converging to 0, respectively.

\subsection{Below the Donoho–Tanner power limit}


Next, we continue to interpret \Cref{thm:lower trade-off-curve-no-q-form} and \Cref{thm:upper trade-off-curve-no-q-form}, but with a focus on the regime below the DT power limit.

\begin{figure}[!htp]
	\centering
	\includegraphics[width=7cm,height=5cm]{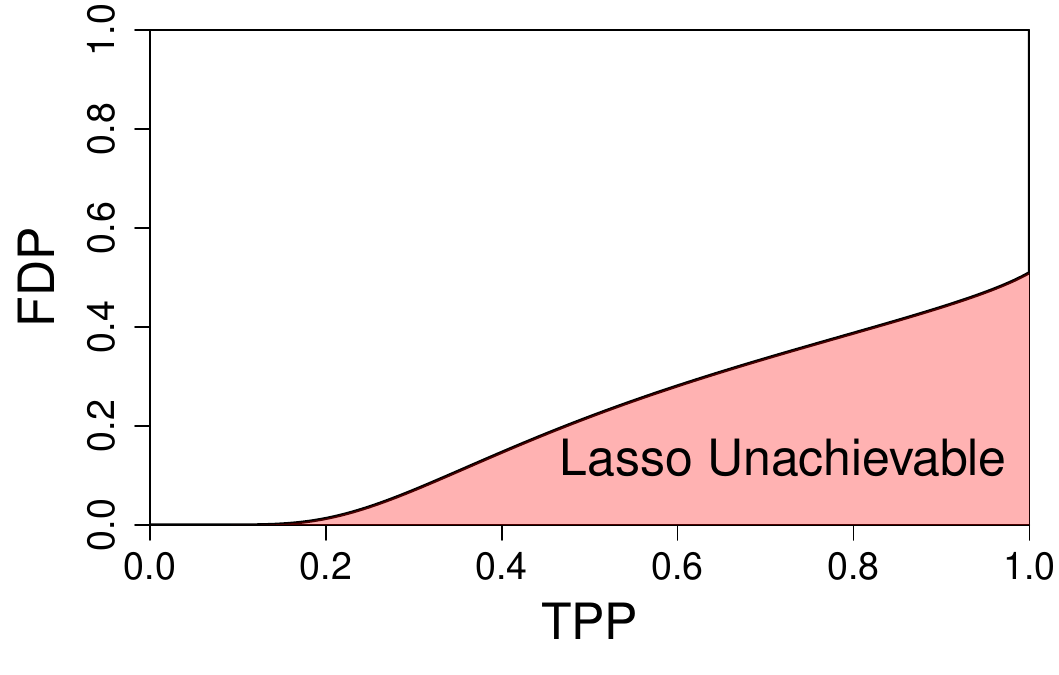}
	\includegraphics[width=7cm,height=5cm]{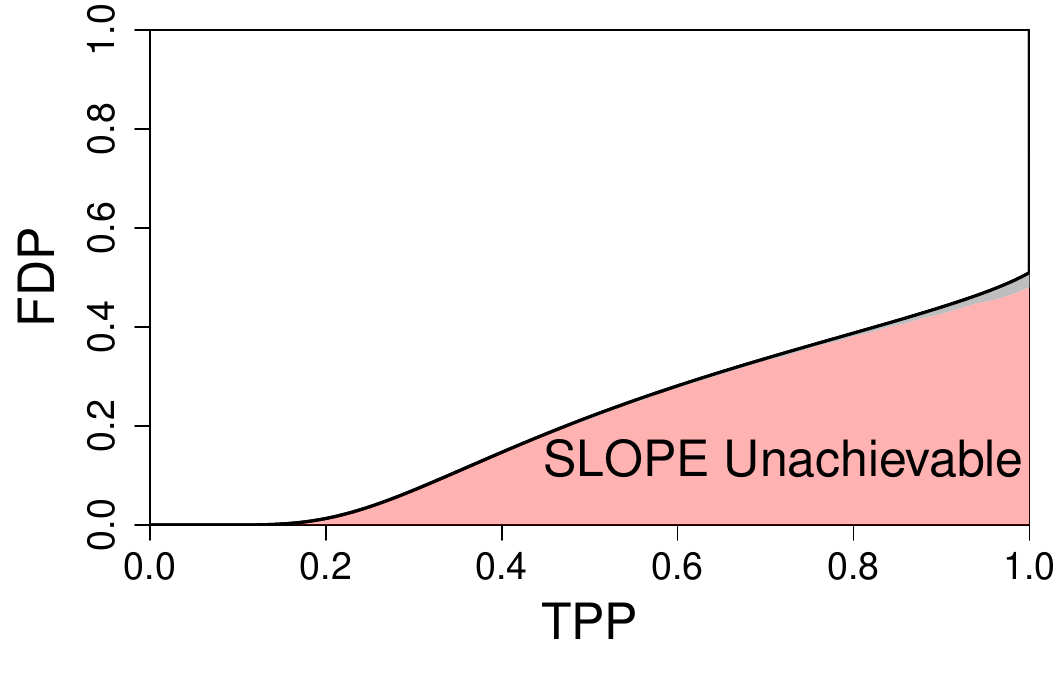}
	\includegraphics[width=7cm,height=5cm]{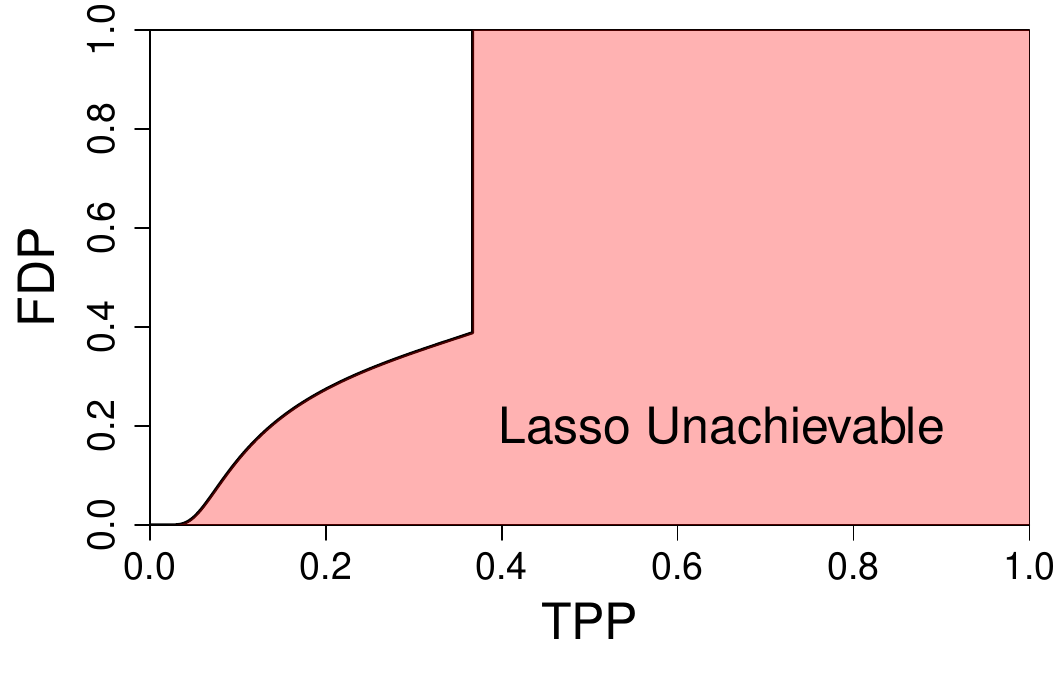}
	\includegraphics[width=7cm,height=5cm]{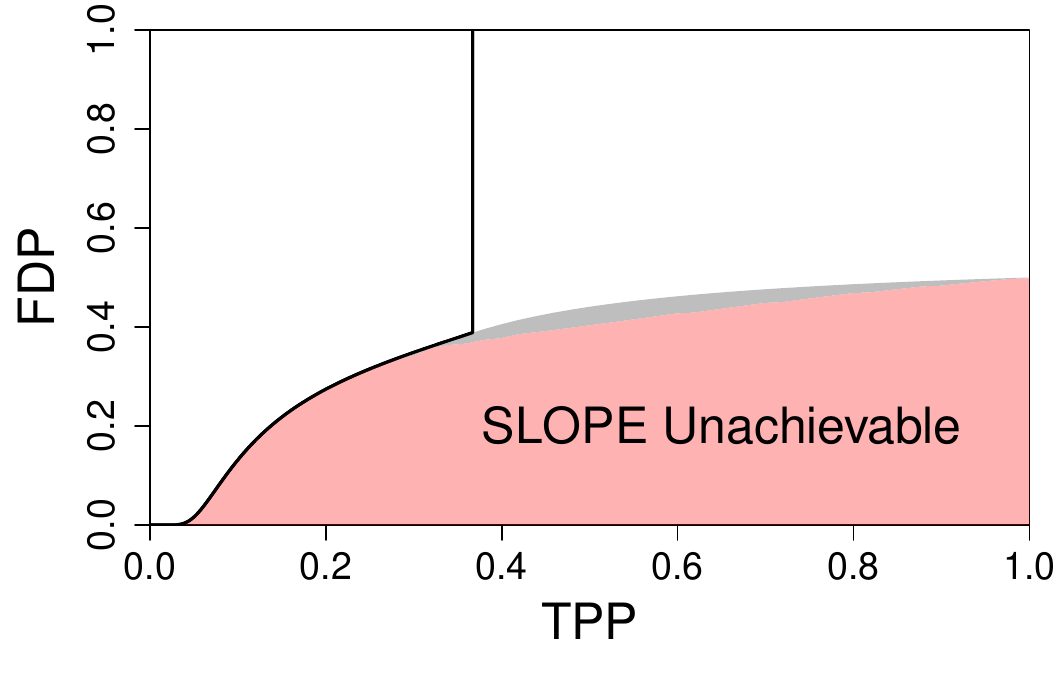}
	\caption{Examples of the TPP--FDP trade-off curve, with $(\delta,\epsilon)=(0.3,0.2)$ on the top panel and $(0.3,0.5)$ on the bottom. The left plot is the Lasso trade-off curve and the right plot describes the SLOPE trade-off gain. Neither the Lasso nor SLOPE can approach the red regions. The gray regions are sandwiched by the upper and lower bounds on the SLOPE trade-off.}
	\label{fig:lasso_phase}
\end{figure}


First of all, the two right plots of \Cref{fig:lasso_phase} show that the lower bound and the upper bound for $q_{_{\textnormal{\tiny SLOPE}}}$ are very close to each other when $0 \le \TPP \le \tppmax$ (recall that $\tppmax = 1$ in the subcritical regime). As a matter of fact, the upper bound in this regime is given by \citet{lassopath}, which showed that, under the working assumptions, there exists a function $q^{\star}_{\textnormal{\tiny Lasso}}(\cdot; \delta, \epsilon)$ such that
\begin{equation*}
	\FDP_{\textnormal{\tiny Lasso}}(\bet,\lambda) \geq q_{\textnormal{\tiny Lasso}}^\star(\TPP_{\textnormal{\tiny Lasso}}(\bet,\lambda);\delta,\epsilon) - 0.0001,
\end{equation*}
holds with probability tending to one as $n, p \goto \infty$. Here 0.0001 can be replaced by any arbitrarily small positive constant. Moreover, $q_{\textnormal{\tiny Lasso}}^\star$ is tight in the sense that the Lasso can come arbitrarily close to any point on this curve by specifying a prior and a penalty parameter (see refined results in \cite{wang2020price}). Recognizing that the Lasso is an instance of SLOPE, the tightness of $q_{\textnormal{\tiny Lasso}}^\star$ allows us to set $q^\star(u) = q_{\textnormal{\tiny Lasso}}^\star(u)$ for $0 \le u \le \tppmax$. To be more precise, letting $t^\star(u)$ be the largest positive root of the equation
\begin{equation}\label{eq:t_star_eq}
	\frac{2(1 - \epsilon)\left[ (1+x^2)\Phi(-x) - x\phi(x) \right] + \epsilon(1 + x^2) - \delta}{\epsilon\left[ (1+x^2)(1-2\Phi(-x)) + 2x\phi(x) \right]} = \frac{1 - u}{1 - 2\Phi(-x)},
\end{equation}
we have
\begin{align}
	q^{\star}(u;\delta,\epsilon)&=
	\begin{cases}
		q_{\textnormal{\tiny Lasso}}^\star(u;\delta,\epsilon)= \dfrac{2(1-\epsilon)\Phi(-t^{\star}(u))}{2(1-\epsilon)\Phi(-t^{\star}(u))+\epsilon u}, &\text{if $u \le \tppmax(\delta,\epsilon)$,} \\[20pt]
		\dfrac{\epsilon(1-\epsilon) u-\epsilon^{\star}(1-\epsilon)}{\epsilon(1-\epsilon^{\star}) u-\epsilon^{\star}(1-\epsilon)}, &\text{if $u > \tppmax(\delta,\epsilon)$}.
	\end{cases}
	\label{cor:whole tradeoff}
\end{align}
In the above expressions, $\phi(\cdot)$ and $\Phi(\cdot)$ are the probability density function and cumulative distribution function of the standard normal distribution, respectively.


Returning to the lower bound, in stark contrast, the situation becomes much more challenging. To be sure, to obtain a lower bound requires a good understanding of the superiority of sorted $\ell_1$ regularization over its usual $\ell_1$ counterpart. From a theoretical viewpoint, a major difficulty in the analysis of SLOPE arises from the \textit{non-separability} of sorted $\ell_1$ regularization. Note that the non-separability results from the sorting operation in the penalty term $\sum_{i=1}^p \lambda_i|b|_{(i)}$ in the SLOPE optimization program~\eqref{eq:SLOPE_cost}. To tackle this technical issue, in this paper we formulate the SLOPE trade-off as a calculus of variations problem and further cast it into infinite-dimensional convex optimization problems (see more details in \Cref{sec:lower bound}).


In a nutshell, the flexibility of the SLOPE regularization sequence seems to only bring up limited improvement on the trade-off between the TPP and FDP below the DT power limit. However, the two right plots of \Cref{fig:lasso_phase} present a noticeable departure between the two bounds when the TPP is slightly below $\tppmax$. This departure is not an artifact of our analysis. Indeed, in \Cref{sec:difference upper lower} we provide a problem instance whose asymptotic TPP and FDP trade-off falls strictly between the upper bound and the lower bound:
\[
q_\star(u) + 0.0001 < \FDP < q^\star(u) - 0.0001,
\]
and $\TPP \approx u < \tppmax$ with probability tending to one. 



\subsection{On model selection and estimation}


An important but less-emphasized point is that the above-mentioned comparison between the two methods is over the \textit{lower envelope} of all the instance-specific problems. In this regard, it would be too quick to conclude that the flexibility of the penalty sequence does not gain any benefits for SLOPE, even at points where $q_\star(u)$ may be very close to $q^\star_{\textnormal{\tiny Lasso}}(u)$. Under the working hypotheses, indeed, we can formally prove that SLOPE is superior to the Lasso in the sense that we can always find a SLOPE regularization prior that strictly improves the Lasso on the same linear regression problem in terms of both model selection and estimation. Below, we let $\widehat\bet$ denote the SLOPE or the Lasso estimate, and use the subscript to distinguish between the two methods.


\begin{theorem}\label{thm:instance better}
	Under the working assumptions, namely (A1), (A2), and (A3), given any bounded signal prior $\Pi$ and any Lasso regularization parameter $\lambda > 0$, there exists a SLOPE regularization $\Lambda$ such that the following inequalities hold simultaneously with probability tending to one:
	\begin{enumerate}
		\item[(a)] $\TPP_{\textnormal{\tiny SLOPE}}(\bet,\blam) > \TPP_{\textnormal{\tiny Lasso}}(\bet,\lambda)$;
		\item[(b)] $\FDP_{\textnormal{\tiny SLOPE}}(\bet, \blam)<\FDP_{\textnormal{\tiny Lasso}}(\bet,\lambda)$;
		\item[(c)] $\|\widehat\bet_{\textnormal{\tiny SLOPE}}(\bet, \blam)-\bet\|^2 < \|\widehat\bet_{\textnormal{\tiny Lasso}}(\bet,\lambda)-\bet\|^2$.
	\end{enumerate}
\end{theorem}
This theorem shows that SLOPE can outperform the Lasso from both the model selection and estimation viewpoints. We stress, however, that the result is \emph{non-constructive} in that it does not provide the actual SLOPE penalty vector $\boldsymbol{\lambda}$ giving the good performance—it only claims that one \emph{exists}. In practice, one would likely want to find a SLOPE sequence to optimize performance along one attribute only, depending on the goal (i.e., by considering model selection or estimation separately). The task of finding optimally performing SLOPE penalty sequences for any given fixed prior is an important open question, which we leave for future work.

The proof strategy of Theorem \ref{thm:instance better} leverages a simple form of SLOPE regularization sequences that admits two distinct values (see \eqref{eq:theta_func}). Due to space constraints, we relegate the proof of this theorem to \Cref{sec:slope-instance-better}. It is somewhat surprising that such a simple two-level sequence can already exploit the benefits of using SLOPE over the Lasso. 

As an aside, we remark that SLOPE has been shown to achieve the asymptotically exact minimax estimation when the sparsity level is much lower than considered in the present paper, largely owing to the adaptivity of sorted $\ell_1$ regularization~\citep{SLOPE2}. When it comes to the Lasso, however, cross validation is needed to select a penalty parameter that enables the Lasso to achieve similar estimation performance, which however is not exact as the constant is not sharp~\citep{bellec2018slope}.

\section{Preliminaries for Proofs}
\label{sec:prelim}
In this section, we collect some preliminary results about SLOPE and AMP theory that allow us to get analytic expressions of the TPP and FDP asymptotically. Informally speaking, the AMP theory given in \citet[Theorem 3]{ourAMP} characterizes the \emph{asymptotic} joint distribution of the SLOPE estimator $\widehat\bet$ and the true regression coefficients $\bet$ (similar results are given in \citet[Theorem 1]{SLOPEasymptotic} using the convex Gaussian minimax theory (CGMT) instead of AMP). Notably, since $\widehat\bet$ depends on $(\bet,\blam)$, when studying asymptotic properties of $\widehat\bet$, we will work with their asymptotic distributions $(\Pi,\Lambda)$. In this way, we drop the dependence on finite-sample quantities like $n, p$ and the sparsity level $|\{j:\beta_j \neq 0\}|$ and instead work with asymptotic quantities such as $(\delta,\epsilon)$ henceforth. 

To be specific, under pseudo-Lipschitz functions (see \citet[Definition 3.1]{ourAMP}) on $(\widehat{\bet},\bet)$, the asymptotic distribution of the SLOPE (including the Lasso) estimator  $\widehat{\bet}$, which we denote as $\widehat{\Pi}$, can be described as
\begin{align}
	\label{eq:beta_dist}
	\widehat{\Pi}\overset{\mathcal{D}}{=}\h{\Pi+\tau Z,\A\tau}(\Pi+\tau Z),
\end{align}
where $Z$ is an independent standard normal and the superscript $\mathcal{D}$ means "in distribution". We will refer to $\eta$ (to be introduced in \eqref{eq:yue_limit}) as the \textit{limiting scalar function} in \citet{SLOPEasymptotic}, and $(\tau,\A)$ is the unique solution to the \textit{state evolution} and the \textit{calibration} equations
\Align{
	\tau^2&=\sigma^2+\frac{1}{\delta}\E\left(\h{\Pi+\tau Z, \A\tau}(\Pi+\tau Z)-\Pi\right)^2,
	\label{eq:SE}
	\\
	\Lambda&\overset{\mathcal{D}}{=}\A\tau\left(1- \frac{1}{\delta}\E\Big(\hprime{\Pi+\tau Z, \A\tau}(\Pi+\tau Z)\Big) \right).
	\label{eq:cali}
}

In order to discuss properties of the limiting scalar function $\h{}$, we first introduce the SLOPE proximal operator on $(\bm y,\thet)\in\R^\p\times\R^\p$, where $\thet$ is proportional to $\blam$ and $\theta_1\geq\theta_2\geq\cdots\geq\theta_\p\geq 0$ with at least one inequality. We define the proximal operator as
\begin{align}
	\label{eq:slope_prox}
	\prox_J(\bm y;\thet)&:=\arg\min_{\bm  b \in \mathbb{R}^p}\Big\{\frac{1}{2}\|\bm y - \bm b\|^2+ J_{\thet}(\bm b)\Big\},
\end{align}
where $J_{\thet}(\bm b)=\sum_{i=1}^p \theta_i|b|_{(i)}$. In the Lasso case when the penalty parameter is a constant, the proximal operator reduces to the soft-thresholding function:
$$\prox_J(\y;\theta)=\eta_{\text{soft}}(\y;\theta):=\text{sign}(\y)\cdot \max\{|\y|-\theta,0\}.$$

Generally speaking, the SLOPE proximal operator in \eqref{eq:slope_prox} is \textit{adaptive} and \textit{non-separable}, in the sense that an element of the output generally will depend on all elements of the input. As a concrete example, we obtain via \Cref{alg:prox} that the proximal operator for SLOPE is given by
\begin{align*}
	&\prox_J([20,13,10,6,4];[12,10,5,5,5]) =  \eta_{\textnormal{soft}}([20,13,10,6,4];[12,9,6,5,5])
	= [8,4,4,1,0].
\end{align*}

On the one hand, the adaptivity arises from the fact that larger penalties are applied to larger elements of the input. On the other hand, for example, two elements of input $[13,10]$ are not directly thresholded by the penalty $[10,5]$, but rather an averaging step is triggered by the existence of the other inputs, which gives an effective threshold of $[9,6]$. This is illustrated in Figure~\ref{fig:prox}.

\begin{figure}[!htb]
	\centering
	\includegraphics[width=0.3\linewidth]{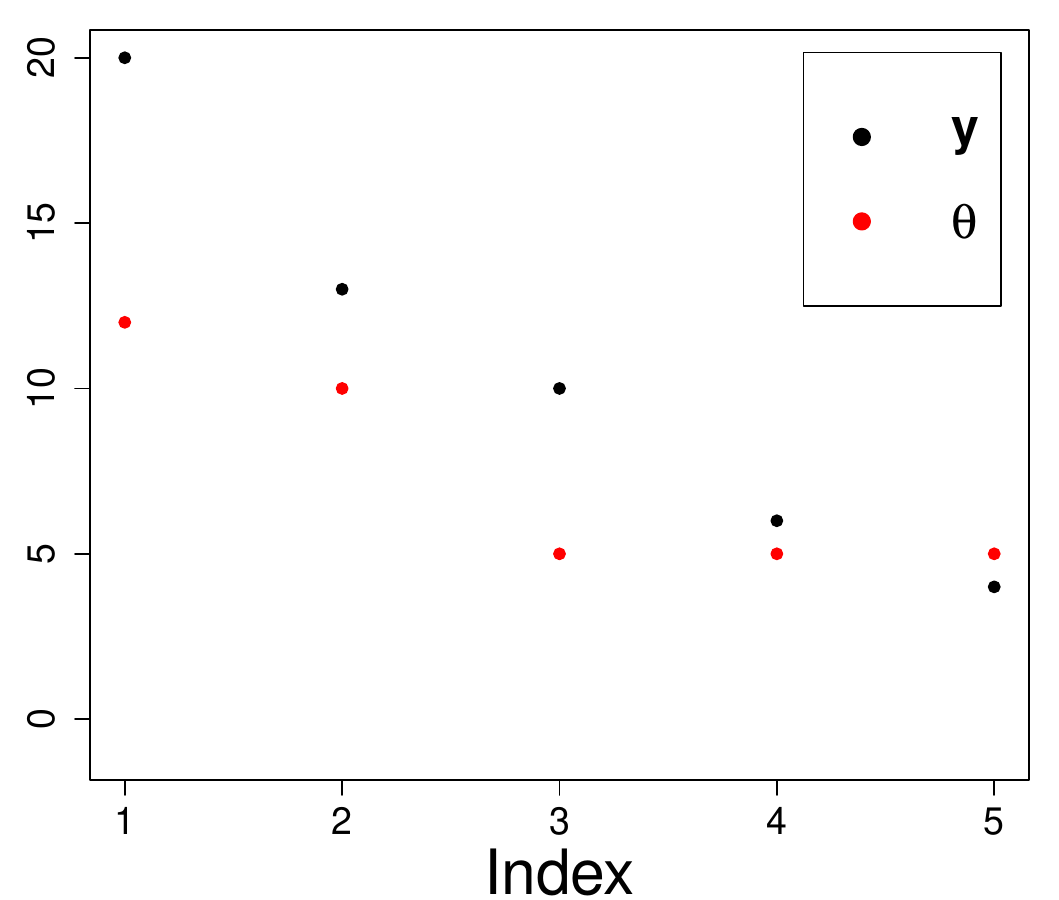}
	\includegraphics[width=0.3\linewidth]{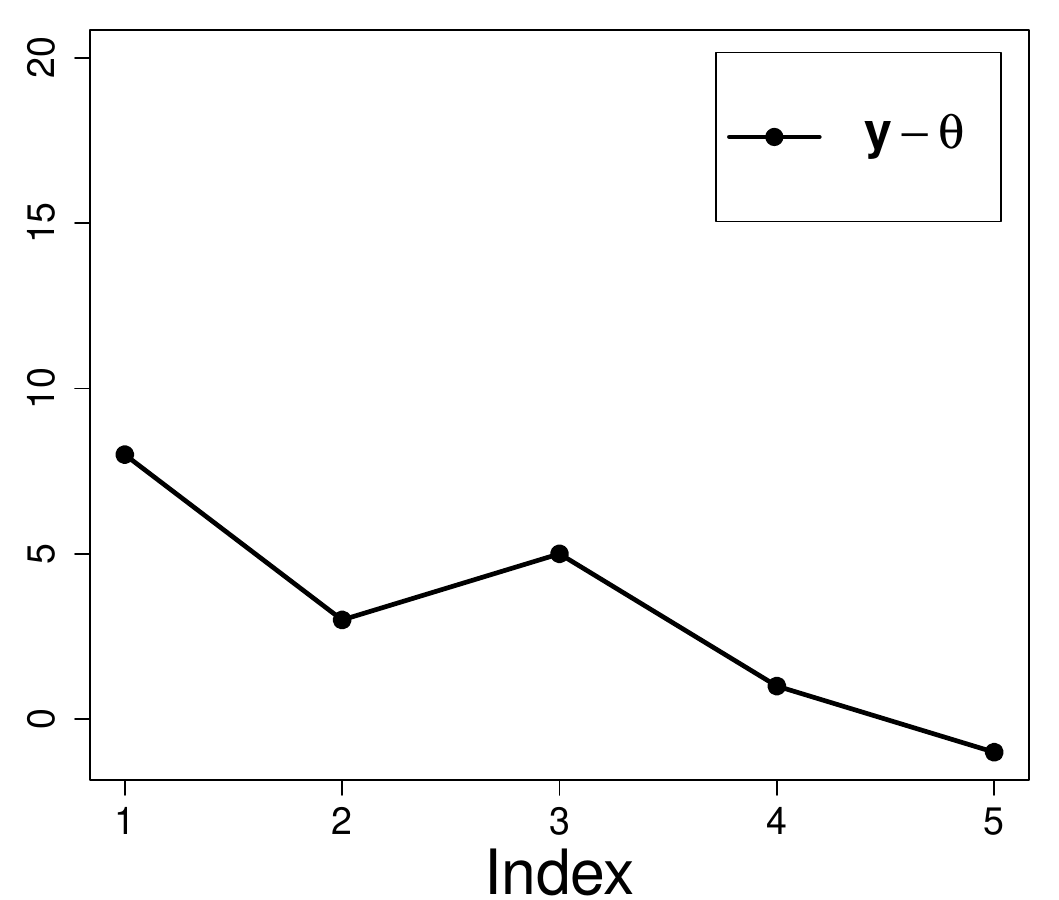}
	\includegraphics[width=0.3\linewidth]{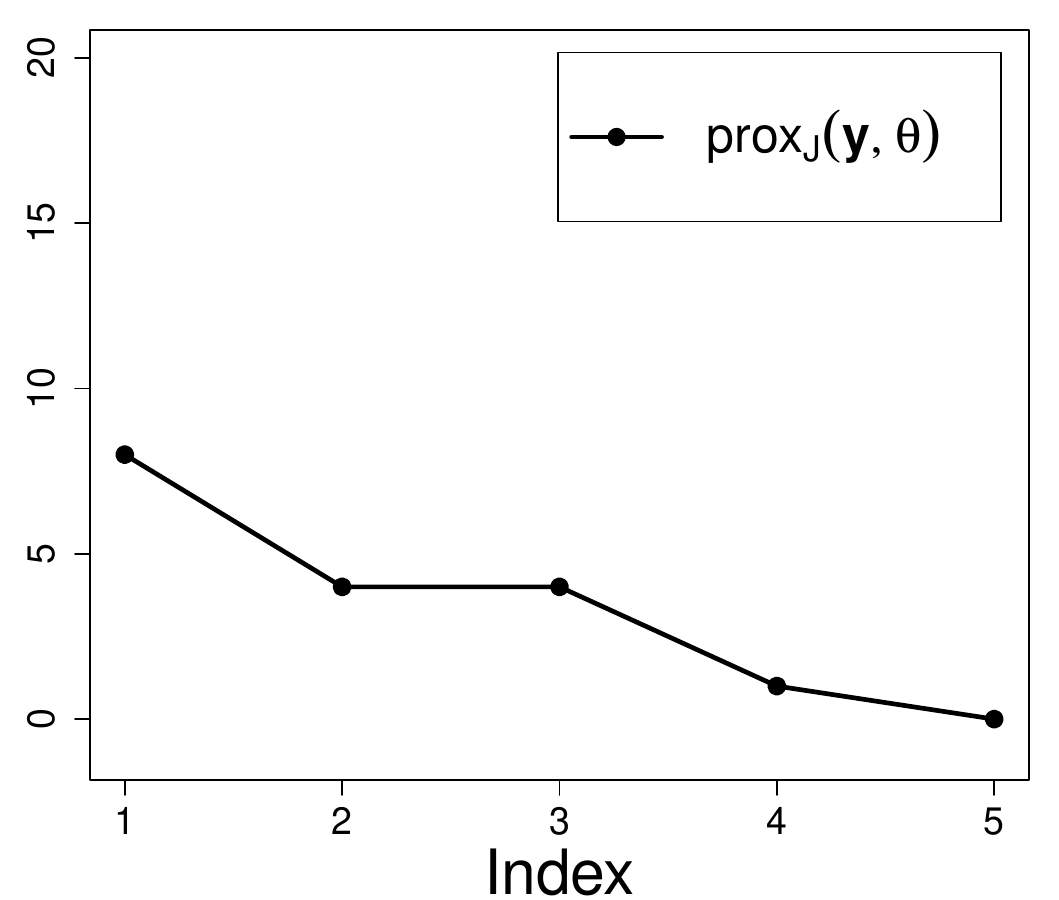}
	\caption{Illustration of how the SLOPE proximal operator can be interpreted as using an effective threshold. The leftmost figure plots two vectors $\mathbf{y}$ and $\boldsymbol{\theta}$. The middle image plots their difference $\mathbf{y} - \boldsymbol{\theta}$ and the rightmost image plots the output of the proximal operator $\prox_J(\mathbf{y}; \boldsymbol{\theta})$.}
	\label{fig:prox}
\end{figure}

Although the SLOPE proximal operator given in \eqref{eq:slope_prox} is non-separable, nevertheless, as introduced in \citet[Proposition 1]{SLOPEasymptotic}, the SLOPE proximal operator is \textit{asymptotically separable}: for sequences $\{\thet(p)\}$ and $\{\bm v(p)\}$ growing in $p$ with empirical distributions that weakly converge to distributions $\Theta$ and $V$, respectively, there exists a limiting scalar function $\h{}$ (determined by $\Theta$ and $V$) such that as $\p\goto\infty$,
\be
\frac{1}{p}\norm{\prox_{J}(\bm v{(p)};{\thet{(p)}}) - \h{V, \Theta} (\bm v{(p)})}^2 \rightarrow 0.
\label{eq:yue_limit}	\ee
The work in \cite{SLOPEasymptotic} discusses many properties of this limiting scalar function, $\h{}$. Indeed, it is shown to be odd, increasing, Lipschitz continuous with constant 1 and applied coordinate-wise to $\bm v{(\p)}$ (hence it is separable; see \citet[Proposition 2]{SLOPEasymptotic}). In more details, $\h{V, \Theta}(x)$ takes a scalar input, $x$, and performs soft-thresholding with a penalty adaptive to $x$ in a way that depends on $V$ and $\Theta$,  meaning there is an input-dependent penalty $\lambda_{V,\Theta}(x)$ such that $\h{V, \Theta}(x)=\eta_{\textnormal{soft}}(x;\lambda_{V,\Theta}(x))$. More details on the adaptive penalty function that relates the SLOPE proximal operator to the soft-thresholding function can be found in \Cref{bridge}.  

We now discuss in more detail the so-called state evolution and calibration equations given in \eqref{eq:SE} and \eqref{eq:cali}. We refer to $\A$, which is defined implicitly via \eqref{eq:cali}, as the \textit{normalized penalty} distribution. Notice that $\A$ only differs from the original penalty distribution $\Lambda$ by a constant factor. In fact, there exists a one-to-one mapping between $\A$ and $\Lambda$ by \citet[Proposition 2.6]{ourAMP}, allowing one to analyze in either regime flexibly. Moreover, for a fixed $\Pi$, the quantity $\tau(\A)$ can be uniquely derived from \eqref{eq:SE} and, as shown in \citet[Corollary 3.4]{ourAMP}, it can be used to characterize the estimation error via $\|\widehat{\bet}-\bet\|^2/p\to\delta(\tau^2-\sigma^2)$. In this work, we will use $\tau := \tau(\Pi,\Lambda)$ as a factor to define the \textit{normalized prior},
\begin{align*}
	\pi(\Pi,\A):=\Pi/\tau(\Pi,\Lambda),
\end{align*}
and, in particular, when it is clear from the context, we will use $(\Pi,\Lambda)$ and $(\pi,\A)$ interchangeably since there exists a bijective calibration between the original problem instance and the normalized one provided by the fixed point recursion for the state evolution and the calibration mappings, \eqref{eq:SE} and \eqref{eq:cali}. We refer the interested readers to \Cref{app:prelim} for a discussion of many nice properties of this fixed point recursion, such as the explicit form of the divergence $\h{}'$.

Under the characterization of the asymptotic SLOPE distribution given in \eqref{eq:beta_dist}, we define $\fdpinf(\Pi, \Lambda)$ and $\tppinf(\Pi, \Lambda)$ as the large system limits  of FDP and TPP. The proof of convergence in probability is given in the next lemma.  We will eventually let $\GAM \rightarrow 0$, and in order for the FDP and TPP to converge, we consider $\FDP_\GAM$ and $\TPP_\GAM$ in \eqref{eq:new_defs} with $\GAM$ in the set
	\begin{align}
		\Xi:=\{\GAM:\PP(\widehat{\Pi}(\Pi,\Lambda)=\GAM)=0\},
		\label{eq:valid xi}
	\end{align}
	where $\widehat\Pi$ is the limiting distribution of $\widehat{\beta}_j$, defined in \eqref{eq:beta_dist}.

	\begin{lemma}\label{lem:A1}
		Under the working assumptions, namely (A1), (A2), and (A3), for $\GAM\in\Xi$ in \eqref{eq:valid xi}, the SLOPE estimator $\widehat{\bet}(\blam)$ with the penalty sequence $\blam$ satisfies
		\begin{equation*}
			\begin{split}
				\FDP_\GAM(\bet,\blam)&=\frac{|\{j:|\widehat\beta_j|>\GAM,\beta_j=0\}|}{|\{j:|\widehat\beta_j|>\GAM\}|}\overset{P}{\to}\FDP_\GAM^\infty(\Pi,\Lambda):=\frac{(1-\epsilon)\PP\left(\left|\h{\Pi+\tau Z,\A\tau}(\tau Z)\right|>\GAM\right)}{\PP\left(\left|\h{\Pi+\tau Z,\A\tau}(\Pi+\tau Z)\right|>\GAM\right)},
				\\
				\TPP_\GAM(\bet,\blam)&=\frac{|\{j:|\widehat\beta_j|>\GAM,\beta_j\neq 0\}|}{|\{j:\beta_j\neq 0\}|}\overset{P}{\to}\TPP_\GAM^\infty(\Pi,\Lambda):=\PP\left(\left|\h{\Pi+\tau Z,\A\tau}(\Pi^\star+\tau Z)\right|>\GAM\right),	
			\end{split}
		\end{equation*}
		where superscript $P$ denotes convergence in probability,
		$Z$ is a standard normal independent of $\Pi$, and $(\tau,\A)$ is the unique solution to the state evolution \eqref{eq:SE} and calibration \eqref{eq:cali}. Furthermore, $\Pi^{\star}:=\left(\Pi|\Pi\neq 0\right)$ is the signal prior distribution of the non-zero elements.
	\end{lemma}
	
	By the continuity of the probability measure, we obtain
	\begin{equation}
		\begin{split}
			\lim_{\GAM\to 0}\FDP_\GAM^\infty(\Pi,\Lambda)&=\fdpinf(\Pi,\Lambda):=\frac{(1-\epsilon)\PP\left(\h{\pi+Z,\A}(Z)\neq 0\right)}{\PP\left(\h{\pi+Z,\A}(\pi+Z)\neq 0\right)},
			\\
			\lim_{\GAM\to 0}\TPP_\GAM^\infty(\Pi,\Lambda)&=\tppinf(\Pi,\Lambda):=\PP\left(\h{\pi+Z,\A}(\pi^\star+Z)\neq 0\right).	
			\label{eq:inf_props}
		\end{split}
	\end{equation}
	Here, $\pi=\Pi/\tau$ is the normalized prior distribution and $\pi^{\star}:=\Pi^{\star}/\tau$. We give the proof of \Cref{lem:A1} in \Cref{app:lemma31proof} by extending  \citet[Theorem B.1]{bogdan2013statistical}.

Following the notions of $\fdpinf$ and $\tppinf$ given in Lemma~\ref{lem:A1}, we mathematically define the SLOPE trade-off curve as the envelope of all achievable SLOPE $(\tppinf,\fdpinf)$ pairs:
\begin{align*}
	q_{_{\textnormal{\tiny SLOPE}}}(u;\delta,\epsilon) 
	:=\inf_{(\Pi,\Lambda): \, \tppinf(\Pi,\Lambda)=u}\fdpinf(\Pi,\Lambda).
\end{align*}

To study the SLOPE trade-off, we will make use of a critical concept, the \textit{zero-threshold} $\alpha(\Pi,\Lambda)$, which will be defined in \Cref{zero threshold}. Using the zero threshold, the limiting values in \eqref{eq:inf_props} can be simplified to
\begin{align}
	\begin{split}
		\tppinf(\Pi,\Lambda)
		&=\PP(|\pi^\star+Z|>\alpha(\Pi,\Lambda)),
		\\
		\fdpinf(\Pi,\Lambda)
		&=\dfrac{2(1-\epsilon)\Phi(-\alpha(\Pi,\Lambda))}{2(1-\epsilon)\Phi(-\alpha(\Pi,\Lambda))+\epsilon\cdot\tppinf(\Pi,\Lambda)}.
	\end{split}
	\label{eq:tpp fdp zero-threshold}
\end{align}
Note from the equations above that for fixed $\tppinf=u$, the formula of $\fdpinf$ is decreasing in $\alpha$. Therefore we consider the maximum of feasible zero-thresholds, $$\alpha^\star(u):=\sup_{(\Pi,\Lambda): \, \TPP^\infty=u}\alpha(\Pi,\Lambda),$$
in order to derive the minimum $\fdpinf$ on the SLOPE trade-off
\begin{equation}
	q_{_{\textnormal{\tiny SLOPE}}}(u;\delta,\epsilon) 
	:=\dfrac{2(1-\epsilon)\Phi(-\alpha^{\star}(u))}{2(1-\epsilon)\Phi(-\alpha^{\star}(u))+\epsilon u}.
	\label{eq:maximize zero threshold}
\end{equation}

	\section{Lower bound of SLOPE trade-off}
\label{sec:lower bound}


The main purpose of this section is to provide a lower bound $q_\star$ on $q_{_{\textnormal{\tiny SLOPE}}}$. We accomplish this by (equivalently) giving an upper bound for $\alpha^\star(u)$ for \emph{fixed} $u$, which we denote as $t_\star(u)$. As we shall see, in contrast to Lasso, our derivation for SLOPE requires non-standard tools from the calculus of variations and quadratic optimization programming. The optimization problem is a constrained one involving the SLOPE penalty and the probability density function of the normalized prior $\pi$ as the decision variables, subject to the fixed $\TPP=u$ and the monotonicity of the penalty.

To construct the upper bound $t_\star(u)$, we examine the state evolution \eqref{eq:SE}, which gives
$$\tau^2\geq\frac{1}{\delta}\E\left(\h{\Pi+\tau Z,\A\tau}(\Pi+\tau Z)-\Pi\right)^2=\frac{\tau^2}{\delta}\E\left(\h{\pi+Z,\A}(\pi+Z)-\pi\right)^2.$$
Rearranging the above inequality yields the state evolution condition
\begin{align}
\label{eq:SE contraint_quote}
E(\Pi,\Lambda):=\E\left(\h{\pi+Z,\A}(\pi+Z)-\pi\right)^2\leq \delta.
\end{align}
Here the quantity $E(\Pi,\Lambda)$ can be viewed as the asymptotic mean squared error between the SLOPE estimator and the truth, scaled by $1/\tau^2$, since $\|\widehat\bet-\bet\|/p\to\tau^2 E(\Pi,\Lambda)$ in probability by \citet[Corollary 3.4]{ourAMP}.

Before we proceed, we first introduce an important (scalar) quantity that governs the sparsity, the TPP, and the FDP of the SLOPE estimator and will be used throughout the paper.

\begin{definition}\label{zero threshold}
Let $ (\Pi,\Lambda) $ be a pair of prior and penalty distributions (or, equivalently, the normalized $(\pi,\A)$) and suppose $ \alpha(\Pi,\Lambda) $ is a positive number such that $\h{\pi+Z,\A}(x)= 0$ if and only if $|x|\leq \alpha(\Pi,\Lambda)$. Then we say that $\alpha = \alpha(\Pi,\Lambda)$ is the \emph{zero-threshold}.
\end{definition}

Intuitively, the zero-threshold is a positive threshold, below which, the input is mapped to zero. Note that the necessary condition \eqref{eq:SE contraint_quote} sets the feasible domain of $(\pi,\A)$ pairs and thus prescribes limits to the zero-threshold $ \alpha $. In the Lasso case, the zero-threshold is indeed equivalent to the normalized penalty scalar $\A$; but in SLOPE, it is a quantity derived from the normalized penalty distribution $\A$ in a highly nontrivial manner (see \Cref{lambda meets} for details).

Next, we state another useful definition. Recall from \Cref{sec:prelim} that the limiting scalar function $\h{}$ of SLOPE is separable and assigns a different penalty to different inputs. We therefore define the \textit{effective penalty function} accordingly.
\begin{definition}
\label{def:essential penalty}
Given a normalized pair of prior and penalty $(\pi,\A)$, the effective penalty function $\widehat{\mathrm{A}}_\textnormal{eff}:\R\rightarrow\R_+$ is a function such that
$$\eta_{\textnormal{soft}}(x;\widehat{\mathrm{A}}_\textnormal{eff}(x))=\h{\pi+Z,\A}(x).$$
\end{definition}
It is not hard to show that $\widehat{\mathrm{A}}_\text{eff}$ is well-defined. In fact, given $\h{\pi+Z,\A}$, we can represent $\widehat{\mathrm{A}}_\text{eff}$ via the zero-threshold from \Cref{zero threshold}, namely,
\begin{align*}
\widehat{\mathrm{A}}_\text{eff}(x)=
\begin{cases}
	x-\h{\pi+Z,\A}(x) & \text{ if $x>\alpha(\pi,\A)$},
	\\
	-x+\h{\pi+Z,\A}(x) & \text{ if $x<-\alpha(\pi,\A)$},
	\\
	\alpha(\pi,\A) & \text{ if $|x|<\alpha(\pi,\A)$}.
\end{cases} 
\end{align*}

Equipped with this effective penalty function, we can rewrite the state evolution condition \eqref{eq:SE contraint_quote} as
\begin{align*}
F_{\alpha}[\widehat{\mathrm{A}}_\textnormal{eff},p_{\pi^\star}]:=\E\left(\eta_{\textnormal{soft}}(\pi+Z;\widehat{\mathrm{A}}_\textnormal{eff}(\pi+Z))-\pi\right)^2\leq \delta,
\end{align*}
in which the functional objective $F_\alpha$ is defined on the effective penalty function $\widehat{\mathrm{A}}_\textnormal{eff}$ as well as the probability density function of $\pi^\star$. Note here that $\pi^\star$ and $\pi$ determine each other uniquely since $\pi^\star:=\pi|\pi\neq 0$. We provide an explicit expression for $ F_{\alpha}[\widehat{\mathrm{A}}_\textnormal{eff},p_{\pi^\star}] $ in \eqref{eq:F_alpha definition}.


Since the constraint \eqref{eq:SE} remains the same if $\pi$ is replaced by $|\pi|$, we assume $\pi\geq 0$ without loss of generality. We minimize $F_{\alpha}[\widehat{\mathrm{A}}_\textnormal{eff},p_{\pi^\star}]$ over the functional space of $(\widehat{\mathrm{A}}_\text{eff}, p_{\pi^\star})$ through a relaxed variational problem: 
\begin{align}
\begin{split}
	&\min_{\AA,\rho\geq 0} \quad F_\alpha[\AA,\rho]
	\\
	&\text{s.t.} \quad \AA(\alpha)\geq \alpha, \AA'(z)\geq 0 \text{ for all } z\geq\alpha,
	\\
	&\int_{0}^\infty \rho(t)dt=1,\int_{0}^\infty [\Phi(t-\alpha)+\Phi(-t-\alpha)]\rho(t)dt=u.
\end{split}
\label{eq:variational optimization}
\end{align}

Here the function $\AA$ is implicitly defined on $[\alpha,\infty)$ as $\AA(z)=\alpha$ for $0\leq z<\alpha$ and $\rho$ is a probability measure defined on $[0,\infty)$. 
We remark that the constraints for $\AA$ in problem \eqref{eq:variational optimization} are derived from the properties of $\widehat{\mathrm{A}}_\text{eff}$ in \Cref{bridge}, i.e. $\AA'\geq 0$ comes from \Cref{fact:hat lam decrease} and the boundary condition $\AA(\alpha)\geq\alpha$ comes from \Cref{lambda meets}. Because some additional properties of $\widehat{\mathrm{A}}_\text{eff}$ may have been excluded in the relaxation, these constraints are only necessary and may not be sufficient. Therefore,
$$\min_{(\widehat{\mathrm{A}}_\text{eff},p_{\pi^\star})} F_\alpha[\widehat{\mathrm{A}}_\text{eff},p_{\pi^\star}]\geq \min_{(\AA,\rho)} F_\alpha[\AA,\rho],$$
with the inequality possibly being strict, provided the left optimization problem above is solved subject to (i) $\widehat{\mathrm{A}}_\text{eff}$ corresponds to the effective penalty in the limiting scalar function; and (ii) $p_{\pi^\star}$ is a probability density function such that $\tppinf=\P(|\pi^\star+Z|>\alpha)=u$.

Leveraging the above relaxation \eqref{eq:variational optimization}, in order to lower bound $q_{_{\textnormal{\tiny SLOPE}}}$ in \eqref{eq:maximize zero threshold}, we can analogously define the maximum feasible zero-threshold $\alpha^\star(u)$ and upper bound it with $t_\star(u)$ as follows:
\begin{align}
\alpha^\star(u):=\sup\left\{\alpha: \min_{(\Pi,\Lambda)} F_\alpha[\widehat{\mathrm{A}}_\text{eff},p_{\pi^\star}]\leq \delta\right\}\leq t_\star(u):=\sup\left\{\alpha: \min_{(\AA,\rho)} F_\alpha[\AA,\rho]\leq \delta\right\}.
\label{eq:t_star maybe}
\end{align}

With these definitions in place, we are now in a position to describe the procedure to find the optimal prior and the optimal penalty in problem \eqref{eq:variational optimization}, given $\tppinf=u$ and $\alpha(\Pi,\Lambda)=\alpha$.

\subsection{Optimal prior is three-point prior}
To solve problem \eqref{eq:variational optimization}, we must search over all possible distributions $\pi^\star$, which is generally infeasible. To overcome this obstacle, we use the concept of extreme points (i.e. points that do not lie on the line connecting any other two points of the same set)
to show that the optimal $\pi^\star$ for problem \eqref{eq:variational optimization} is a two-point distribution, having probability mass at only two non-negative (and possibly infinite) values $(t_1,t_2)$. In doing so, we significantly reduce the search domain, from infinite dimensional to two-dimensional.
Because $\pi$ has an additional point mass at 0, the optimal prior $\pi$ (that can achieve minimum FDP when accompanied with the properly chosen penalty) is a three-point prior taking values at $(0, t_1, t_2)$. We recall that the two-point $\pi^\star$ is consistent to the Lasso result in \citet[Section 2.5]{lassopath}, where the optimal $\pi^\star$ is the infinity-or-nothing distribution with $t_1=0^+, t_2=\infty$.

To see that $\pi^\star$ admits a two-point form, suppose that $(\AA^*, \rho^*)$ is the global minimum of problem \eqref{eq:variational optimization}. Then clearly $\rho^*$ is also the global minimum of the following linear problem \eqref{eq:linear programming} with linear constraints.
\begin{align}
\begin{split}
&\min_{\rho\geq 0} \quad F_\alpha[\AA^*,\rho]
\\
&\text{s.t.} \quad 
\int_{0}^\infty \rho(t)dt=1, \int_{0}^\infty [\Phi(t-\alpha)+\Phi(-t-\alpha)]\rho(t)dt=u.
\end{split}
\label{eq:linear programming}
\end{align}

Intuitively, since there are two constraints, we need two parameters (which will be $t_1,t_2$) to characterize the minimum. We formalize this intuition in the next lemma (proved in Appendix \ref{all_other_proofs}) and show that $\rho^*$ indeed takes the form of a sum of two Dirac delta functions.

\begin{lemma}
\label{lem:two point vertex}
If $\rho^*$ is a global minimum of problem \eqref{eq:linear programming}, then 
$$\rho^*(t)=p_1\delta(t-t_1)+p_2\delta(t-t_2)$$ 
for some constants $p_1, p_2, t_1, t_2$, and $p_1+p_2=1$, $p_1,p_2\geq 0$.
\end{lemma}

The above specific form of the optimal $\rho^*$ allows us to search over all $(t_1,t_2)$, each pair of which uniquely corresponds to either a single-point prior $\rho(t;t_1,t_2)=\delta(t-t_1)$ if $t_1=t_2$, or a two-point prior by
\begin{align}
\begin{split}
\rho(t;t_1,t_2)&=p_1\delta(t-t_1)+p_2\delta(t-t_2),
\\
p_1(t_1,t_2)&=\frac{u-[\Phi(t_2-\alpha)+\Phi(-t_2-\alpha)]}{[\Phi(t_1-\alpha)+\Phi(-t_1-\alpha)]-[\Phi(t_2-\alpha)+\Phi(-t_2-\alpha)]},
\\
p_2(t_1,t_2)&=1-p_1(t_1,t_2),
\end{split}
\label{eq:dirac delta prior}
\end{align}
where the last two equations come from the constraints in problem \eqref{eq:linear programming}. 

In light of \Cref{lem:two point vertex}, each pair $(t_1,t_2)$ forms a different instantiation of problem \eqref{eq:variational optimization}, which will be problem \eqref{eq:quadratic programming functional} and whose optimal penalty is denoted by $\AA^*(\cdot;t_1,t_2)$ so as to be explicitly dependent on $(t_1,t_2)$. Before we proceed to optimize the penalty $\AA(\cdot;t_1,t_2)$, we assure the skeptical reader that our procedure -- doing a grid search on $(t_1,t_2)$ and considering the minimal value of all programs \eqref{eq:quadratic programming functional} parameterized by $(t_1,t_2)$ to be equivalent to the minimal value of problem \eqref{eq:variational optimization} -- is indeed a valid approach. This claim is theoretically grounded by noting that $F_\alpha[\AA^*(\cdot;t_1,t_2),\rho(\cdot;t_1,t_2)]$ is continuous in $(t_1,t_2)$. Continuity can be seen from a perturbation analysis of the optimal value in problem \eqref{eq:quadratic programming functional}. In our case, the perturbation analysis is not hard since the constraint is independent of $(t_1,t_2)$ and $F_\alpha$ depends on $\AA^*$ in a strongly-convex manner: a small perturbation in $(t_1,t_2)$ only results in a small perturbation in $\AA^*$ and thus in $F_\alpha[\AA^*(\cdot;t_1,t_2),\rho(\cdot;t_1,t_2)]$. We refer the curious reader to a line of perturbation analysis for such optimization problems in \cite{bonnans2013perturbation,shapiro1992perturbation,bonnans1998optimization}.

\subsection{Characterizing the optimal penalty analytically}
By \Cref{lem:two point vertex}, we reduce the multivariate non-convex problem \eqref{eq:variational optimization} to a set of univariate convex problems \eqref{eq:quadratic programming functional} over $\AA$. In this section, we describe the optimal penalty function $\AA^*(\cdot;t_1,t_2)$, which is the solution to the problem below:
\begin{align}
\begin{split}
&\min_{\AA} \quad F_\alpha[\AA,\rho(\cdot;t_1,t_2)]
\\
&\text{s.t.} \quad \AA(\alpha)\geq \alpha,\quad \AA'(z)\geq 0 \text{ for all } z\geq\alpha.
\end{split}
\label{eq:quadratic programming functional}
\end{align}

This is a quadratic problem with a non-holonomic constraint. To see this, we can expand the objective functional $F_\alpha$ from \eqref{eq:F_alpha definition} and split it into a functional integral that involves $\AA$ and other terms which do not, i.e.
\begin{align*}
F_\alpha[\AA,\rho(\cdot;t_1,t_2)]=\int_{\alpha}^{\infty}L(z,\AA)dz&+\epsilon p_1t_1^2\Big[\Phi(\alpha-t_1)-\Phi(-\alpha-t_1)\Big]
\\
&+\epsilon p_2t_2^2\Big[\Phi(\alpha-t_2)-\Phi(-\alpha-t_2)\Big].
\end{align*}


This split changes our objective functional from $F_\alpha[\AA,\rho(\cdot;t_1,t_2)]$ to the new functional $\int_{\alpha}^{\infty}L(z,\AA)dz$ with
\begin{align}
\begin{split}
L(z,\AA):&=2(1-\epsilon)(z-\AA(z))^2\phi(z)
\\
&+\epsilon p_1\left(\big(z-t_1-\AA(z)\big)^2\phi(z-t_1)+\big(-z-t_1+\AA(z)\big)^2\phi(-z-t_1)\right)
\\
&+\epsilon p_2\left(\big(z-t_2-\AA(z)\big)^2\phi(z-t_2)+\big(-z-t_2+\AA(z)\big)^2\phi(-z-t_2)\right).
\end{split}
\label{eq:functional split}
\end{align}

We will numerically optimize the functional $\int_{\alpha}^{\infty}L(z,\AA)dz$ together with the constraints in problem \eqref{eq:quadratic programming functional}. In addition, although we cannot derive the analytic form of $\AA^*(\cdot;t_1,t_2)$ from problem \eqref{eq:quadratic programming functional}, we can still analytically characterize it at points $ z $ where the monotonicity constraint is non-binding (that is, when $\AA^*(\cdot;t_1,t_2)$ is strictly increasing in a neighborhood of $ z $), as shown in \Cref{app:Euler-Lagrange for three point prior}.


\subsection{Searching over the optimal penalty numerically}\label{subsec:4.3}

To solve the functional optimization problem \eqref{eq:quadratic programming functional}, we approximate it by a discrete optimization problem via Euler's finite difference method. Specifically, we approximate the function $L(z,\AA)$ (and hence $F_{\alpha}$) on a discretized uniform grid of $z$ and solve the resulting quadratic programming problem with linear constraints.

To this end, we denote vectors $\bm z=[\alpha,\alpha+\Delta z,\alpha+2\Delta z,\cdots,\alpha+m\Delta z]$ and $\bm\A=[\AA(\alpha),\AA(\alpha+\Delta z),\cdots,\AA(\alpha+m\Delta z)]$ for some small $\Delta z$ and large $m$. Then problem \eqref{eq:quadratic programming functional} is discretized into the convex quadratic program
\begin{align}
\begin{split}
&\min_{\bm\AA} \quad \bar{F}_{\alpha}(\bm\A;t_1,t_2)
\\
&\text{s.t.}\quad
\begin{pmatrix}
	1&0&0&\cdots&0
	\\
	-1&1&0&\cdots&0
	\\
	0&-1&1&\cdots&0
	\\
	\cdots&\cdots&\cdots&\cdots&\cdots
	\\
	0&\cdots&0&-1&1
\end{pmatrix}\bm\A\geq
\begin{pmatrix}
	\alpha
	\\
	0
	\\
	\vdots
	\\
	0
\end{pmatrix},
\end{split}
\label{eq:quadratic programming vectorized}
\end{align}
in which the new objective $\bar{F}_{\alpha}(\bm\A;t_1,t_2)$ (derived in \eqref{eq:F bar} and also presented below) is the discretized objective of $F_\alpha[\AA,\rho(\cdot;t_1,t_2)]$ from problem \eqref{eq:quadratic programming functional}.

As $\Delta z\to 0$ and $m\to\infty$, problem \eqref{eq:quadratic programming vectorized} recovers problem \eqref{eq:quadratic programming functional} by well-known convergence theory for Euler's finite difference method. To simplify the exposition, we write the objective of problem \eqref{eq:quadratic programming vectorized} in matrix and vector notation as follows:
\begin{align*}
\bm{\mathrm{Q}}&=\textup{diag}\left(2(1-\epsilon)\phi(\bm z)+\epsilon\sum_{j=1,2} p_j\Big[\phi(\bm z-t_j)+\phi(-\bm z-t_j)\Big]\right),
\\
\bm{\mathrm{d}}&=2(1-\epsilon)\bm z\phi(\bm z)+\epsilon\sum_{j=1,2} p_j\Big[(\bm z-t_j)\phi(\bm z-t_j)+(\bm z+t_j)\phi(\bm z+t_j)\Big],
\end{align*}
and observe that
$$\bar{F}_{\alpha}(\bm\A;t_1,t_2)=(\bm\A^\top\bm{\mathrm{Q}}\bm\A-2\bm \A^\top\bm{\mathrm{d}})\Delta z+\epsilon\sum_{j=1,2} p_j t_j^2 \Big[\Phi(\alpha-t_j)-\Phi(-\alpha-t_j)\Big].$$

The discretized problem \eqref{eq:quadratic programming vectorized} is equivalent to a standard quadratic programming problem, whose objective is the discrete version of $\int_{\alpha}^{\infty}L(z,\AA)dz$ in \eqref{eq:functional split},
\begin{align}
\begin{split}
&\min_{\bm\A}\quad\frac{1}{2}\bm \A^\top\bm{\mathrm{Q}}\bm \A-\bm \A^\top\bm{\mathrm{d}}
\\
&\text{s.t.}\quad
\begin{pmatrix}
	1&0&0&\cdots&0
	\\
	-1&1&0&\cdots&0
	\\
	0&-1&1&\cdots&0
	\\
	\cdots&\cdots&\cdots&\cdots&\cdots
	\\
	0&\cdots&0&-1&1
\end{pmatrix}\bm\A\geq
\begin{pmatrix}
	\alpha
	\\
	0
	\\
	\vdots
	\\
	0
\end{pmatrix}.
\end{split}
\label{eq:quadratic programming standard}
\end{align}


\subsection{Solving the quadratic program}

Here we briefly discuss our numerical approach to solving the quadratic program \eqref{eq:quadratic programming standard}. Generally speaking, quadratic programming problems do not admit closed-form solutions. However, they can be efficiently solved by classical numerical methods, including the interior point method \citep{dikin1967iterative,sra2012optimization}, active set method \citep{murty1988linear,ferreau2014qpoases} and other dual methods \citep{goldfarb1983numerically,frank1956algorithm}. In this work, we use the dual method in \cite{goldfarb1983numerically}, as implemented in the R library \texttt{quadprog}, to solve \eqref{eq:quadratic programming standard}.

We remark that problem \eqref{eq:quadratic programming standard} is not the only way to discretize problem \eqref{eq:quadratic programming functional} and we now mention other approaches, which can result in better discretization accuracy. The discretization of problem \eqref{eq:quadratic programming functional} contains two parts: (i) a numerical integration to approximate the objective and (ii) a numerical differentiation to approximate the constraints. 

When formulating the quadratic programming problem \eqref{eq:quadratic programming standard}, we chose to apply the left endpoint rule to approximate the objective integral $\int_{\alpha}^{\infty}L(z,\AA)dz$ in \eqref{eq:functional split} by $(\bm\A^\top\bm{\mathrm{Q}}\bm\A-2\bm \A^\top\bm{\mathrm{d}})\Delta z$, as well as the backward finite difference (with first-order accuracy) to describe the constraint $\AA'(z)\geq 0$.
Alternatively, one can use different numerical quadratures to approximate the integral $\int_{\alpha}^{\infty}L(z,\AA)dz$ or use a change of variable to approximate a different integral. We can also apply different finite differences to discretize the monotonicity constraint in problem \eqref{eq:quadratic programming functional}.

\subsubsection{Numerical integration to approximate the objective}
More specifically, for the approximation of the objective in problem \eqref{eq:quadratic programming functional}, we can alternatively apply numerical quadratures such as the trapezoid rule, Simpson's rule, or Gauss-Laguerre quadrature \citep{salzer1949table} to improve the numerical integration for $\int_{\alpha}^{\infty}L(z,\AA)dz$. On the other hand, we may use a change of variable $z=\frac{x}{1-x}+\alpha$ to transform the integral $\int _{\alpha}^{\infty }L(z)dz$ over an infinite interval $[\alpha,\infty)$ to the integral $\int _{0}^{1}L\left({\frac {x}{1-x}}+\alpha\right){\frac {dx}{(1-x)^{2}}}$ over a finite interval $[0,1]$. This new integral can then be approximated by the same left endpoint rule (or other rules) but with different $\bm{\mathrm{Q}}$ and $\bm{\mathrm{d}}$.

\subsubsection{Numerical differentiation to approximate the constraints}
As for the monotonicity constraint $\AA'(z)\geq 0$, we may alternatively use other difference methods, e.g. the central difference, or higher-order accuracies. Doing so will result in a different matrix that left-multiplies $\bm\A$ in the constraint of \eqref{eq:quadratic programming standard}.

In conclusion, different numerical integration and differentiation schemes will lead to other formulations of the quadratic programming that are different from \eqref{eq:quadratic programming standard}. We do not pursue these additional numerical aspects in the present work.

\subsection{Summary}

To summarize everything so far, the procedure of finding the lower bound $q_\star(u)$ involves the following steps: fixing $\tppinf=u$, we search over a line of zero-thresholds $\{\alpha\}$; for each $\alpha$, we search over a two-dimensional finite grid of $(t_1,t_2)$, each pair defining a standard quadratic programming problem \eqref{eq:quadratic programming standard}; we then solve the quadratic problem and reject $(t_1,t_2)$ if the minimal value of the equivalent problem \eqref{eq:quadratic programming vectorized} is larger than $\delta$; if all $(t_1,t_2)$ are rejected, then the current zero-threshold $\alpha$ is too large to be valid. We set the largest valid zero-threshold as $t_\star(u)$ in \eqref{eq:t_star maybe} and write the lower bound of the $\fdpinf$ as $q_\star(u)=\frac{2(1-\epsilon)\Phi(-t_\star(u))}{2(1-\epsilon)\Phi(-t_\star(u))+\epsilon u}$. Note that $q_\star(u)>0$ for any possible $t_\star(u)$.

We finally mention that, in addition to minimizing $\FDP$ at a fixed $\TPP$ over all penalty-prior pairs, our quadratic programming approach also works when the prior $\Pi$ is fixed. The fixed prior scenario has been extensively studied in \cite{SLOPEasymptotic}, who optimize over the limiting scalar function $\eta$ while we are optimizing over the penalty function $\A_\textnormal{eff}$. Our approach adds a new angle that can be algorithmically more efficient. We defer the details of the procedure to \Cref{app:quadratic programming}.

\subsection{Differences between SLOPE and Lasso}

We end this section by discussing why deriving the SLOPE trade-off is fundamentally more complicated than the Lasso case. We highlight that the variational problem \eqref{eq:variational optimization} is \textit{non-convex}, even though it is convex with respect to each variable $\AA$ and $\rho$ (i.e. it is bi-convex but non-convex). Generally speaking, approximate solutions to non-convex problems are not accompanied by theoretical guarantees, except for some special cases. Our bi-convex problem \eqref{eq:variational optimization} cannot be solved by alternating descent, namely, fixing one variable, optimizing over the other and then alternating. 
Furthermore, our constraints only add another layer of complexity to the problem: in particular, the monotonicity constraint of $\AA$ is non-holonomic (i.e. the constraint $\AA'\geq 0$ does not depend explicitly on $\AA$).

More precisely, the difficulty in directly solving the problem \eqref{eq:variational optimization} is two-fold. The first difficulty lies in the search for the optimal penalty. For the Lasso case, the penalty distribution $\A$ and the penalty function $\widehat\A_\textnormal{eff}$ are not adaptive to the input and hence they both equal the zero-threshold $\alpha$. Therefore, we can perform a grid search on $\A\in\R$ and simply optimize over $\rho$. However, for SLOPE, the penalty $\widehat\A_\textnormal{eff}$ is a \emph{function} and hence it is intractable to search over the SLOPE penalty function space. The functional form of the penalty is the reason we must rely on the calculus of variations to study the associated optimization problem.

To demonstrate the second difficulty, we again consider the convex problem \eqref{eq:linear programming}, which is over the probability density function $\rho$, assuming the optimal penalty $\AA^*$ has been obtained. In the Lasso case, it was shown in \citet[Equation (C.2)]{lassopath} that the optimal $\pi^\star$ is the infinity-or-nothing distribution: $\PP(\pi^\star=0)=1-\epsilon' $ and $\PP(\pi^\star=\infty)=\epsilon'$. In other words, given $\AA^*$, we can easily derive the optimal $\rho$. 
However, a key concavity result in \citet[Lemma C.1]{lassopath}, which holds for Lasso and determines the optimal $\pi^\star$, unfortunately breaks in SLOPE. Therefore, the optimal form of $\pi^\star$ is inaccessible for SLOPE with existing tools, even if the optimal penalty $\AA^*$ is known. 
	

\section{Upper bound of SLOPE trade-off}
\label{sec:mobius}
In this section, we rigorously analyze the  SLOPE trade-off upper boundary curve $q^\star$ (defined in \eqref{cor:whole tradeoff}). As stated in \Cref{thm:upper trade-off-curve-no-q-form}, $q^\star$ takes two forms: below the DT power limit, i.e. when $\tppinf<\tppmax$ for $\tppmax$ defined in \eqref{eq:DT_power_limit}, we have $q^\star=q^{\star}_{\textnormal{\tiny Lasso}}$, and beyond the DT power limit, $q^\star$ is a \Mobius curve.

We start by giving some intuition for why the domain of $q^\star$ is the entire interval $[0,1]$, whereas, the Lasso trade-off curve is only defined on $[0,\tppmax)$. Intuitively, SLOPE is capable of overcoming the DT power limit and achieving 100\% $\TPP$ since it is possible for SLOPE estimators to select all $p$ features, hence, by the definition of TPP (see \Cref{sec:main-results}), one can find a completely dense SLOPE estimator whose TPP is automatically $1$. This is not true for the Lasso, since it can select \emph{at most} $n$ out of $p$ features. The corresponding constraint for the SLOPE estimator follows from the AMP calibration in \eqref{eq:cali} (discussed in detail in \Cref{app:prelim}), namely it says that the number of \emph{unique absolute values} in the entries of the SLOPE estimator is at most $n$ out of $p$. However, this does not directly constrain the sparsity of SLOPE estimator, and thus it can still be dense. In other words, the SLOPE estimator always satisfies
the following:
\begin{align}
	\text{the number of unique non-zero magnitude }|\widehat\beta_i| \text { in } \widehat\bet(\p)  \leq \n.
	\label{eq:uniqueness quota}
\end{align}
Notice that, in the Lasso sub-case, the above implies a direct sparsity constraint
$|\{i:\widehat\beta_i\neq 0\}| \leq\n$ as just discussed, since all non-zero entries in Lasso have unique magnitudes. We also remark that the asymptotically \eqref{eq:uniqueness quota} is a necessary and sufficient condition to satisfy the constraint \eqref{eq:cali}.

With this intuition, we are prepared to prove \Cref{thm:upper trade-off-curve-no-q-form} and show that $q^\star$ indeed serves as an upper bound of $q_{_{\textnormal{\tiny SLOPE}}}$. Following \Cref{prop:whole_sharp}, we have the tightness of $q^\star$ when $u\geq\tppmax$. We will further discuss the proof of \Cref{prop:whole_sharp} in \Cref{sec:approach Mobius}, but leave the full details for \Cref{app:achieving q upper star}. The tightness of $q^\star$ when $u<\tppmax$ follows from the existing tightness result on the Lasso trade-off (see \citet[Section 2.5]{lassopath}), since the Lasso is a sub-case of SLOPE and $q^\star$ matches the Lasso trade-off curve for $u<\tppmax$. Hence, we have the corollary below.
\begin{corollary}
	\label{approach all q upper}
	For any $0\leq u\leq 1$, there exists an $\epsilon'\in[0,\epsilon^\star/\epsilon]$, and values $r(u) \in [0,1]$ and $w(u) \in [0,1]$, both depending on $u$, such that the penalty $\bm\lambda=\bm\lambda_{\sqrt{M},r(u)\sqrt{M},w(u)}$ (defined in \eqref{eq:theta_func}) and the prior $\bet_M(\epsilon')$ (defined in \eqref{eq:sharp_prior}) make SLOPE approach the point $(u,q^\star(u))$ in the sense

		\begin{align*}
			\lim_{M\to\infty}\,\lim_{\GAM\to 0} \lim_{n,p\to\infty} \, \left(\TPP_\GAM(\bet_M(\epsilon'),\bm\lambda), \, \FDP_\GAM(\bet_M(\epsilon'),\bm\lambda)\right)\rightarrow(u,q^\star(u)).
		\end{align*}
	
	Moreover, when $u<\tppmax$, we can set $r(u)=1$, without specifying $w(u)$, and $\epsilon' = \epsilon'(u)$ will also depend on $u$. When $u\geq\tppmax$, we fix $\epsilon'=\epsilon^\star/\epsilon$ and set $r(u)$ via \eqref{u and r} and $w(u)$ via \eqref{r and w} below.
\end{corollary}
An interesting aspect of this result is that there are two different strategies for attaining $q^\star(u)$, depending on whether $\tppinf=u$ is above or below the DT power limit. 
In both cases, we use a two-level penalty $\bm\lambda_{\sqrt{M},r(u)\sqrt{M},w(u)}$ and a sparse prior (see \eqref{eq:sharp_prior}) with very small and very large non-zeros.
However, when $\tppinf=u<\tppmax$, the strategy for attaining $q^\star(u)$ is to vary the proportion of strong signals (which equals $\epsilon\epsilon'$ and $\epsilon'$ varies with $u$), but when $u \geq \tppmax$,  sharpness in the \Mobius part of $q^\star$ is attained by 
keeping the sequence of priors fixed
and instead tuning the ratio between strong and weak penalties.

The sharpness result of \Cref{approach all q upper} shows that over the entire domain, $q^\star$ is arbitrarily closely achievable, thus, $q^\star(u)$ must serve as the upper bound of the minimum $\fdpinf$, $q_{_{\textnormal{\tiny SLOPE}}}(u)$,  hence we have completed the proof of \Cref{thm:upper trade-off-curve-no-q-form}.

\subsection{\Mobius upper bound is achievable}
\label{sec:approach Mobius}
In this section, we will sketch the proof of  \Cref{prop:whole_sharp}, which is used to prove  \Cref{approach all q upper} in the regime $u\geq\tppmax$. To complement \Cref{prop:whole_sharp} and \Cref{approach all q upper}, for concreteness, we give a specific prior and penalty pair in \eqref{eq:mob_part} that approaches $q^\star(u)$ when  $u\geq\tppmax$.  The fully rigorous proof of \Cref{prop:whole_sharp}, together with the derivation of $(r,w)$, is given in \Cref{app:achieving q upper star}.

Before we sketch the proof, we will provide some intuition for what makes the specific choice of priors and penalties behave effectively in terms of reducing the $\fdpinf$ while still driving $\tppinf$ to $1$, in order that we are able to approach $q^\star(u)$ for all  $u\geq\tppmax$.  
We remind the reader that, because there is a one-to-one correspondence between original instance $(\Pi,\Lambda)$ and the normalized $(\pi,\A)$, we will use the two notations interchangeably.

First, for fixed $\tppinf = u$, we can reduce the $\fdpinf$ through a smart use of the priors defined in \eqref{eq:sharp_prior}, where many elements equal $0$ exactly, while some non-zero elements are small (equal to $1/M$) and others large (equal to $M$) with $M$ tending to $\infty$. This is the same strategy as was used for demonstrating the achievability of the Lasso curve in \cite{lassopath}, and the intuition that we present here is based on this analysis. At a high level, extremely strong signals are unlikely to be missed, and thus the $\tppinf$ can be high at the cost of rendering the constraint \eqref{eq:SE contraint_quote} tight. On the other hand, weak signals help reduce the FDP because they are not counted toward the number of false positives and have little influence on \eqref{eq:SE contraint_quote}. Mathematically speaking, for the Lasso, \citet[Lemma C.1]{lassopath} revealed a concave relationship in $\Pi$ between the normalized estimation error $E(\Pi,\Lambda) = \mathbb{E}(\h{\pi + Z, \A}(\pi + Z) - \pi)^2$ in \eqref{eq:SE contraint_quote} and the sparsity $\PP(\h{\pi + Z, \A}(\pi + Z)\neq 0)$, which also depends on the pair $(\Pi,\Lambda)$. 
We remind the reader that, because there is a one-to-one correspondence between original instance $(\Pi,\Lambda)$ and the normalized $(\pi,\A)$, we will use the two notations interchangeably.
The idea is that minimizing $\fdpinf$ corresponds to minimizing the sparsity (this can be seen, for example, by the relationship in \eqref{eq:fdp_inf} where $\kappa(\Pi,\Lambda)$ denotes the sparsity).
Therefore, to find a prior $\Pi$ that satisfies the state evolution condition \eqref{eq:SE contraint_quote}, while minimizing the sparsity, the optimal (normalized) distribution for the non-zero elements, $\pi^\star$, for the Lasso case has probability masses concentrated at the endpoints of the domain, namely $0^+$ and $\infty$. In this way, the form of the signal prior $\Pi$ contributes to reducing the $\fdpinf$ by mixing the weak effects $\beta_i$ with the zero effects. 

Combining the priors discussed above, with a special subset of the possible penalties, namely the two-level penalties defined in \eqref{eq:theta_func}, we are able to reduce the $\fdpinf$ while still increasing the $\tppinf$ to its maximum value of $1$, hence attaining $q^\star(u)$ for all  $u\geq\tppmax$.
Interestingly, the fact that SLOPE can do this, is through its penalty, which mixes the weak predictors $\widehat\beta_i$ and the zero predictors (see \Cref{fig:approach proof}). This mix-up is in fact triggered by the averaging step in the SLOPE proximal operator (see \Cref{alg:prox}; the averaging is determined by the sorted $\ell_1$ norm in the SLOPE problem), which creates non-zero magnitudes that are shared by some predictors and hence maintains the quota of unique magnitudes in \eqref{eq:uniqueness quota}. As a consequence, the SLOPE estimator can overcome the DT power limit (and reach higher $\tppinf$) without violating the uniqueness constraint \eqref{eq:uniqueness quota} on its magnitudes. 

When constructing the two-level penalties just discussed, we must choose a pair $(r,w)$ that, respectively, defines the downweighting of the $\sqrt{M}$ used for the smaller penalty and the proportion of penalties getting each value. Concretely speaking, in \Cref{prop:whole_sharp} and \Cref{approach all q upper}, we set
\Align{
	r(u)=\Phi^{-1}\left(\frac{2\epsilon-\epsilon^\star-\epsilon u}{2(\epsilon-\epsilon^\star)}\right)/t^\star(\tppmax).
	\label{u and r}
}
where $\epsilon^\star$ and $\tppmax$ define the DT power limit and are given in \eqref{eq:eps_star}-\eqref{eq:DT_power_limit} and $t^\star$ is defined in \eqref{eq:t_star_eq}. Moreover, 
\Align{
	w(u)=\epsilon^\star+\frac{2(1-\epsilon^\star)}{1-r}\left[\Phi(-t^\star(\tppmax))-r\Phi(-rt^\star(\tppmax))-\frac{\phi(-t^\star(\tppmax))-\phi(-rt^\star(\tppmax))}{t^\star(\tppmax)}\right],
	\label{r and w}
}
where $r$ in the above is shorthand for the $r(u)$ from \eqref{u and r}. 

Without going into details, the key reason for choosing such pair $(r,w)$ is so that the sequence of two-level penalties have two different penalization effects: for one, the SLOPE estimator $\h{\pi+Z,A}(\pi+Z)$ is equivalent to a Lasso estimator $\eta_{\text{soft}}(\pi+Z;t^\star(\tppmax))$ in the sense of \eqref{eq:er1}; for the other, the SLOPE estimator is equivalent to a different Lasso estimator $\eta_{\text{soft}}(\pi+Z;rt^\star(\tppmax))$ in the sense of \eqref{eq:er2}.

To be precise, it can be shown that
$$\h{\pi+Z,\A}(\pi+Z)\overset{P}{=}\eta_{\text{soft}}(\pi+Z;t^\star(\tppmax)),$$
and 
\begin{align}
	\E(\h{\pi+Z,\A}(\pi+Z)-\pi)^2=\E(\eta_{\text{soft}}(\pi+Z;t^\star(\tppmax))-\pi)^2,
	\label{eq:er1}
\end{align} 
so when considering the asymptotic magnitude of the elements of the SLOPE estimator, or its asymptotic estimation error \eqref{eq:SE contraint_quote}, we can analyze the limiting scalar function instead using a soft-thresholding function with threshold given by $t^\star(\tppmax)$. Moreover, this implies that SLOPE satisfies the state evolution constraint \eqref{eq:SE contraint_quote} in a similar way to how the Lasso satisfies its corresponding state evolution constraint.

However, analysis of the asymptotic sparsity of the SLOPE estimator or of its asymptotic TPP and FDP, relies on the fact that one can prove
\begin{align}
	\PP(\h{\pi+Z,\A}(\pi+Z)\neq 0)=\PP(\eta_{\text{soft}}(\pi+Z;r t^\star(\tppmax))\neq 0),
	\label{eq:er2}
\end{align} 

Hence, again, instead of analyzing the limiting scalar function one can analyze a soft-thresholding function, but now with a smaller threshold given by $r t^\star(\tppmax)$ for some $0\leq r\leq 1$ defined in \eqref{u and r}. Reducing the threshold in this way functions to improve the attainable TPP--FDP over the comparable Lasso problem by allowing more elements in the estimate with non-zero values.  We visualize the above claims in \Cref{fig:approach proof}(d). 

Essentially, the state evolution condition \eqref{eq:SE contraint_quote} must always hold, but it uses the larger pseudo zero-threshold $t^\star(\tppmax)$,  while inference is conducted on the true, but smaller, zero-threshold $r t^\star(\tppmax)$. In this way, we can extend attainability of $q_\text{Lasso}^\star$ to attainability $q^\star$, while still working within the state evolution constraint \eqref{eq:SE contraint_quote}.

\subsection{Infinity-or-nothing prior has FDP above upper bound}
The goal of this section is to provide some intuition for the \Mobius form of the curve $q^\star(u)$ when $u$ is larger than the DT power limit. This will be done by demonstrating that, in the case of infinity-or-nothing priors, with a special subset of penalties, the SLOPE $\fdpinf$ is always above $q^\star$ in Proposition~\ref{thm:inf-nothing q upper star}. This also motivates the achievability results of Section~\ref{sec:approach Mobius}, as the proof given in Section~\ref{sec:approach Mobius} essentially tries to construct prior penalty pairs such that the inequality in Proposition~\ref{thm:inf-nothing q upper star} becomes an equality. While we only consider infinity-or-nothing priors here, we remark that in the Lasso case  these are actually the \emph{optimal} priors (ses \citet[Section 2.5]{lassopath}), meaning that they achieve the minimum $\fdpinf$ given $\tppinf$.

\begin{proposition}
	\label{thm:inf-nothing q upper star}
		Under the working assumptions, namely (A1), (A2), and (A3), for $\GAM\in\Xi$ in \eqref{eq:valid xi}, assuming that $\bet$ is sampled i.i.d.\ from \eqref{eq:sharp_prior} for any $\epsilon'\in [0,1]$, $M\to\infty$, and that $\blam$ is the order statistics of i.i.d.\ realization of a non-negative $\Lambda$ with $\PP(\Lambda=\max\Lambda)\geq\epsilon\epsilon'$, the following inequality holds with probability tending to one:
		\begin{align*}
			\FDP_\GAM(\bet_M(\epsilon'),\blam) \geq q^\star \left( \TPP_\GAM(\bet_M(\epsilon'),\blam); \delta,\epsilon \right) - c_\GAM(\Pi_M(\epsilon'),\Lambda),
		\end{align*}
		for some positive constant $c_\GAM$ which tends to 0 as $\GAM\to 0$.
\end{proposition}

\begin{proof}[Proof of \Cref{thm:inf-nothing q upper star}]
	As in \Cref{sec:lower bound}, we assume $\pi\geq 0$ without loss of generality since the analysis holds if we replace $\pi$ by $|\pi|$. Consider a \textit{subset of priors}, namely the infinity-or-nothing priors: for some $\epsilon'\in[0,1]$,
	\begin{align}\label{infinity or nothing prior}
		\pi_{\infty}(\epsilon')=
		\begin{cases}
			\infty & \text{w.p. } \epsilon\epsilon',
			\\
			0 &\text{w.p. } 1-\epsilon\epsilon'.
		\end{cases}
	\end{align}
	Although the infinity-or-nothing prior in \eqref{infinity or nothing prior} does not satisfy the assumption (A2) that $\PP(\Pi\neq 0)=\PP(\pi\neq 0)=\epsilon$, this does not affect our discussion\footnote{The infinity-or-nothing prior can be approximated arbitrarily closely by a sequence of priors that satisfy the assumption. For example, let $M\goto\infty$ and consider $\pi_{M}(\epsilon')$ defined in \eqref{eq:sharp_prior}.}.
	
	In fact, as demonstrated by \Cref{lem:inf-nothing eps' upper} below, for infinity-or-nothing priors, the state evolution constraint \eqref{eq:SE contraint_quote} guarantees that $\epsilon'\leq \epsilon^\star/\epsilon$. Since $\epsilon^\star$ is the same for the Lasso and SLOPE, this means that the maximum proportion of $\infty$ signals in the infinity-or-nothing prior is the same for both as well.
	
	\begin{lemma}\label{lem:inf-nothing eps' upper}
		Under assumptions in \Cref{thm:inf-nothing q upper star}, 
		we must have $\epsilon'\in[0,\epsilon^\star/\epsilon]$.
	\end{lemma}

	The proof of \Cref{lem:inf-nothing eps' upper} is given in \Cref{app:proof lemma53}. It turns out that the DT threshold $\epsilon^\star$ plays an important role in understanding the relationship between the sparsity and $\tppinf$.
	Before illustrating this relationship, we introduce the concept of \textit{sparsity}. In a finite dimension, the sparsity of SLOPE estimator is $|\{j:\widehat\beta_j\neq 0\}|$. However, as $\p\goto\infty$, the count of non-zeros will also go to infinity, meaning a quantity like $\lim_p |\{j: \widehat{\beta}_j \neq 0\}|$ is not well-defined. Therefore we introduce the \textit{asymptotic sparsity} of the SLOPE estimator via the distributional characterization in \eqref{eq:beta_dist}, denoting the limit in probability by $\plim$,
	\begin{align}
		\kappa(\Pi,\Lambda)&:=\PP\big(\h{\pi+Z,\A}(\pi+Z)\neq 0\big)=\PP\big(\widehat{\Pi}\neq 0\big)=\plim|\{j:\widehat\beta_j\neq 0\}|/p. 
		\label{eq:chi_def}
	\end{align}
	
	Making use of the DT threshold $\epsilon^\star(\delta)$, we show in \Cref{lem:SLOPEphase} that the sparsity $\kappa(\Pi,\Lambda)$ sets an upper bound on achievable $\tppinf$.
	
	\begin{lemma}\label{lem:SLOPEphase}
		Consider SLOPE based on the pair $(\Pi,\Lambda)$ with $\Pi$ from \eqref{eq:sharp_prior} and set $M\to\infty$. Then with the asymptotic sparsity $0\leq \kappa(\Pi,\Lambda)\leq 1$\footnote{To distinguish from the Lasso, we note that SLOPE can reach $\kappa=1$ and thus gives a dense solution whose TPP is 1.}, we have $\tppinf(\Pi,\Lambda)\leq u^{\star}(\kappa(\Pi,\Lambda);\epsilon,\delta)$
		where 	
		\begin{align}
			\label{ustar}
			u^{\star}(\kappa;\epsilon,\delta):=
			\begin{cases}
				1-\frac{(1-\kappa)(\epsilon-\epsilon^{\star})}{\epsilon(1-\epsilon^{\star})}, & \text{ if }\delta<1 \text{ and } \epsilon>\epsilon^{\star}(\delta), \\
				1, & \text{ otherwise. }
			\end{cases}
		\end{align}
	\end{lemma}
	
	\begin{proof}[Proof of Lemma \ref{lem:SLOPEphase}]
		We will only prove $\tppinf(\Pi,\Lambda)\leq 1-\frac{(1-\kappa)(\epsilon-\epsilon^{\star})}{\epsilon(1-\epsilon^{\star})}$ when $\delta<1 \text{ and } \epsilon>\epsilon^{\star}(\delta)$. 
		We note that the bound on $u^{\star}$ given in \eqref{ustar} when $\delta\geq1 \text{ or } \epsilon \leq \epsilon^{\star}(\delta)$ is trivial since, by definition, $\tppinf(\Pi,\Lambda)\leq 1$.
		
		As $M\to\infty$ in \eqref{eq:sharp_prior}, the prior $\pi$ converges to the infinity-or-nothing priors $\pi_\infty(\epsilon')$ in \eqref{infinity or nothing prior}. In addition, $\pi^\star=\pi_\infty(\epsilon'/\epsilon)$. By the intermediate value theorem, there must exist some $\epsilon'\in [0,1]$ such that 
		\begin{align*}
			\tppinf(\Pi,\Lambda)=\PP(|\pi^{\star}+Z|>\alpha) &=(1-\epsilon^{\prime}) \PP(|Z|>\alpha)+\epsilon^{\prime}\PP(|\infty+Z|>\alpha) \\
			&=2(1-\epsilon^{\prime}) \Phi(-\alpha)+\epsilon^{\prime}.
		\end{align*}
		Here the first equality is given by \eqref{eq:tpp fdp zero-threshold} and $\alpha \equiv \alpha(\Pi,\Lambda)$ is the zero-threshold in \Cref{zero threshold}. The second equality follows from substituting the infinity-or-nothing $\pi^\star$. Therefore, the asymptotic sparsity in \eqref{eq:chi_def} is 
		\[\kappa(\Pi,\Lambda)=\PP(|\pi+Z|>\alpha)=(1-\epsilon)\PP(|Z|>\alpha)+\epsilon \tppinf=2(1-\epsilon\epsilon')\Phi(-\alpha)+\epsilon\epsilon',\]
		where the first equality follows by the definition of the zero-threshold in \Cref{zero threshold}, the second uses that $\tppinf(\Pi,\Lambda)=\PP(|\pi^{\star}+Z|>\alpha) $, and the third is the result from the previous equation. 
		
		Some rearrangement gives
		\begin{align}
			\Phi(-\alpha)=\frac{\kappa(\Pi,\Lambda)-\epsilon \epsilon^{\prime}}{2(1-\epsilon \epsilon^{\prime})},
			\quad \text{	and } \quad \tppinf(\Pi,\Lambda)&=\frac{(1-\epsilon^{\prime})(\kappa(\Pi,\Lambda)-\epsilon \epsilon^{\prime})}{1-\epsilon \epsilon^{\prime}}+\epsilon^{\prime}.
			\label{eq:TPP_equiv}
		\end{align}
		Simple calculus shows that the $\tppinf(\Pi,\Lambda)$ in \eqref{eq:TPP_equiv} is an increasing function of $ \epsilon^{\prime}$. To see this, notice that the derivative is $\frac{(1-\epsilon)(1-\kappa)}{(1-\epsilon\epsilon')^2}\geq 0$. Given that $\epsilon'\leq\epsilon^\star/\epsilon$ by \Cref{lem:inf-nothing eps' upper}, we have	\begin{align*}
			\tppinf(\Pi,\Lambda) 		\leq\frac{(1-\frac{\epsilon^{\star}}{\epsilon})(\kappa(\Pi,\Lambda)-\epsilon \cdot\frac{\epsilon^{\star}}{\epsilon})}{1-\epsilon \cdot \frac{\epsilon^{\star}}{\epsilon}}+\frac{\epsilon^{\star}}{\epsilon} 
			=1-\frac{(1-\kappa)(\epsilon-\epsilon^{\star})}{\epsilon(1-\epsilon^{\star})}. 
		\end{align*}
	\end{proof}
	
	In fact, \Cref{lem:SLOPEphase} is an extension of \citet[Lemma C.2]{lassopath} (restated in \Cref{cor:achievability}(a)), which claims that, in the Lasso case, for all priors including those are not infinity-or-nothing, $\tppinf\leq u^\star(\delta;\epsilon,\delta)$. In particular, we remark that $u^{\star}(\delta;\epsilon,\delta)$ is equivalent to $\tppmax(\delta,\epsilon)$, since any Lasso estimator has an asymptotic sparsity no larger than $\delta$. 
	
	
	As an immediate consequence of \Cref{lem:SLOPEphase}, we can reversely set a lower bound on the sparsity $\kappa(\Pi,\Lambda)$ given $\tppinf(\Pi,\Lambda)$. This is achieved by inverting the mapping in \eqref{ustar} and setting $u^\star=\tppinf$:
	\begin{align}
		\kappa(\Pi,\Lambda)\geq 1-\frac{\epsilon(1-\tppinf(\Pi,\Lambda))(1-\epsilon^{\star})}{\epsilon-\epsilon^{\star}}.
		\label{invTPP}
	\end{align}
	Finally, leveraging the lower bound on the sparsity, we can minimize the $\fdpinf$ by minimizing the sparsity $\kappa(\Pi,\Lambda)$, since by definition
	\begin{equation}
		\fdpinf(\Pi,\Lambda)=1-\frac{\epsilon\cdot \tppinf(\Pi,\Lambda)}{\kappa(\Pi,\Lambda)}.
		\label{eq:fdp_inf}
	\end{equation}
	
	Plugging \eqref{invTPP} into \eqref{eq:fdp_inf}, we finish the proof that $\fdpinf\geq q^\star(\tppinf)$ for the SLOPE when we restrict the priors to be infinity-or-nothing: with $\tppinf=u$,
	\begin{align*}
		\fdpinf(\Pi,\Lambda) \geq q^\star(u;\delta,\epsilon):=1-\frac{\epsilon u}{1-\frac{\epsilon(1-u)(1-\epsilon^{\star})}{\epsilon-\epsilon^{\star}}}
		=\frac{\epsilon u(1-\epsilon) -\epsilon^{\star}(1-\epsilon)}{\epsilon u (1-\epsilon^{\star}) -\epsilon^{\star}(1-\epsilon)}.
	\end{align*}
\end{proof}



\subsection{Gap between upper and lower bounds}
\label{sec:difference upper lower}
Considering \Cref{fig:intro_result}, we observe that the upper and lower boundary curves, $q_\star$ and $q^\star$, can be visually and numerically close to each other, especially when $\tppinf<\tppmax$. One may wonder whether these boundaries actually coincide below the DT power limit. We answer this question in the negative and show analytically that there may exist pairs of $(\tppinf,\fdpinf)$ with the $\fdpinf$ strictly below $q^\star(\tppinf)$ when $\tppinf<\tppmax$. In other words, there are instances where $(\tppinf,\fdpinf)$ points lie between the boundary curves $q_\star$ and $q^\star$.

	\begin{proposition}
		\label{prop:gap exists}
		For some $(\delta,\epsilon)$, there exists $\tppinf<\tppmax(\delta,\epsilon)$ defined in \eqref{eq:DT_power_limit} such that
		$$q_\star(\tppinf)<\fdpinf<q^\star(\tppinf).$$
	\end{proposition}
	In the following, we prove \Cref{prop:gap exists} by constructing a specific problem instance $(\Pi,\Lambda)$ which has $\fdpinf$ falling between the bounds. By showing that the gap between $q^\star(u)$ and $q_\star(u)$ indeed exists, we rigorously demonstrate a gap between $q^\star(u)$ and the unknown SLOPE trade-off $q_\text{\tiny SLOPE}$.

We note that, for the Lasso trade-off at $(u,q^\star(u))$, the zero-threshold $\alpha(\Pi,\lambda)=t^\star(u)$ (defined in \eqref{eq:t_star_eq}) exactly and the state evolution constraint \eqref{eq:SE contraint_quote} is binding, i.e. $E(\Pi,\lambda)=\delta$ (see \citet[Lemma C.4, Lemma C.5]{lassopath}). 

Fixing $\tppinf=u$, our strategy (detailed in \Cref{app:upper lower different}) is to construct $(\pi,\A)$ for SLOPE such that $\alpha(\pi,\A)=t^\star(u)$ as well but the state evolution constraint \eqref{eq:SE contraint_quote} is not binding, i.e. $E(\Pi,\Lambda)<\delta$. If such a construction succeeds, we can use a strictly larger zero-threshold than $t^\star(u)$ that can increase until $E(\Pi,\Lambda) > \delta$. Then, by using a larger zero-threshold, the SLOPE $\fdpinf$ is guaranteed to be strictly smaller than $q^\star(\tppinf)$ by \eqref{eq:tpp fdp zero-threshold}. Thus we will complete the proof that $q_\star(u)<q^\star(u)$ for some $u<\tppmax$.

To construct $(\pi,\A)$ satisfying $\alpha(\pi,\A)=t^\star(u)$ with $E(\Pi,\Lambda)<\delta$, we leverage our empirical observation that the optimal priors $\pi^\star$, in the sense of problem \eqref{eq:variational optimization}, which achieves the lower bound $q_\star$, are oftentimes either infinity-or-nothing or constant. This motivates us to consider constant priors $\pi^\star=t_1$, for some constant $t_1$ (i.e. $p_1=1, t_1=t_2$ in \eqref{eq:dirac delta prior}), and hence
\begin{align*}
	\pi=\begin{cases}
		t_1& \text{ w.p. } \epsilon,
		\\
		0 & \text{ w.p. } 1-\epsilon.
	\end{cases}
\end{align*}
In fact, conditioning on $\alpha(\Pi,\Lambda)=t^\star$ and $\tppinf=u$, the constant $t_1(u)$ is uniquely determined by \eqref{eq:tpp fdp zero-threshold}:
\begin{align*}
	\PP\left(|t_1+Z|>t^\star(u)\right)=u,
\end{align*}
where $Z$ is a standard normal.

Next, we use a common tool in the calculus of variations, known as the Euler-Lagrange equation (detailed in \Cref{app:two point analytic}), to construct an effective penalty function $\AA(z)$ analytically on the interval $[0,\infty)$. The explicit form of $\AA(z)$ is defined in \eqref{eq:A_S form} with $\alpha=t^\star$. We emphasize that the constructed $\AA$ may not be a feasible SLOPE penalty function in the sense that it may violate the constraints in problem \eqref{eq:quadratic programming functional}; however, if $\AA$ is increasing, then the optimal SLOPE effective penalty must be $\AA$, as it is the minimizer of the unconstrained version of problem \eqref{eq:quadratic programming functional} and clearly satisfies the constraints. In the case that $\AA$ is feasible, we compare $E(\Pi,\Lambda)=F_{t^\star(u)}[\AA,p_{t_1}]$ with $\delta$ to determine whether $q^\star(u)>q_\star(u)$.

We now give a concrete example, which is elaborated in \Cref{app:concrete example}. When $\delta=0.3, \epsilon=0.2, \Pi^\star= 4.9006, \tppinf=\tppmax=0.5676$, the maximum Lasso zero-threshold $t^\star(\tppmax)=1.1924$ and the minimum Lasso $\fdpinf=0.6216$. We can construct the SLOPE penalty $\AA$ that has the same zero-threshold and achieves $E(\Pi,\Lambda)=0.2773<\delta$. We can further construct the SLOPE penalty with larger zero-threshold, up to 1.2567, eventually have the SLOPE $\fdpinf=0.5954$, which is much smaller than the minimum Lasso $\fdpinf$. In fact, our method can construct SLOPE penalty that outperforms the Lasso trade-off for any $\tppinf\in (0.5283,1]$, as shown in \Cref{fig:u dagger visualization}.

\section{Discussion}
\label{sec:discussion}

In this paper, we have investigated the possible advantages of employing sorted $\ell_1$ regularization in model selection instead of the usual $\ell_1$ regularization. Focusing on SLOPE, which instantiates sorted $\ell_1$ regularization, our main results are presented by lower and upper bounds on the trade-off between false and true positive rates. On the one hand, the two tight bounds together demonstrate that type I and type II errors cannot both be small simultaneously using the SLOPE method with any regularization sequences, no matter how large the effect sizes are. This is the same situation as the Lasso~\citep{lassopath}, which instantiates $\ell_1$ regularization. More importantly, our results on the other hand highlight several benefits of using sorted $\ell_1$ regularization. First, SLOPE is shown to be capable of achieving arbitrarily high power, thereby breaking the DT power limit. For comparison, the Lasso cannot pass the DT power limit in the supercritical regime, no matter how strong the effect sizes are. Second, moving to the regime below the DT power limit, we provide a problem instance where the SLOPE TPP and FDP trade-off is strictly better than the Lasso. Third, we introduce a comparison theorem which shows that any solution along the Lasso path can be dominated by a certain SLOPE estimate in terms of both the TPP and FDP and the estimation risk. In other words, the flexibility of sorted $\ell_1$ regularization can always improve on the usual $\ell_1$ regularization in the instance-specific setting.

The assumptions underlying the above-mentioned results include the random designs that have independent Gaussian entries and linear sparsity. In the venerable literature on high-dimensional regression, however, a more common sparsity regime is sublinear regimes where $k/p$ tends to zero. As such, it is crucial to keep in mind the distinction in the sparsity regime when interpreting the results in this paper. From a technical viewpoint, our assumptions here enable the use of tools from AMP theory and in particular a very recent technique for tackling non-separable penalties. To obtain the lower bound, moreover, we have introduced several novel elements that might be useful in establishing trade-offs for estimators using other penalties.

In closing, we propose several directions for future research. Perhaps the most pressing question is to obtain the exact optimal trade-off for SLOPE. Regarding this question, a closer look at \Cref{fig:SLOPEtradeoffExamples} and \Cref{fig:lasso_phase} suggests that our lower and upper bounds seem to coincide exactly when the TPP is small. If so, part of the optimal trade-off would already be specified. Having shown the advantage of SLOPE over the Lasso, a question of practical importance is to develop an approach to selecting regularization sequences for SLOPE to realize these benefits. Next, we would welcome extensions of our results to other methods using sorted $\ell_1$ regularization, such as the group SLOPE~\citep{groupslope}. For this purpose, our optimization-based technique for the variational
calculus problems would likely serve as an effective tool. Recognizing that we have made heavy use of the two-level regularization sequences in many of our results, one is tempted to examine the possible benefits of using multi-level sequences for SLOPE \citep{zhang2021efficient}. Finally, a challenging question is to investigate the SLOPE trade-off under correlated design matrices; the recent development by \citet{celentano2020lasso} can be a stepping stone for this highly desirable generalization.

	\section*{Acknowledgments}
	
	
	Weijie Su was supported in part by NSF through CAREER DMS-1847415 and CCF-1934876, an Alfred Sloan Research Fellowship, and the Wharton Dean's Research Fund.
	Cynthia Rush was supported by NSF through CCF-1849883, the Simons Institute for the Theory of Computing, and NTT Research. Jason M.~Klusowski was supported in part by NSF through DMS-2054808 and HDR TRIPODS DATA-INSPIRE DCCF-1934924.

	\bibliographystyle{abbrvnat}
	\bibliography{FDRref.bib}
	
	\clearpage
	
	\appendix
	

\section{When does SLOPE outperform Lasso?}
\label{sec:slope-instance-better}
When studying the SLOPE tradeoff curve, we consider results that hold true for all combinations of signal prior distribution and penalty distribution, $(\Pi, \Lambda)$.  In this section, we will instead look at instances of \emph{fixed} bounded signal prior distributions.

Although the SLOPE trade-off upper bound $q^\star$ is no better than the Lasso one $q^{\star}_{\textnormal{\tiny Lasso}}$ when $\tppinf<u_{\textnormal{DT}}^\star(\delta)$, as has been studied extensively in the previous sections of the paper, it is still possible that for a \emph{fixed} prior distribution $\Pi$, the SLOPE can outperform Lasso using a smart choice of penalty vector. We emphasize that such cases are important, since in the real-world, the ground truth prior of the signal is indeed unknown but fixed. In fact, we will demonstrate that the SLOPE can always outperform Lasso in terms of the TPP, the FDP, the mean squared error (MSE).

\begin{figure}[!htb]
	\centering
	\includegraphics[width=7cm,height=5cm]{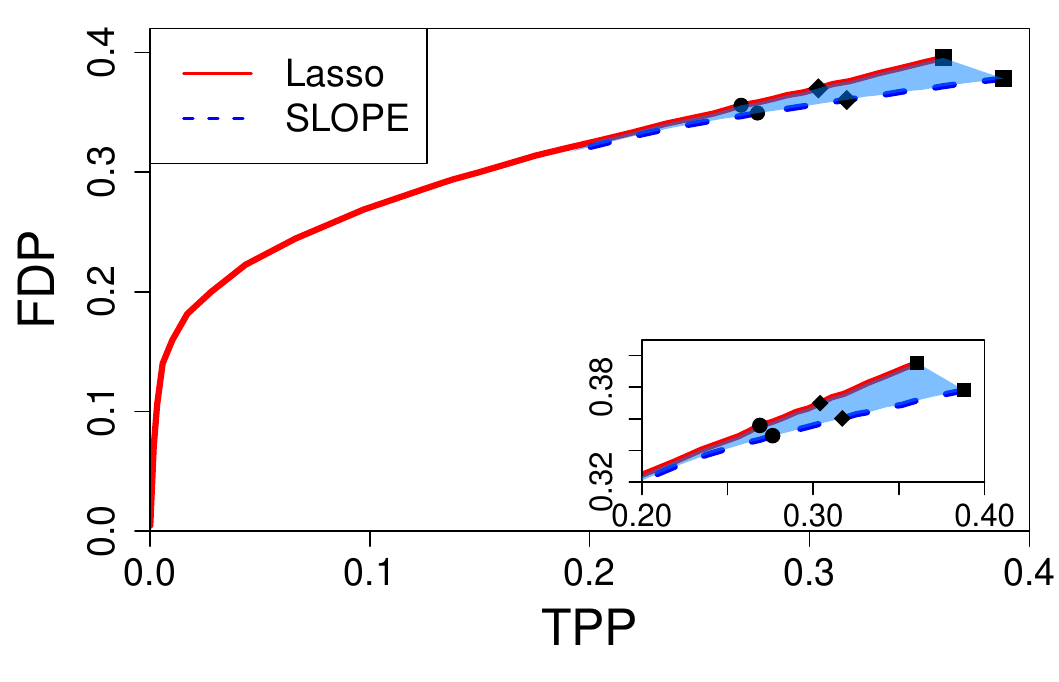}
	\includegraphics[width=7cm,height=5cm]{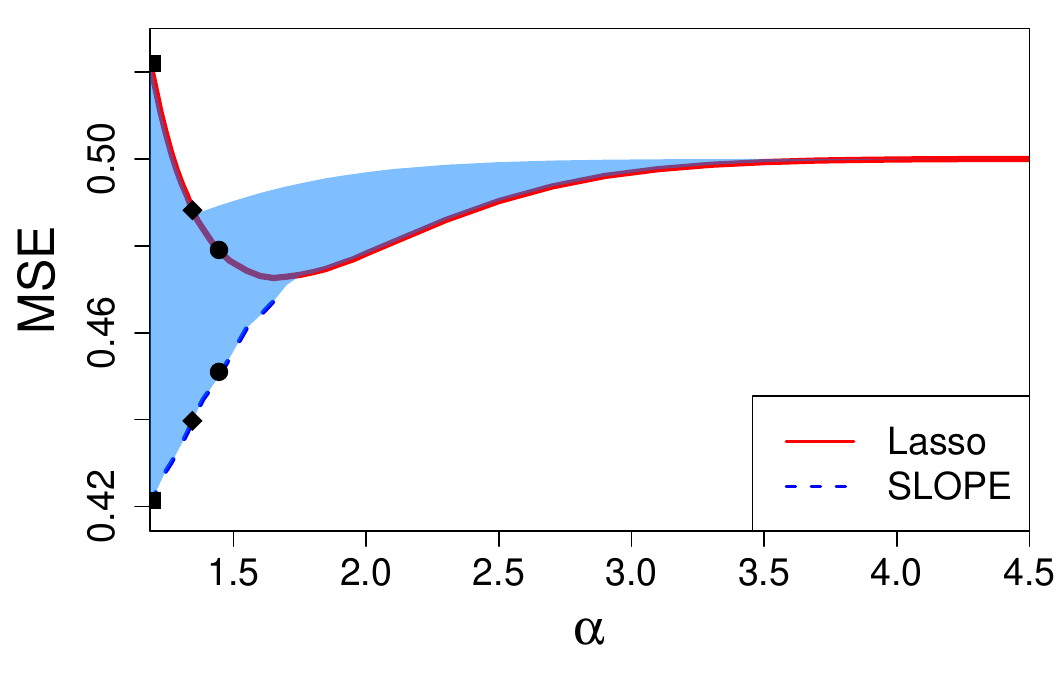}
	\caption{SLOPE outperforms the Lasso below the DT power limit. The red line is the Lasso paths when $\Pi$ is Bernoulli($\epsilon$), $\delta=0.3, \epsilon=0.5,$ and $\sigma=0$. The blue region is the SLOPE $(\tppinf,\fdpinf)$, produced by $\Theta(\ell,\alpha_L,0.1)$ where $\alpha_L\in(\alpha_0,\infty)$ is the zero-threshold shared by the Lasso and the SLOPE for all $\ell\geq\alpha_L$. The blue dashed line is the boundary of blue region. The black dots on the red line are specific $(\tppinf,\fdpinf$, MSE) by the Lasso, while the dots on the blue dashed line correspond to the Lasso dots by shape.}
	\label{fig:instance}
\end{figure}

The proofs we provide only consider the two-level SLOPE penalty sequences of the form $\blam=\bm\theta_{\lambda_1,\lambda_2, w}$ as in \eqref{eq:theta_func}. Despite the simplicity of the penalty sequence, we are already able to leverage the advantages of the flexibility of the SLOPE penalty relative to the Lasso. We moreover believe that the advantages of the SLOPE over the Lasso could be even greater when more general SLOPE penalty sequences are considered, though we leave this to future work.

To be specific, we consider a Lasso pair $(\Pi_L,\Lambda_L)$ and aim to construct a corresponding SLOPE pair $(\Pi_S, \Lambda_S)$ that outperforms the Lasso, under the requirement that $\Pi_L=\Pi_S=\Pi$. We will demonstrate that, for any fixed bounded $\Pi$, each Lasso penalty $\Lambda_L$ can be dominated by some two-level SLOPE penalty distributions  $\Lambda_S$, in the sense that the SLOPE produces strictly better $(\tppinf,\fdpinf)$ and MSE. We further demonstrate a method to search for such dominating SLOPE penalties $\Lambda_S$ and then we reinforce these ideas with simulation results.

The theoretical result of this section can be found in \Cref{thm:instance better}. In specific, we demonstrate that switching from Lasso to the simple two-level SLOPE can achieve better $\tppinf$, better $\fdpinf$ and better MSE at the same time. The full proof is in Appendix \ref{proof_instance_better} and we discuss the ideas of proof here. 

In the following, we work in the normalized $\A$ or $\alpha$ regime (given by the AMP calibration \eqref{eq:cali}; see also the interpretations below that equation) instead of the $\Lambda$ or $\lambda$ regime. The minimum $\alpha\in\R_+$ such that the corresponding the Lasso penalty $\lambda(\alpha)$ is non-negative, is denoted $\alpha_0$. We denote the normalized prior $\pi_L:=\Pi/\tau_L<\Pi/\tau_S:=\pi_S$ and their non-zero conditional distribution as $\pi_L^\star,\pi_S^\star$ respectively. Here $\tau_L,\tau_S$ are computed from the state evolution \eqref{eq:SE} of the Lasso and the SLOPE.

The high-level idea of the proof is, for any Lasso penalty $\A_L$, to find a SLOPE penalty distribution $\A_S$ which has 
\begin{enumerate}[label=(\arabic*)]
	\item the same zero-threshold $\alpha(\pi_S,\A_S)=\alpha(\pi_L,\A_L)$ (defined in \Cref{zero threshold}); 
	\item a smaller $\tau_S$ than the Lasso $\tau_L$ (from \Cref{eq:SE contraint_quote}); 
	\item a larger sparsity $\kappa(\Pi,\A_S)$ than $\kappa(\Pi,\A_L)$ (defined in \Cref{eq:chi_def}).
\end{enumerate}

To see why such $\A_S$ is dominating, we can show (2) together with \citet[Corollary 3.4]{ourAMP}, restated in \eqref{eq:my corollary34}, implies that the SLOPE MSE is strictly smaller than the Lasso MSE.

Further results follow from the definitions of $\TPP$ and $\FDP$: by \eqref{eq:tpp fdp zero-threshold}, we get
$$\tppinf(\Pi,\Lambda_S)=\PP(|\pi_S^\star+Z|>\A_L)>\PP(|\pi_L^\star+Z|>\A_L)=\tppinf(\Pi,\Lambda_L),$$
where we have used the equal zero-threshold condition (1). Finally, we finish the proof for the $\fdpinf$ result by using \eqref{eq:fdp_inf} as well as the sparsity condition (3).

Using the above conditions as the searching criteria, we have designed an algorithm that, for any fixed prior $\Pi$ and for each Lasso penalty $\A_L=\alpha_L$, finds a superior two-level SLOPE penalty $\A_S=\Theta_{\ell,\alpha_L,w}$ by searching over $(\ell,w)$. As presented in Figure \ref{fig:instance}, the SLOPE $(\tppinf,\fdpinf)$ with two-level penalties outperforms the Lasso path.


 
\section{Detailed preliminary results of SLOPE AMP}
\label{app:prelim}
In this section, we introduce the proximal operator of SLOPE, its limiting form (known as the limiting scalar function, on which the SLOPE AMP algorithm is based), and the SLOPE AMP theory relating to the state evolution and calibration equations.

\subsection{SLOPE proximal operator}
We start with the definition of the proximal operator. For input $\bm y \in \mathbb{R}^p$, define the \textbf{proximal operator} of a function $f:\mathbb{R}^p \rightarrow \mathbb{R}$ as
\begin{align*}
	\prox_f(\bm y)&:=\arg\min_{\bm  b \in \mathbb{R}^p}\Big\{\frac{1}{2}\|\bm y - \bm b\|^2+ f(\bm b)\Big\}.
\end{align*}

For SLOPE, the proximal operator uses $f(\bm b)=J_{\thet}(\bm b):= \sum_{i=1}^p \theta_i|b|_{(i)}$ for some penalty vector $\thet$ and as discussed in \cite{SLOPE1}, the SLOPE proximal operator can be computed by \Cref{alg:prox}\footnote{The SLOPE proximal operator can be computed by R library \texttt{SLOPE}.}. 

\begin{algorithm}[!htb]
	\caption{Solving $\prox_J(\bm s;\thet)$ by \cite[Algorithm 3]{SLOPE1}}
	\begin{algorithmic}
		\STATE (1). \textbf{Sorting:} Sort $|\bm s|$ in decreasing order, returning $\text{sort}(|\bm s|)$; 
		\STATE (2). \textbf{Differencing:} Calculate a difference sequence, $\bm S$, defined as
		\begin{center}
			$\bm S =\text{sort}(|\bm s|)-\thet$;
		\end{center}
		\STATE (3). \textbf{Averaging:} Repeatedly average out strictly increasing subsequences in $\bm S$ until none remains. We refer to the decreasing sequence after all the averaging as $\AVE(\bm S)$, and reassign 
		\begin{center} 
			$\bm S=\AVE(\bm S)$;
		\end{center}
		\STATE (4). \textbf{Truncating:} Set negative values in the difference sequence to $0$ and reassign 
		\begin{center}
			$\bm S =\max(\bm S,0)$;
		\end{center}
		\STATE (5). \textbf{Restoring:} Restore the order and the sign of $\bm s$ from step (1) to $\bm S$. Now $\bm S$ with the restored order and sign is the final output.
	\end{algorithmic}
	\label{alg:prox}
\end{algorithm}

For the Lasso, the relevant proximal operator uses $f(\bm b)=\theta\|\bm b\|_1$ and is known as the soft-thresholding function, which we will denote as $\eta_{\text{soft}}: \mathbb{R}^{p} \times \mathbb{R}_+ \rightarrow \mathbb{R}^p$.  Namely, for any index $i \in \{1, 2, \ldots, p\}$, the soft-thresholding function is defined as
\begin{align*}
	[\prox_{\theta\|\cdot\|_1}(\bm y)]_i = [\eta_{\text{soft}}(\bm y;\theta)]_i :=
	\begin{cases}
		y_i-\theta,&\text{ if } y_i>\theta,
		\\
		y_i+\theta,&\text{ if } y_i<-\theta,
		\\
		0,&\text{ otherwise. }
	\end{cases}
\end{align*}

Note that the Lasso proximal operator is indeed \emph{separable}, meaning that any element of its output depends only on the corresponding element of its input. This generally does not hold for the SLOPE proximal operator, which renders the analysis of SLOPE much more difficult. Nevertheless, the SLOPE proximal operator is an \textit{asymptotically separable} function (as discussed in \eqref{eq:yue_limit}) and enables the analysis of the input-dependent penalty, which is detailed in \Cref{bridge}.

In what follows, we denote $\prox_J(\bm v;\blam)$ as the SLOPE proximal operator $\prox_{J_{\blam}}(\bm v)$.

\subsection{SLOPE AMP algorithm}

Under the working assumptions (see $(A1)\sim(A3)$ in \Cref{sec:extend-donoho-tann}) and using the SLOPE proximal operator, the SLOPE optimization problem \eqref{eq:SLOPE_cost} can be solved by the following AMP algorithm with any intial conditions (\cite{ourAMP}):
\begin{align*}
\bet^{t+1}&=\prox_{J}(\X^\top\z^t+\bet^t;\bfalph\tau_t),
\\
\z^{t+1}&=\y-\X\bet^{t+1}+\frac{\z^t}{\delta}\left[ \nabla\prox_{J} (\X^\top\z^t+\bet^t;\bfalph\tau_t)\right],
\end{align*}
where $\nabla$ denotes the divergence and $(\bfalph,\tau_t)$ is defined in the equations known as the state evolution and the calibration, which we will describe shortly. 

It has been shown in \cite[Theorem 2]{ourAMP} that asymptotically $\bet^t$ converges to the true minimizer $\widehat{\bet}$. In addition, for uniformly pseudo-Lipschitz sequence functions $\phi_\p$, we have from \cite[Theorem 3]{ourAMP} that
$$\plim_\p \phi_p(\widehat\bet,\bet)=\lim_t\plim_\p \E_{Z}[\phi_\p(\prox_{J}(\bet+\tau_t \bm{Z};\bfalph(\p)\tau_t),\bet)].$$


Loosely speaking, AMP theory characterizes the SLOPE estimator by
\begin{align*}
	\widehat\bet\overset{\mathcal{D}}{\approx}\prox_{J}(\bet+\tau \bm Z;\bfalph\tau),
\end{align*}
whose empirical distribution weakly converges to $\h{\Pi+\tau Z,\A\tau}(\Pi+\tau Z)$ and we describe $(\A,\tau)$ below.

\subsection{State evolution of SLOPE AMP}
Rigorously speaking, the \textbf{state evolution} for SLOPE is
\begin{align}
\label{SE_formula}
\tau^2&=\sigma^2+\frac{1}{\delta}\E\big(\h{\Pi+\tau Z;\A\tau}(\Pi+\tau Z)-\Pi\big)^2=\sigma^2+\frac{\tau^2}{\delta}\E\big(\h{\pi+Z;\A}(\pi+Z)-\pi\big)^2,
\end{align}
which can be solved iteratively via
\begin{align*}
\tau_{t+1}^2&=F(\tau_t,\A\tau_t)=\sigma^2+\frac{1}{\delta}\E\big(\h{\Pi+\tau_t Z;\A\tau_t}(\Pi+\tau_t Z)-\Pi\big)^2,
\end{align*}

From the algorithmic perspective of AMP, we use the finite approximation of the state evolution,
\begin{align}
\tau^2&=F(\tau,\bm\alpha \tau):=\sigma^2+\frac{1}{\delta}\E\left\langle[\prox_{J}(\bet+\tau \Z;\bm\alpha\tau)-\bet]^2\right\rangle,
\label{finite_SE}
\end{align}
which can be recursively solved from the fixed point recursion
$\tau_{t+1}^2(p)=F(\tau_t(p),\bm\alpha \tau_t(p))$ for each vector $\bfalph\in\R^\p$. Here $\langle \bm u\rangle=\sum_{i=1}^\p u_i/\p$. Furthermore, this state evolution enjoys nice convergence properties: it is shown in \cite[Theorem 1]{ourAMP} that $\{\tau_t\}$ converges monotonically to a unique fixed point $\tau$, under any initial condition.

\subsection{Calibration of SLOPE AMP}
For finite $\p$, we have seen that the state evolution term $\tau$ depends on $\Pi$ and $\bfalph$. Therefore fixing $\Pi$, we can view $\tau(\bfalph)$ as a function of $\bfalph$ and then gives the \textbf{calibration} mapping of the SLOPE penalty $\blam\in\R^\p$ through \cite[Lemma 2.2]{ourAMP},
\begin{align}
\label{calibration}
\blam(\bfalph)&=\bfalph\tau\left(1-\frac{1}{\delta\p}\E( \nabla\prox_{J}(\bm\Pi+\tau \Z;\bfalph\tau))\right),
\end{align}
where the divergence of the proximal operator is defined as
\begin{align*}
\nabla\prox_{J}(\bm v;\bfalph\tau)&:=\text{diag}\left(\frac{\partial}{\partial v_1},\frac{\partial}{\partial v_2},...\right)\cdot \prox_{J}(\bm v;\bfalph\tau).
\end{align*}

There are certain critical observations given by \cite[Lemma 2.1]{ourAMP} and \cite[Proofs of Fact 3.2 and 3.3]{SLOPE2} that explain the divergence term as
\begin{align*}
\nabla\prox_{J}(\bm v;\bfalph\tau) =\|\prox_{J}(\bm v;\bfalph\tau)\|_0^*,
\end{align*}
where $\|\bm{x}\|_0^*$ counts the \textit{unique} non-zero magnitudes in the vector $\bm{x}$. E.g. $\|(2,-1,1,0)\|_0^*=2$. This norm reduces to $\ell_0$ norm in the Lasso case, where all non-zero elements in $\prox_{J}(\bm v;\bfalph\tau)$ have unique magnitudes. Therefore we can express \eqref{calibration} as
\begin{align}
\blam(\bfalph)&=\bfalph\tau\big(1-\frac{1}{n}\E\|\prox_{J}(\bm\Pi+\tau \Z;\bfalph\tau)\|_0^*\big).
\label{finite_calibration}
\end{align}

In the asymptotic case, the calibration lies between the original penalty distribution $\Lambda$ and the normalized penalty $\A$, to which the empirical distributions of $\{\blam\}$ and $\{\bfalph\}$ converge weakly:
\Align{
	\begin{split}
	\Lambda&\overset{\mathcal{D}}{=}\A\tau\left(1-\frac{1}{\delta}\lim_p\frac{1}{p}\E\|\prox_{J}(\bm\Pi+\tau\Z;\bm A(\p)\tau)\|_0^*\right)
	\\
	&=\A\tau\left(1-\frac{1}{\delta}\PP\Big(\h{\pi+Z;\A}(\pi+Z)\in\mathcal{U}(\h{\pi+Z;\A}(\pi+Z))\Big)\right),
	\end{split}
	\label{asym calibration}
}
where $\mathcal{U}(\cdot)$ is defined by
\begin{align*}
	\mathcal{U}(\h{})&:= \left \{h\in \R_+: \PP(|\h{}|=h)= 0 \right\}.
\end{align*}

This quantity represents the portion of the probability space on which $|\h{\pi+Z,\A}(\pi+Z)|$ has zero probability mass. In addition, $\PP(\h{}\in \mathcal{U}(\h{}))$ is the asymptotic proportion of \emph{unique} non-zeros in the SLOPE estimator $\h{}$. For example, if $\h{}$ follows a Bernoulli-Gaussian distribution with 30\% probability being zero, then $\mathcal{U}=(0,\infty)$, since the only point mass is concentrated at $0$ and $\PP(\h{}\in \mathcal{U}(\h{}))=0.7$.



\section{Bridging SLOPE and soft-thresholding}\label{bridge}
In this section we describe a connection between the SLOPE proximal operator and the Lasso proximal operator, i.e. the soft-thresholding function. This connection is built on top of the concept of \textbf{effective penalty} in \Cref{def:essential penalty}, which allows one to \textbf{reduce} the SLOPE proximal operator to the soft-thresholding function with an input-dependent penalty. In analyzing both bounds of the SLOPE trade-off, $q_\star$ and $q^\star$, we use this technique so that we can study the much more amenable soft-thresholding function in place of the SLOPE proximal operator. 

Here we use `$\prox_J(\bm v;\bfalph)$' to denote the SLOPE proximal operator and $\eta_{\textnormal{soft}}$ to denote the soft-thresholding function, both defined in \Cref{sec:prelim}. Note that unlike the soft-thresholding function, the SLOPE proximal operator does not have an explicit formula (nor does its limiting form given by the limiting scalar function $\h{}$), however it can be efficiently computed by \Cref{alg:prox}. Recall that the SLOPE penalty vector $\bfalph$ is decreasing and non-negative. The first result we present in this section says that in finite dimension, we can always design an effective penalty $\widehat\bfalph(\bm v,\bfalph)$, such that applying $\prox_J$ on penalty $\bfalph$ is equivalent to applying elementwise soft-thresholding on $\widehat\bfalph$. 

\begin{fact}\label{fact:sl}
	For any $\bfalph,\bm v\in \R^\p$, there exists $\widehat\bfalph\in\R^\p$ such that
	\begin{align*}
	\prox_J(\bm v;\bfalph)=\eta_{\textnormal{soft}}(\bm v;\widehat\bfalph).
	\end{align*}
\end{fact}
\begin{proof}[Proof of \Cref{fact:sl}]
For $\bm v\geq 0$, the soft-thresholding operator is $\eta_{\textnormal{soft}}(\bm v;\widehat\bfalph)=\max(\bm v-\widehat\bfalph,0)$. Note that $\bm v\geq 0$ implies $\prox_J(\bm v;\bfalph)\geq 0$ and $[\prox_J(\bm v;\bfalph)]_i\leq v_i$ for every $i$. Then we can simply design $\widehat\bfalph$ by setting $\widehat\bfalph=\bm v-\prox_J(\bm v;\bfalph)$. More generally, for any $\bm v$, we can set $\widehat{\bfalph}=|\bm v|-\prox_J(|\bm v|;\bfalph)$ (c.f. \cite[Proposition 2]{SLOPEasymptotic}).
\end{proof}

We notice that there are possibly multiple valid designs of $\widehat\bfalph$. An example would be 
\begin{equation}
\label{eq:prox_example}
\prox_J([6,5,3,2,1];[5,2,2,2,2]) = [2,2,1,0,0] =\eta_{\textnormal{soft}}([6,5,3,2,1];\widehat\bfalph),
\end{equation}
and both $\widehat\bfalph=[4,3,2,2,1]$ or $[4,3,2,2,2]$ give the desired result.

We remark that the asymptotic version of the above fact is established in \cite[Proposition 1 and Algorithm 1]{SLOPEasymptotic}. However, we emphasize that although the construction of $\widehat{\bfalph}$ is trivial once $\prox_J(\bm v;\bfalph)$ is known beforehand, it is difficult to derive $\widehat\bfalph$ in general: $\prox_J(\bm v;\bfalph)$ has no explicit form and its computation is complicated, as can be seen in \Cref{alg:prox}. Nevertheless, certain useful properties of the effective penalty $\widehat\bfalph$ can be extracted.

\begin{fact}\label{fact:lambda meet}
Suppose $\bm v$ is sorted in decreasing absolute values, then $\widehat\bfalph$ agrees with $\bfalph$ at the non-zero entries of $\prox_J(\bm v;\bfalph)$ where the proximal operation takes no averaging.
\end{fact}
\begin{proof}[Proof of \Cref{fact:lambda meet}]
From \Cref{alg:prox}, for each entry of $\bm v$, one may think of $\prox_J$ as either applying a soft-thresholding or applying a soft-thresholding followed by an averaging. 
\end{proof}

In the example given in \eqref{eq:prox_example}, the subsequence $[3,2,1]$ of $\bm v$ experiences the soft-thresholding with respect to the penalty subsequence $[2,2,2]$; on the other hand, the subsequence $[6,5]$ experiences the soft-thresholding with respect to the penalty subsequence $[5,2]$ (resulting in $[1,3]$) then the averaging (resulting in $[2,2]$); this output is equivalent to $[6,5]$ experiencing the soft-thresholding with respect to the effective penalty subsequence $[4,3]$ instead of the actual penalty subsequence $[5,2]$. 

In other words, if $v_i$ is indeed penalized by $\alpha_i$ without averaging, then the effective penalty $\widehat\alpha_i$ agrees with the actual penalty $\alpha_i$.

The above result generally does not hold when $v$ is not sorted in decreasing magnitudes. For instance,
$$\prox_J([3,5,-6];[5,2,2]) = [1,2,-2] =\eta_{\textnormal{soft}}([3,5,-6];[2,3,4]).$$
Nevertheless, we show that larger input (in magnitude) matches with larger penalties.

\begin{fact}\label{fact:hat lam decrease}
	Suppose $\bm v$ is sorted in decreasing absolute values, so is $\widehat\bfalph$. Then larger input will have larger effective penalty.
\end{fact}

\begin{proof}[Proof of \Cref{fact:hat lam decrease}]
For the simplicity of discussion, we assume $\bm v\geq 0$. Then we have $\prox_J(\bm v;\bfalph)=\max(\AVE(\bm v-\bfalph),0)$ where $\AVE(\cdot)$ is the averaging operator in \Cref{alg:prox}. For indices where the averaging does not take place on the sequence $\bm v-\bfalph$, we have $\widehat\alpha_i=\alpha_i$ from \Cref{fact:lambda meet}. Clearly $\widehat\bfalph$ is decreasing on these indices as $\bfalph$ is decreasing by the definition of the sorted $\ell_1$ norm. For indices where the averaging does take place, say the averaged magnitude is $c:=[\prox_J(\bm v;\bfalph)]_I$ for some set of indices $I$, then $\widehat\alpha_i=v_i-[\prox_J(\bm v;\bfalph)]_i=v_i-c$ (by \Cref{fact:sl}), which is decreasing in $i$ since $\bm v_I$ is a decreasing subsequence.
\end{proof}

Now that we have derived some properties of the sequence $\widehat{\bfalph}$ as a whole, we will focus on a particular point of the sequence. Before we move on, we introduce a quantile-related concept.
\begin{definition}
	\label{def:upper quantile}
	For a vector $\bm v\in\R^\p$, we denote the $k$-th largest element in absolute values as $\bm v_{(k)}$. For a distribution $V$, we denote $V_{(\mathtt{k})}$ as the \textbf{upper $\mathtt{k}$-quantile} with $\mathtt{k}\in[0,1]$:
	$$\PP(|V|\geq V_{(\mathtt{k})})=\mathtt{k}.$$
\end{definition}
For example, $V_{(0.25)}$ is the upper quartile of $|V|$; $V_{(0.5)},V_{(0)},V_{(1)}$ are the median, maximum and minimum of $|V|$ respectively. 

We show an asymptotic result that $\bfalph$ and $\widehat\bfalph$ agree at a specific point closely related to the zero-threshold defined in \Cref{zero threshold}.

\begin{proposition}\label{lambda meets}
	Suppose $\bm v(\p),\bfalph(\p),\widehat\bfalph(\p)$ converge weakly to distributions $V,\A,\widehat\A$ respectively, with $V$ being a continuous distribution whose support contains 0. Then 
	$$\widehat\A_{(\kappa)}=\A_{(\kappa)}=|V|_{(\kappa)},$$	
	where $\kappa$ is the asymptotic sparsity of the SLOPE estimator, defined in \eqref{eq:chi_def}. 
\end{proposition}
To see how this asymptotic result relates to the zero-threshold $\alpha(\pi,\A)$ in \Cref{zero threshold}, it is helpful to consider $V=\pi+Z$ (which is continuous even if $\pi$ is discrete), since the SLOPE estimator's distribution is $\widehat{\Pi}=\h{\pi+Z,\A}(\pi+Z)$.

\begin{proof}[Proof of \Cref{lambda meets}]
Then the asymptotic sparsity is 
\begin{align*}
	\kappa:=\plim|\{i: [\prox_J(\bm v;\bfalph)]_i\neq 0\}|/p=\PP\big(\h{\pi+Z,\A}(\pi+Z)\neq 0\big)=\PP\big(|\pi+Z|> \alpha(\pi,\A)\big).
\end{align*}
On the other hand, from the soft-thresholding effect of $\h{V,\A}$, we have 
$$\h{V,\A}(V)=\eta_{\textnormal{soft}}(V;\widehat{\A}),$$
and equivalently 
$$\kappa=\PP\big(|\h{V,\A}(V)|\neq 0\big)=\PP\big(\eta_{\textnormal{soft}}(|V|;\widehat{\A})\neq 0\big)=\PP\big(|V|>\widehat{\A}\big),$$
which indicates
$$|V|_{(\kappa)}=\widehat{\A}_{(\kappa)}=\alpha(\pi,\A).$$

From the proof of \Cref{fact:sl} (also from \cite[Proposition 2]{SLOPEasymptotic}), we know $\widehat{\A}_{(\kappa)}=|V|_{(\kappa)}-\h{V,\A}(|V|_{(\kappa)})$. Together with the above, it holds that $\h{V,\A}(|V|_{(\kappa)})=0$.

Notice that $\h{V,\A}$ is continuous, thus there must exist some interval $[|V|_{(\kappa)},x]$ where $\h{V,\A}$ is not constant (i.e. penalties are not averaged), because $\h{V,\A}(|V|_{(\kappa)})=0$ but $\h{V,\A}(x)>0$. Hence by \Cref{fact:lambda meet}, we obtain $\widehat{\A}_{(\kappa)}=\A_{(\kappa)}$. 
\end{proof}
	
To summarize, we can reduce the non-separable SLOPE proximal operator to some separable soft-thresholding, asymptotically. In this way, we can alternatively study the effective penalty used in the soft-thresholding, instead of the implicit SLOPE proximal operator. We emphasize that \Cref{bridge} is the key to study the SLOPE TPP-FDP trade-off bounds $q_\star$ and $q^\star$ in \Cref{sec:lower bound} and \Cref{sec:mobius}.


\section{SLOPE trade-off and \Mobius upper bound}

In this section we provide some useful results that describe the SLOPE TPP--FDP trade-off curve beyond the Lasso phase transition. In particular, we show that the SLOPE state evolution and calibration constraints can be translated to analogous constraints based on the soft-thresholding function.

\subsection{Using AMP to characterize the asymptotic TPP and FDP}
\label{app:lemma31proof}
In this section, we give a sketch of the proof of \Cref{lem:A1}, which consists of justifying the use of AMP to characterize the FDP and TPP of SLOPE asymptotically.

It has been rigorously proven in \cite[Theorem 3]{ourAMP} that $\frac{1}{n} \sum_{i=1}^n \psi ([\h{\Pi+\tau Z,\A\tau}(\bet+\tau Z)]_i,\beta_i)$ is asymptotically equal in distribution to that of $\frac{1}{n} \sum_{i=1}^n \psi (\widehat\beta_i,\beta_i)$, when $\psi: \mathbb{R}^2 \rightarrow \mathbb{R}$ is a pseudo-Lipschitz continuous function. We would like to use this result to analyze the $\FDP$ and $\TPP$, where from \Cref{lem:A1} we see that

	\begin{equation}
		\FDP_\GAM(\bet, \blam) = \frac{|\{j: |\widehat{\beta}_j|>\GAM, \beta_j = 0\}|}{|\{j: |\widehat{\beta}_j|>\GAM\}|} = \frac{\sum_j  \varphi_{V,\GAM}(\widehat{\beta}_j ,\beta_j )}{\sum_{j} \varphi_{V,\GAM}(\widehat{\beta}_{j}, \beta_{j})+ \sum_{j} \varphi_{T,\GAM}(\widehat{\beta}_{j}, \beta_{j})},
		\label{eq:FDP_lemD1}
	\end{equation}
	and 
	\begin{equation}
		\TPP_\GAM(\bet, \blam) = \frac{|\{j: |\widehat{\beta}_j|>\GAM, \beta_j \neq 0\}|}{|\{j: \beta_j \neq 0\}|} = \frac{\sum_j  \varphi_{T,\GAM}(\widehat{\beta}_j ,\beta_j )}{\sum_j  1(\beta_j \neq 0)},
		\label{eq:TPP_lemD1}
	\end{equation}
	are determined by sums of discontinuous functions, $ \varphi_{V,\GAM}(x, y)=1(|x|>\GAM) 1(y=0) $ and $ \varphi_{T,\GAM}(x, y)=1(|x|>\GAM)1(y \neq 0)$, and not pseudo-Lipschitz functions. Therefore \cite[Theorem 3]{ourAMP} does not apply directly. Nevertheless, we are still able to use the characterization  given by AMP, as is demonstrated in  \Cref{lem:A2}. The proof of \Cref{lem:A2} is an extension of the analogous result for the Lasso case given in  \cite[Lemma A.1]{lassopath}. We notice that the result of \Cref{lem:A1} is just that given in \eqref{eq:lemD1_res}.
	
	\begin{lemma}\label{lem:A2}
		Under the working assumptions, namely (A1), (A2), and (A3), for $\GAM$ such that $\PP(\h{\pi+Z,\A}(\pi+Z)=\GAM)=0$, the SLOPE estimator $\widehat{\bet}(\blam)$ obeys
		\begin{align}
			\frac{ V_\GAM(\blam)}{\p}&:= \sum_j \frac{\varphi_{V,\GAM}(\widehat{\beta}_j,\beta_j)}{\p}  = \frac{  |\{j:|\widehat\beta_j|>\GAM,\beta_j=0\}|}{\p}\overset{P}{\to}\PP(|\h{\pi+Z,\A}(\pi+Z)|>\GAM,\pi=0), \label{eq:Vlam}\\
			\frac{T_\GAM(\blam)}{\p}  &:= \sum_j \frac{\varphi_{T,\GAM}(\widehat{\beta}_j,\beta_j) }{\p} = \frac{|\{j:|\widehat\beta_j|>\GAM,\beta_j\neq 0\}|}{\p}  \overset{P}{\to}\PP(|\h{\pi+Z,\A}(\pi+Z)|>\GAM,\pi\neq 0), \label{eq:Tlam}
		\end{align}
		where $Z$ is a standard normal independent of $\Pi$, $(\tau,\A)$ is the unique solution to the state evolution \eqref{eq:SE} and the calibration \eqref{eq:cali}, and $\pi=\Pi/\tau$. Consequently, we have using the representations in \eqref{eq:FDP_lemD1} and \eqref{eq:TPP_lemD1} and the definitions of $V_\GAM(\blam)$ and $T_\GAM(\blam)$ above, that
		\begin{align}
			\FDP_\GAM=\frac{V_\GAM(\blam)}{V_\GAM(\blam)+T_\GAM(\blam)}\overset{P}{\to}\FDP_\GAM^\infty,
			\qquad \text{ and } \qquad
			\TPP_\GAM=\frac{T_\GAM(\blam)}{\sum_j  1(\beta_j \neq 0)}\overset{P}{\to}\TPP_\GAM^\infty.
			\label{eq:lemD1_res}
		\end{align}
	\end{lemma}
	
	\begin{proof}[Proof of \Cref{lem:A2}]
		The analogous result for when $\widehat{\bet}$ is a Lasso solution is proven rigorously in \cite{bogdan2013supplementary}.  Here we adapt their proof for SLOPE. The high level idea for the proof of \eqref{eq:Vlam} and \eqref{eq:Tlam} is to construct two series of pseudo-Lipschitz continuous functions $$\varphi_{V, \GAM,h}(x, y)=(1-R_h(x)) Q_h(y) \quad  \text{ and } \quad \varphi_{T, \GAM,h}(x, y)=(1-R_h(x)) (1-Q_h(y)),$$ that approach $\varphi_{V,\GAM}, \varphi_{T,\GAM}$ as $h\to 0^+$. Here $Q_h(y)=\max\{1-|y/h|,0\}$ and 
		\begin{align*}
			R_h(x)=\begin{cases}
				0&\text{ if }|x|>\GAM+h
				\\
				\frac{\GAM+h-|x|}{2h}&\text{ if }\GAM-h<|x|<\GAM+h
				\\
				1&\text{ if }|x|<\GAM-h
			\end{cases}.
		\end{align*}
		Since for small $h$,
		$$
		\left|\varphi_{V,\GAM, h}(x, y)-\varphi_{V,\GAM}(x, y)\right| \leq 1(\GAM-h<|x|<\GAM+h)+1(0<|y|<h),
		$$
		for any $c>0$,
		\begin{align*}
			&\mathbb{P}\left(\left|\frac{1}{p} \sum_{i=1}^{p} \varphi_{V,\GAM}\left(\hat{\beta}_{i}, \beta_{i}\right)-\frac{1}{p} \sum_{i=1}^{p} \varphi_{V,\GAM, h}\left(\hat{\beta}_{i}, \beta_{i}\right)\right|>c\right) 
			\\
			\leq &\mathbb{P}\left(\frac{1}{p} \sum_{i=1}^{p} 1\left(\GAM-h<\left|\hat{\beta}_{i}\right|<\GAM+h\right)>\frac{c}{2}\right)+\mathbb{P}\left(\frac{1}{p} \sum_{i=1}^{p} 1\left(0<\left|\beta_{i}\right|<h\right)>\frac{c}{2}\right) 
		\end{align*}
		We will show that both terms on the right hand side converge to zero as $p\to\infty$ and then $h\to 0$. The second term converges to zero by the weak Law of Large Numbers. To deal with the first term, we introduce another pseudo-Lipschitz continuous function
		\begin{align*}
			G_h(x)=
			\begin{cases}
				1& \text{ if }\GAM-h<x<\GAM+h
				\\
				0& \text{ if }x>\GAM+2h \text{ or }x<\GAM-2h
				\\
				\frac{x-(\GAM-2h)}{h}& \text{ if }\GAM-2h<x<\GAM-h
				\\
				\frac{(\GAM+2h)-x}{h}& \text{ if }\GAM+h<x<\GAM+2h
			\end{cases}
		\end{align*}
		which upper bounds the function $1\left(<\GAM-h<|x|<\GAM+h\right)$. Then the AMP theory in \citet[Theorem 3]{ourAMP} gives
		$$\lim_{p\to\infty}\frac{1}{p} \sum_{i=1}^{p} 1\left(\GAM-h<\left|\hat{\beta}_{i}\right|<\GAM+h\right)\leq\mathbb{P}\left(\GAM-h<\left|\hat{\Pi}\right|<\GAM+h\right)\to 0$$
		as $h\to 0$, where $\hat\Pi$ is defined in \eqref{eq:beta_dist}. Hence, one can then argue
		\begin{align*}
			\lim _{p \rightarrow \infty} \frac{1}{p} \sum_{i=1}^{p} \varphi_{V,\GAM}(\hat{\beta}_{i}, \beta_{i})
			&\stackrel{P}{=} \lim _{h \rightarrow 0} \lim _{p \rightarrow \infty} \frac{1}{p} \sum_{i=1}^{p} \varphi_{V,\GAM, h}(\hat{\beta}_{i}, \beta_{i}) 
			\\ 
			&\stackrel{P}{=} \lim _{h \rightarrow 0} \mathbb{E} \varphi_{V,\GAM, h}\left(\h{\Pi+\tau Z,\A \tau}(\Pi+\tau Z), \Pi\right) \\
			&=\mathbb{E} \varphi_{V,\GAM}\left(\eta_{\Pi+\tau Z,\A \tau}(\Pi+\tau Z), \Pi\right),
		\end{align*}
		where the second equality in the above employs the AMP results for the pseudo-Lipschitz continuous function $\varphi_{V,\GAM, h}(\cdot, \cdot)$. The technical aspects of the proof involve making this argument rigorous. The final result follows by noticing that
		\begin{align*}
			\mathbb{E} \varphi_{V,\GAM}\left(\eta_{\Pi+\tau Z,\A \tau}(\Pi+\tau Z), \Pi\right) &= \mathbb{E} \left[1\left(|\eta_{\Pi+\tau Z,\A \tau}(\Pi+\tau Z)|>\GAM\right)  1\left(\Pi = 0\right) \right] \\
			&= \PP\left(|\h{\Pi+\tau Z,\A\tau}(\Pi+\tau Z)|>\GAM,\Pi=0\right) \\
			&=\PP(|\h{\pi+Z,\A}(\pi+Z)|>\GAM,\pi=0).
		\end{align*}
		
		Now leveraging results \eqref{eq:Vlam} and \eqref{eq:Tlam}, that give
		\begin{align*}
			V_\GAM(\blam)/\p&\overset{\PP}{\longrightarrow}\PP\left(|\h{\Pi+\tau Z,\A\tau}(\Pi+\tau Z)|>\GAM,\Pi=0\right)=\PP(|\h{\pi+Z,\A}(\pi+Z)|>\GAM,\pi=0),
			\\
			T_\GAM(\blam)/\p&\overset{\PP}{\longrightarrow}\PP\left(|\h{\Pi+\tau Z,\A\tau}(\Pi+\tau Z)|>\GAM,\Pi\neq 0\right)=\PP(|\h{\pi+Z,\A}(\pi+Z)|>\GAM,\pi\neq 0),
		\end{align*}
		and using that
		\[\PP(|\h{\pi+Z,\A}(\pi+Z)|>\GAM,\pi=0) = (1- \epsilon)\PP(|\h{\pi+Z,\A}(Z)|>\GAM),\]
		and
		\[\PP(|\h{\pi+Z,\A}(\pi+Z)|>\GAM,\pi\neq 0) = \epsilon  \PP(|\h{\pi+Z,\A}(\pi^\star+Z)|>\GAM), \]
		where we recall $\pi^\star$ is the distribution of the non-zero part of $\pi$, we finally obtain
		\begin{align*}
			\FDP_\GAM^\infty(\Pi,\Lambda)&=\plim\frac{V_\GAM(\blam)}{V_\GAM(\blam)+T_\GAM(\blam)} \\
			&=\frac{(1-\epsilon)\PP(|\h{\pi+Z,\A}(Z)|>\GAM)}{(1-\epsilon)\PP(|\h{\pi+Z,\A}(Z)|>\GAM)+\epsilon\PP\left(|\h{\pi+Z,\A}(\pi^\star+Z)|>\GAM\right)}, 
			\\
			\TPP_\GAM^\infty(\Pi,\Lambda)&=\plim\frac{T_\GAM(\blam)}{|\{j:\beta_j\neq 0\}|}
			=\PP\left(|\h{\pi+Z,\A}(\pi^\star+Z)|>\GAM\right).
		\end{align*}
	\end{proof}
	
	The result of \Cref{lem:A1} (and \Cref{lem:A2}) implies that when  studying the trade-off between the FDP and TPP asymptotically, we can work with the explicit and amenable quantities $\PP(\h{\pi+Z,\A}(Z)\neq 0)$ and $\PP(\h{\pi+Z,\A}(\pi^\star+Z)\neq 0)$, by considering $\GAM\to 0$.

\subsection{A better understanding the Donoho-Tanner threshold}
\label{understand DT threshold}
In this section, we introduce an equivalent definition of the DT threshold $\epsilon^\star$, originally defined in \eqref{eq:eps_star}, from a non-parametric viewpoint. This definition is necessary for our analysis of the SLOPE trade-off upper bound $q^\star$ discussed in \Cref{sec:mobius}. 

To specify the threshold $\epsilon^\star$ when $\delta < 1$, we consider the equation
\begin{align}\label{epsilon star}
	2(1 - \epsilon)[(1+x^2)\Phi(-x)-x\phi(x)]+\epsilon(1 + x^2) = \delta
\end{align}
in $x > 0$. Above, $\phi(\cdot)$ and $\Phi(\cdot)$ are the probability density function and cumulative distribution function of the standard normal distribution, respectively. We demonstrate the properties of \eqref{epsilon star} can be found in \Cref{fig:eps star} and \Cref{fig:eps_t}. 

The key point we will use is that this equation has a unique positive root in $x$ if and only if $0 < \epsilon < 1$ takes a certain value $\epsilon^\star(\delta)$ that depends only on $\delta$. This unique root is $x:=t^\star(\tppmax(\delta))$, as given by \citet[Appendix C]{lassopath}. Furthermore, \eqref{epsilon star} has two roots when $\epsilon\leq\epsilon^\star$ and no root otherwise. In fact, \eqref{epsilon star} originates from the state evolution \eqref{eq:SE contraint_quote} for the Lasso when we consider the infinity-or-nothing priors defined in \eqref{infinity or nothing prior}, and it can also be found in \citet[Equation (C.5)]{lassopath}.

In summary, \eqref{epsilon star} gives an equivalent representation of $\epsilon^\star(\delta)$ that we will find useful in the upcoming proofs. Namely, $\epsilon^\star(\delta)$ is the specific value of $0 < \epsilon < 1$ such that \eqref{epsilon star} has a unique root.

\begin{figure}[!htb]
	\includegraphics[width=0.25\linewidth,height=0.18\linewidth]{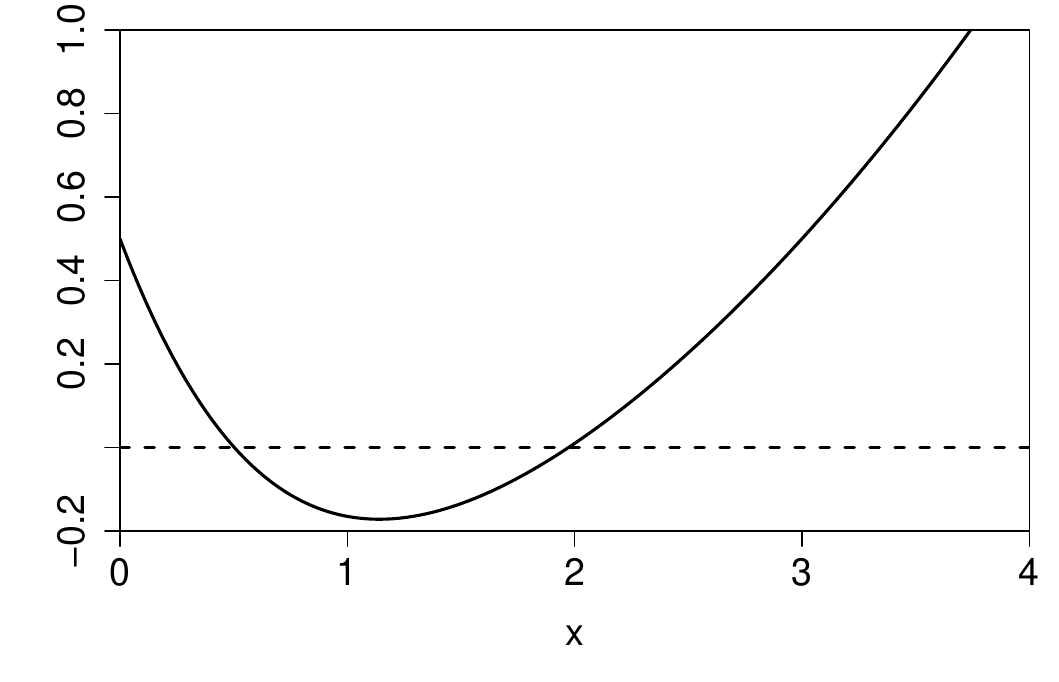}
	\hspace{-0.2cm}
	\includegraphics[width=0.25\linewidth,height=0.18\linewidth]{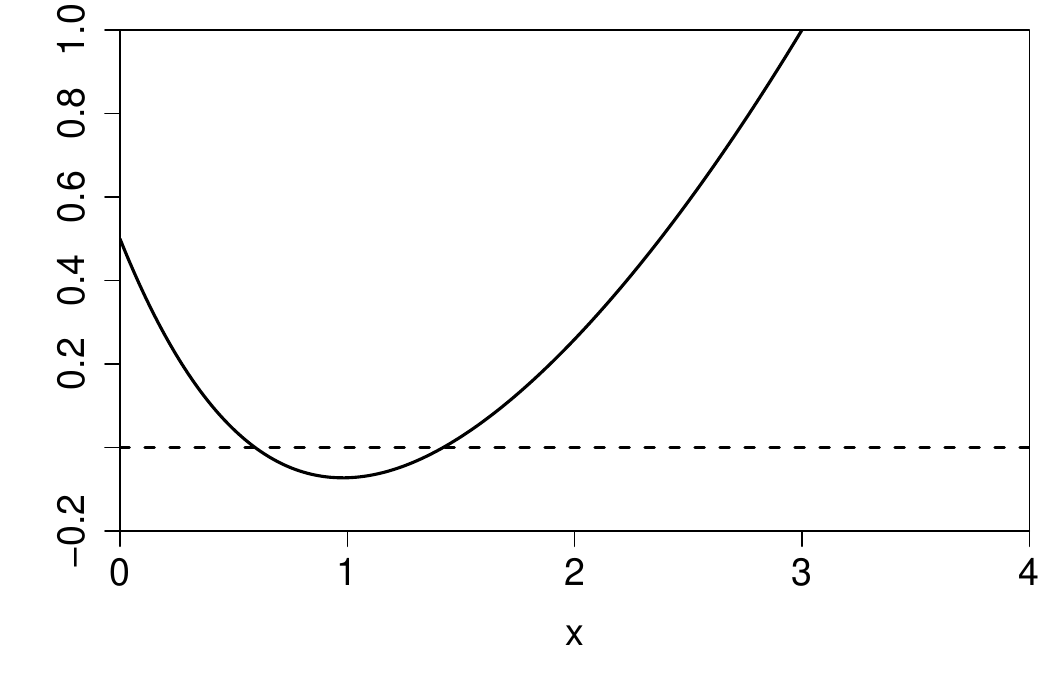}
	\hspace{-0.2cm}
	\includegraphics[width=0.25\linewidth,height=0.18\linewidth]{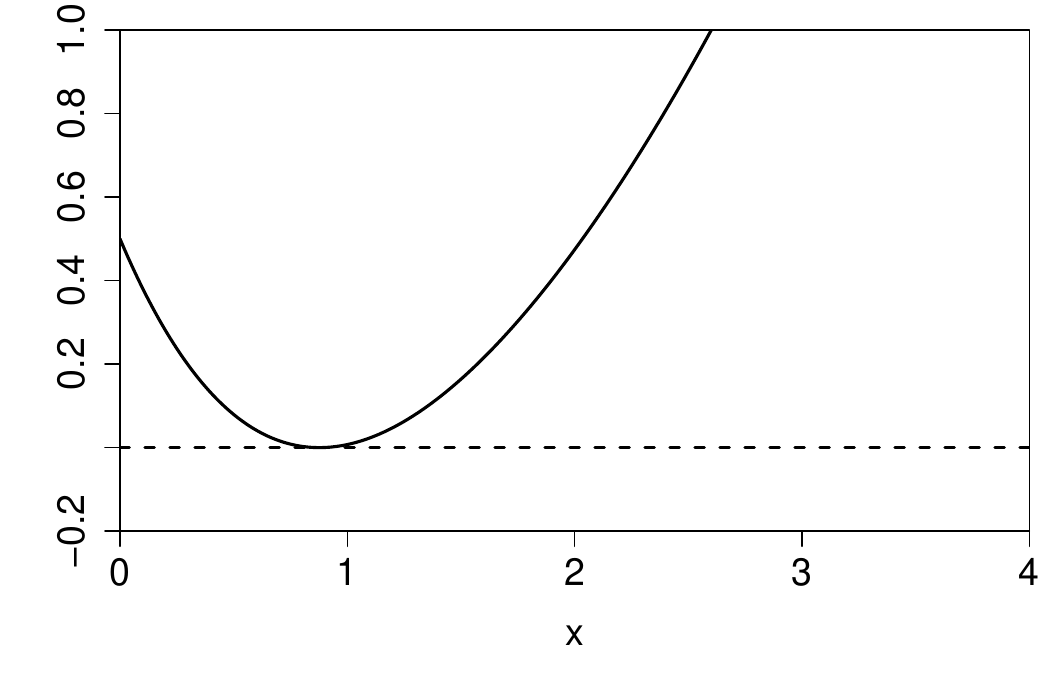}
	\hspace{-0.2cm}
	\includegraphics[width=0.25\linewidth,height=0.18\linewidth]{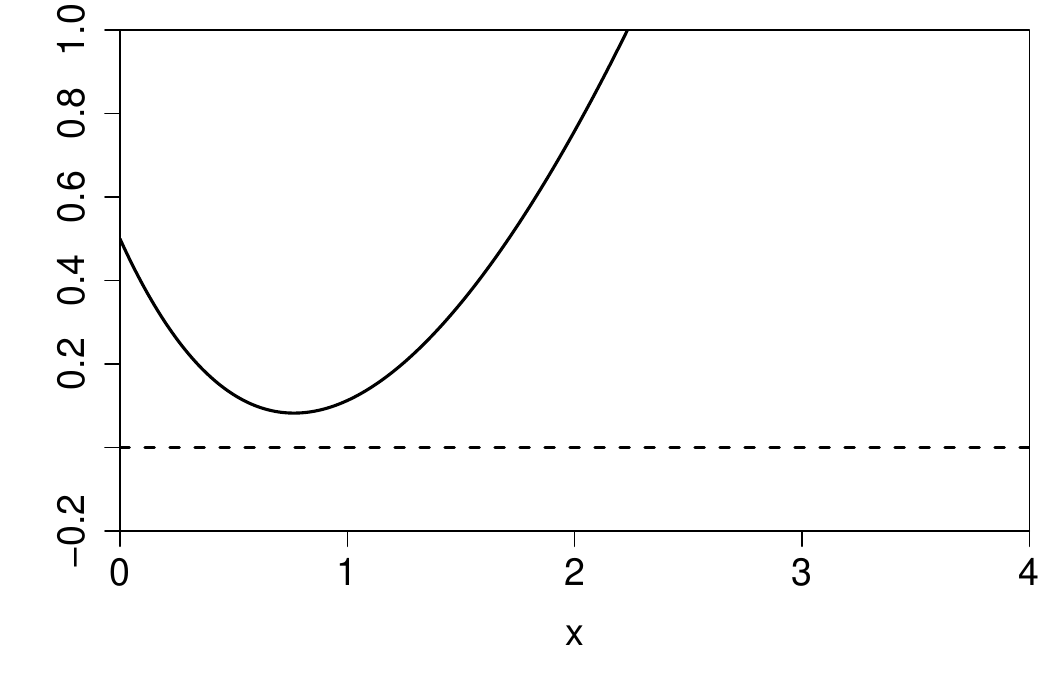}
	\caption{A plot of the function  $f(x) = 2(1 - \epsilon)[(1+x^2)\Phi(-x)-x\phi(x)]+\epsilon(1 + x^2)-\delta$ defined in \eqref{epsilon star} for $\delta=0.5$ and $\epsilon\in\{0.1, 0.15, \epsilon^\star(\delta)=0.1928, 0.25\}$ (from left to right).}
	\label{fig:eps star}
\end{figure}

\begin{figure}[!htb]
	\centering
	\includegraphics[width=8cm,height=5.5cm]{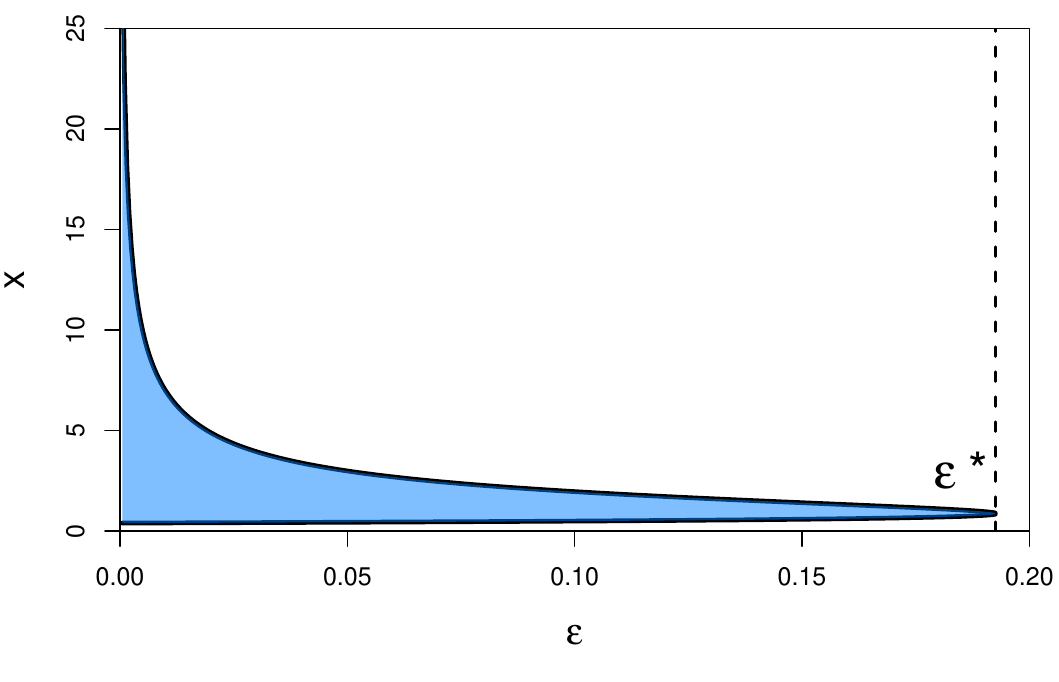}
	\caption{A plot to demonstration the roots of \eqref{epsilon star} with $\delta=0.5$. Here $\epsilon^\star=0.1928$ is the dashed line and the blue area contains valid $x$ for which the inequality in  \eqref{epsilon star} holds for each $\epsilon$. The black solid lines are the roots. Notice that the blue area corresponding to $\epsilon = 0.1$ corresponds to the area under the dashed line in leftmost plot of Figure~\ref{fig:eps star}.}
	\label{fig:eps_t}
\end{figure}


\subsection{Proof of \Cref{lem:inf-nothing eps' upper}}
\label{app:proof lemma53}
\begin{proof}[Proof of \Cref{lem:inf-nothing eps' upper}]
	For infinity-or-nothing priors where $\pi = \infty$ with probability $\epsilon\epsilon'$ or $\pi =0$ with probability $1-\epsilon\epsilon'$, the state evolution constraint \eqref{eq:SE contraint_quote} gives, 
	\begin{equation}
		\begin{split}
			\label{restricted penalty 0}
			\delta & \geq \E\left(\h{\pi+Z,\A}(\pi+Z)-\pi\right)^2\\
			&= \PP(\pi=\infty)\E\left(\h{\pi+Z,\A}(\pi^\star+Z)-\pi^\star\right)^2+\PP(\pi\neq\infty)\E\h{\pi+Z,\A}(Z)^2
			\\
			&=\epsilon\epsilon'\E\left(\h{\pi+Z,\A}(\pi^\star+Z)-\pi^\star\right)^2+(1-\epsilon\epsilon')\E\h{\pi+Z,\A}(Z)^2.
		\end{split}
	\end{equation}
	
	Using the effective penalty function $\widehat\A_\text{eff}$ defined in \Cref{def:essential penalty}, we can write the above as
	\begin{align*}
		\delta\geq\epsilon\epsilon'\E\left(\eta_{\text{soft}}(\pi^\star+Z;\widehat\A_\text{eff}(\pi^\star+Z))-\pi^\star\right)^2+(1-\epsilon\epsilon')\E\eta_{\text{soft}}(Z;\widehat\A_\text{eff}(Z))^2.
	\end{align*}
	
	Now, we denote the distribution $\widehat\A:\overset{\mathcal{D}}{=}\widehat\A_\text{eff}(\pi+Z)$ and in what follows we study the distribution of $\widehat\A_\text{eff}(\pi^\star+Z)$ in more detail. Using the fact that $\pi^\star+Z$ is almost surely larger than $Z$ (since $\pi^\star=\infty$) and \Cref{fact:hat lam decrease}, which states SLOPE assigns larger effective penalty to larger input, we conclude
	$$\widehat\A_\text{eff}(\pi^\star+Z)\overset{\mathcal{D}}{=}\widehat\A\Big|\widehat\A>\widehat\A_{(\epsilon\epsilon')}.$$
	which, we will shortly show, is a constant.	In the above, the quantity with a subscript, $\widehat\A_{(\epsilon\epsilon')}$, is a quantile-related scalar such that $\PP(\widehat\A>\widehat\A_{(\epsilon\epsilon')})=\epsilon\epsilon'$, defined in \Cref{def:upper quantile}. In words, the larger part of $\widehat\A$ $\left(\text{i.e. }\widehat\A\Big|\widehat\A>\widehat\A_{(\epsilon\epsilon')}\right)$ is assigned to the larger part of the input $\pi^\star+Z$; and $\widehat\A\big|\widehat\A\leq\widehat\A_{(\epsilon\epsilon')}$ is assigned to the input $Z$. 
	Furthermore, using the assumption that 
	\begin{equation}
		\label{eq:assumption_used}
		\epsilon \epsilon' \leq \PP(\Lambda = \max \Lambda)  = \PP(\A = \max \A),
	\end{equation}
	where the final equality follows since $\Lambda$ and $\A$ only differ by a constant (see the calibration equation \eqref{eq:cali}), we get
	$$\widehat\A\big|\widehat\A\geq\widehat\A_{(\epsilon\epsilon')}\overset{\mathcal{D}}{=}\A\big|\A\geq\A_{(\epsilon\epsilon')}=\A_{(\epsilon\epsilon')}\in\R.$$
	In the above, the first equality comes from the fact that the upper $\epsilon\epsilon'$ quantile of $\A$ is Lasso-like, following from \eqref{eq:assumption_used} (hence, there is no averaging in the SLOPE proximal operator and \Cref{fact:lambda meet} applies) and the second equality also follows from \eqref{eq:assumption_used} as well.
	
	Therefore, using that $\widehat\A_\text{eff}(\pi^{\star}+Z)$ is a constant equal to $\A_{(\epsilon\epsilon')}$, the state evolution constraint becomes
	\begin{equation}
		\begin{split}
			\delta&\geq\epsilon\epsilon'\E\left(\eta_{\text{soft}}(\pi^\star+Z;\A_{(\epsilon\epsilon')})-\pi^\star\right)^2+(1-\epsilon\epsilon')\E\eta_{\text{soft}}(Z;\widehat\A_\text{eff}(Z))^2
			\\
			&=\epsilon\epsilon'\E(\A_{(\epsilon\epsilon')}-Z)^2+(1-\epsilon\epsilon')\E\eta_{\text{soft}}(Z;\widehat\A_\text{eff}(Z))^2
			\\
			&=\epsilon\epsilon'(1+\A_{(\epsilon\epsilon')}^2)+(1-\epsilon\epsilon')\E\eta_{\text{soft}}(Z;\widehat\A_\text{eff}(Z))^2,
			\label{eq:lemm_delta_bound}
		\end{split}
	\end{equation}
	where the first equality follows by the definition of the soft-thresholding function and the fact that $\A_{(\epsilon\epsilon')}$ is constant and the second from the fact that $Z \sim \mathcal{N}(0,1)$.
	
	Notice that, again by \Cref{fact:hat lam decrease}, $\widehat\A_\text{eff}(z)$ is increasing in absolute value of $z$, hence $$\widehat\A_\text{eff}(Z)\leq\sup\left(\widehat\A\big|\widehat\A\leq\widehat\A_{(\epsilon\epsilon')}\right)=\widehat\A_{(\epsilon\epsilon')}=\A_{(\epsilon\epsilon')},$$
	in which the last equality holds from \Cref{fact:lambda meet}, as a consequence of
	$$\lim_{x\nearrow\epsilon\epsilon'}\left|\h{\pi+Z,\A}(\pi+Z)\right|_{(x)}=\infty>\sup_Z\left|\h{\pi+Z,\A}(Z)\right|=\lim_{x\searrow\epsilon\epsilon'}\left|\h{\pi+Z,\A}(\pi+Z)\right|_{(x)},$$
	i.e. no averaging takes place at the quantile $\epsilon\epsilon'$ (here the limits are one-sided limits). Additionally, we observe that $\eta_{\text{soft}}(z;x)^2$ is decreasing in the scalar $x$. Therefore, we get
	$$\E\eta_{\text{soft}}(Z;\widehat\A_\text{eff}(Z))^2\geq\E\eta_{\textnormal{soft}}(Z;\A_{(\epsilon\epsilon')})^2.
	$$
	
	Applying the above bound into \eqref{eq:lemm_delta_bound} and then using some simple algebra to express the soft-thresholding function, we find
	\begin{equation}
		\begin{split}
			\delta
			&\geq\epsilon\epsilon'(1+\A_{(\epsilon\epsilon')}^2)+(1-\epsilon\epsilon')\E\eta_{\textnormal{soft}}(Z;\A_{(\epsilon\epsilon')})^2
			\\
			&=\epsilon\epsilon'(1+\A_{(\epsilon\epsilon')}^2)+2(1-\epsilon\epsilon')\left[(1+\A_{(\epsilon\epsilon')}^{2}) \Phi(-\A_{(\epsilon\epsilon')})-\A_{(\epsilon\epsilon')} \phi(\A_{(\epsilon\epsilon')})\right].
		\end{split}
		\label{restricted penalty}
	\end{equation}
	
	Following the discussion around \eqref{epsilon star}, the above inequality can only possibly hold when $\epsilon\epsilon'\leq \epsilon^\star$, or when $\epsilon' \in [0, \epsilon^\star/\epsilon]$ as desired.
\end{proof}

\subsection{Achieving the \Mobius curve of $q^\star$}
\label{app:achieving q upper star}
In this section we prove \Cref{prop:whole_sharp}, or in other words, we show that with the special design of a two-level SLOPE penalty and infinity-or-nothing prior, we can approach the \Mobius part of $q^\star$ arbitrarily close.

\begin{proof}[Proof of \Cref{prop:whole_sharp}]
	To give the proof, we consider a specific prior $\Pi_M(\epsilon^\star/\epsilon)$ as in \eqref{eq:sharp_prior} and let $M\to\infty$. Here $\epsilon^\star$ is defined in \eqref{eq:eps_star}. This is equivalent to setting the normalized prior $\pi$ to the infinity-or-nothing prior $\pi_\infty(\epsilon^\star/\epsilon)$, defined in \eqref{infinity or nothing prior} as:
	\begin{equation}
		\begin{split}
			\pi_\infty(\epsilon^\star/\epsilon)=
			\begin{cases}
				\infty & w.p. \quad \epsilon^\star,
				\\
				0 & w.p. \quad 1-\epsilon^\star.
			\end{cases}
			\label{eq:appD_inf_prior}
		\end{split}
	\end{equation}
	
	As for the SLOPE penalty, we consider a sub-class of two-level penalty distributions $\Lambda$ that satisfy $\PP(\Lambda=\max\Lambda)\geq\epsilon^\star$, or in the notation of \eqref{eq:theta_func} we will have $w>\epsilon^\star$. By setting the penalty as such, we satisfy the assumption in \Cref{thm:inf-nothing q upper star} and consequently we can apply the results in \Cref{lem:inf-nothing eps' upper} and \Cref{lem:SLOPEphase}.
	
	Now we are ready to present the proof. For any $\tppinf=u\geq\tppmax(\delta)$, we recall from \eqref{eq:TPP_equiv} in the proof of \Cref{lem:SLOPEphase} that the asymptotic sparsity $\kappa(\Pi,\Lambda)$ (defined in \eqref{eq:chi_def}) satisfies
	\begin{align}
		\label{eq:appD_kappa}
		\kappa(\Pi,\Lambda)= 1-\frac{\epsilon(1-u)(1-\epsilon\epsilon')}{\epsilon-\epsilon\epsilon'}.
	\end{align}
	From \eqref{eq:fdp_inf}, minimizing $\fdpinf(\Pi,\Lambda)$ is equivalent to minimizing $\kappa(\Pi,\Lambda)$, which from \eqref{eq:appD_kappa} we see  is further equivalent to maximizing $\epsilon'$. Since \Cref{lem:inf-nothing eps' upper} states that $\epsilon'\leq\epsilon^\star/\epsilon$, we aim to achieve a sparsity with $\epsilon' = \epsilon^\star/\epsilon$, namely a sparsity of
	\begin{align}
		\kappa(\Pi,\Lambda)=1-\frac{\epsilon(1-u)(1-\epsilon^\star)}{\epsilon-\epsilon^\star},
		\label{eq:kappa smallest}
	\end{align}
	which is given in \eqref{invTPP} as the smallest sparsity for which $\tppinf\geq u$ is possible.
	
	Therefore, we consider a specific prior $\Pi_M(\epsilon^\star/\epsilon)$ as in \eqref{eq:sharp_prior} and let $M\to\infty$. Then the limiting normalized prior $\pi$ is the infinity-or-nothing prior defined in \eqref{eq:appD_inf_prior}.
	Next, we seek the penalty $\Lambda$ that can result in the desired sparsity $\kappa(\pi,\A)$ in \eqref{eq:kappa smallest}, or equivalently, we seek the normalized version of $\Lambda$ given by $\A$, defined via the calibration equation \eqref{eq:cali}. 
	
	To find such a penalty $\Lambda$, we consider the state evolution constraint \eqref{eq:SE contraint_quote}, and emphasize that when achieving the desired sparsity, \eqref{eq:SE contraint_quote} must be satisfied by the pair $(\pi,\A)$. We use the result of \eqref{restricted penalty} and more generally, the proof of \Cref{lem:inf-nothing eps' upper} in \Cref{app:proof lemma53}, to give for $\epsilon' = \epsilon^\star/\epsilon$,
	\begin{align}
		(1-\epsilon^\star)\E\eta_{\textnormal{soft}}(Z;\A_{(\epsilon^\star)})^2+\epsilon^\star(1+\A_{(\epsilon^\star)}^2)
		\leq\E\left(\h{\pi+Z,\A}(\pi+Z)-\pi\right)^2\leq \delta,
		\label{eq:epsilon star smaller delta}
	\end{align}
	where again $\A_{(\epsilon^\star)}$ is a scalar defined in \Cref{def:upper quantile}, i.e., it is chosen such that $\PP(\A>\A_{(\epsilon^\star)})=\epsilon^\star$. In particular, the first inequality above only holds with equality when $\h{\pi+Z,\A}(Z)\overset{\mathcal{D}}{=}\eta_{\textnormal{soft}}(Z;\A_{(\epsilon^\star)})$, which can be seen by comparing the bounds in \eqref{restricted penalty} and \eqref{restricted penalty 0}.
	
	From another direction, by the alternative definition of $\epsilon^\star$ in \eqref{epsilon star}, we have
	\begin{align}
		(1-\epsilon^\star)\E\eta_{\textnormal{soft}}(Z;x)^2+\epsilon^\star(1+x^2)\geq \delta,
		\label{eq:special t u star}
	\end{align}
	for all $x>0$, with the equality holding only when $x=t^\star(\tppmax(\delta))$, as has been discussed in \Cref{understand DT threshold}. We notice that \eqref{eq:special t u star} equals \eqref{epsilon star} since $\E\eta_{\textnormal{soft}}(Z;x)^2 = 2[(1+x^2)\Phi(-x)-x\phi(x)]$. Combining \eqref{eq:epsilon star smaller delta} and \eqref{eq:special t u star}, we obtain
	\begin{align*}	
		\delta&\overset{(a)}{\leq}(1-\epsilon^\star)\E\eta_{\textnormal{soft}}(Z;\A_{(\epsilon^\star)})^2+\epsilon^\star(1+\A_{(\epsilon^\star)}^2)\overset{(b)}{\leq}\E\left(\h{\pi+Z,\A}(\pi+Z)-\pi\right)^2\overset{(c)}{\leq}\delta,
	\end{align*}
	which is only valid when we meet the equality conditions for all the inequalities above, i.e., the penalty distribution $\A$ must be chosen to satisfy
	\begin{align*}
		(a)\quad&t^\star(\tppmax)=\A_{(\epsilon^\star)};
		\\
		(b)\quad&\h{\pi+Z,\A}(Z)\overset{\mathcal{D}}{=}\eta_{\textnormal{soft}}(Z;t^\star(\tppmax)).
	\end{align*}
	Notice that the condition (c) is automatically satisfied when the condition (a) is satisfied.
	
	To design such $\A$, it suffices to set a two-level penalty distribution $\A=\A_{t^\star(\tppmax),r t^\star(\tppmax),w}$ in \eqref{eq:theta_func} for carefully chosen $r(u)$ and $w(u)$, with $w(u)>\epsilon^\star$. Then the condition $w(u) > \epsilon^\star$ gives $\A_{(\epsilon^\star)}=t^\star(\tppmax)$ by design, thus we satisfy the condition (a). In words, the infinite input $\pi^\star+Z$ is assigned to match with the first level of the two-level SLOPE penalty $\A$.
	
	We now turn to the more difficult condition (b) and explicitly choose $r(u)$ and $w(u)$ so that it is satisfied.  Before giving the exact values of $r(u)$ and $w(u)$ and showing how they lead to satisfying condition (b), we take some time to further investigate the sparsity of the SLOPE estimator, $\kappa(\pi,\A)$. Recall that the zero-threshold, defined in \Cref{zero threshold}, is the value $\alpha(\pi,\A)$ such that $\h{\pi + Z, \A}(x) = 0$ if and only if $|x| \leq \alpha(\pi,\A)$ and by \Cref{lambda meets}, we know that the zero-threshold must be equal to one of the two levels of SLOPE penalty.
	
	Now, when $w$ is small, few input values are subjected to the larger level of the penalty and of those inputs, all will correspond to infinite signal prior elements. Thus, the zero-threshold will be the smaller level of the penalty, namely it equals $rt^\star(\tppmax)$ (visualized in \Cref{fig:approach proof}(a)(b)).  In more details, for small $w(u)$, the value $r(u)$ controls the sparsity in the sense that 
	$$\kappa(\pi,\A)=\PP\big(|\pi+Z|>r t^\star(\tppmax)\big)=\epsilon^\star+(1-\epsilon^\star)\PP(|Z|>rt^\star(\tppmax)).$$
	
	Following the above equation, there exists an one-to-one map between $\tppinf=u$ and $r(u)$ to achieve the desired sparsity of \eqref{eq:kappa smallest}:
	\begin{align*}
		1-\frac{\epsilon(1-u)(1-\epsilon^\star)}{\epsilon-\epsilon^\star}  =\kappa(\pi,\A)&=\epsilon^\star+(1-\epsilon^\star)\PP\big(|Z|>r t^\star(\tppmax)\big)
		\\
		&=\epsilon^\star+2(1-\epsilon^\star)\Phi\big(-r t^\star(\tppmax)\big).
	\end{align*}
	Explicitly, by rearranging the above, we conclude that the sparsity condition \eqref{eq:kappa smallest} is satisfied if one sets
	\begin{equation*}
		r(u)=\Phi^{-1}\left(\frac{2\epsilon-\epsilon^\star-\epsilon u}{2(\epsilon-\epsilon^\star)}\right)/t^\star(\tppmax),
	\end{equation*}
	given that $rt^\star(\tppmax)$ is the zero-threshold. In what follows, we always aim to keep the zero-threshold at $rt^\star(\tppmax)$.
	
	As $w$ increases, more and more input values are subjected to the larger level of the penalty. Thus, the zero-threshold and sparsity will remain the same, taking as values the second level of the penalty, $rt^\star(\tppmax)$, and that in \eqref{eq:kappa smallest}, respectively, until $w$ moves above a certain bound and forces the zero-threshold to increase to the larger level of $\A$ (again by \Cref{lambda meets} the zero-threshold can only take these two levels). 	
	\begin{figure}[!htb]
		\subfloat[quantile plot of $|Z|, \A$]{\includegraphics[width=7cm]{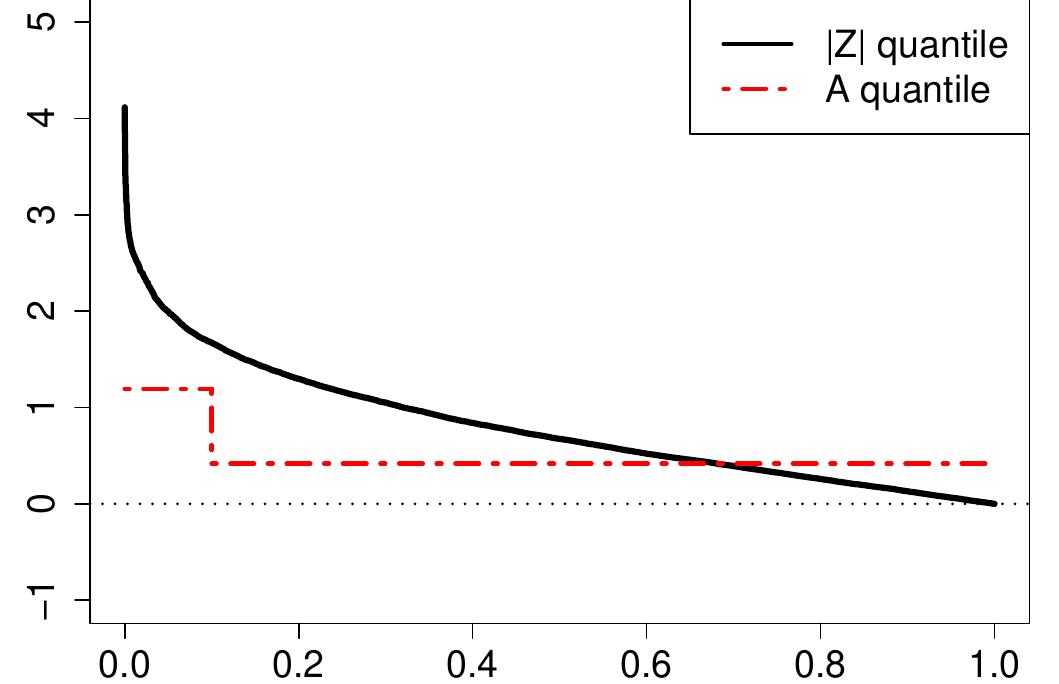}}
		\subfloat[difference of quantiles]{\includegraphics[width=7cm]{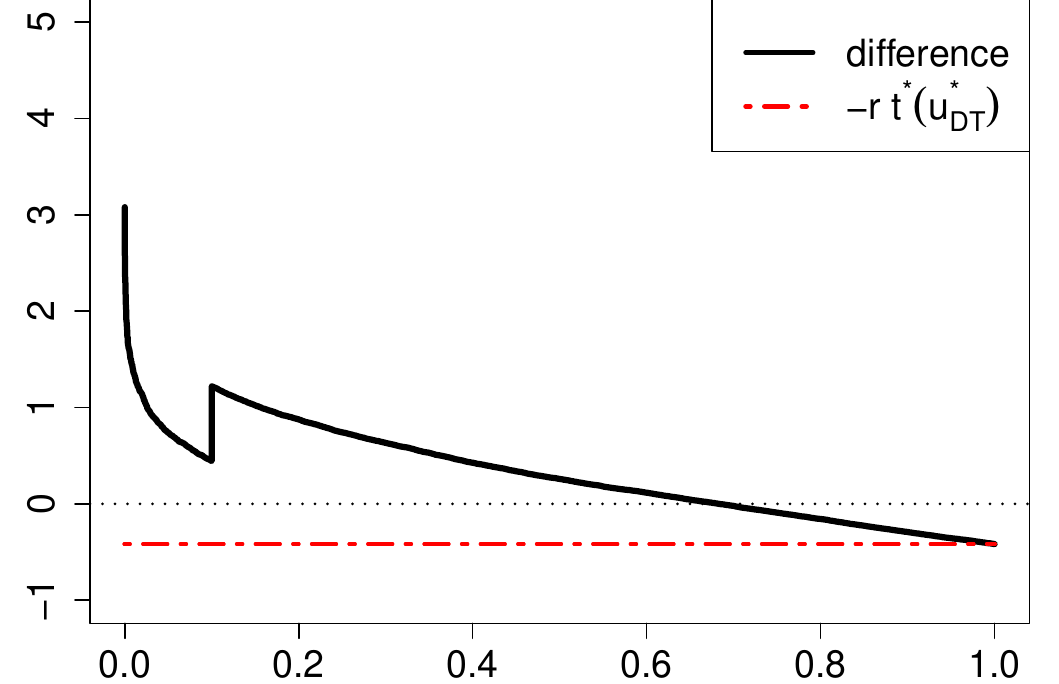}}
		\\
		\subfloat[averaging and truncating]{\includegraphics[width=7cm]{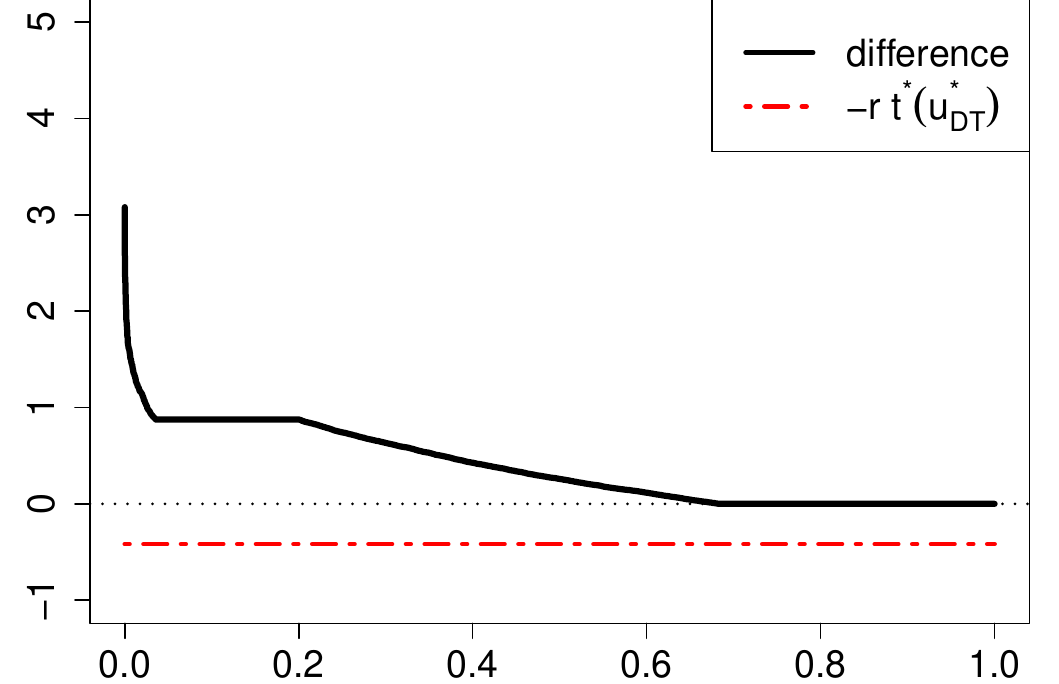}}
		\subfloat[larger $w$]{\includegraphics[width=7cm]{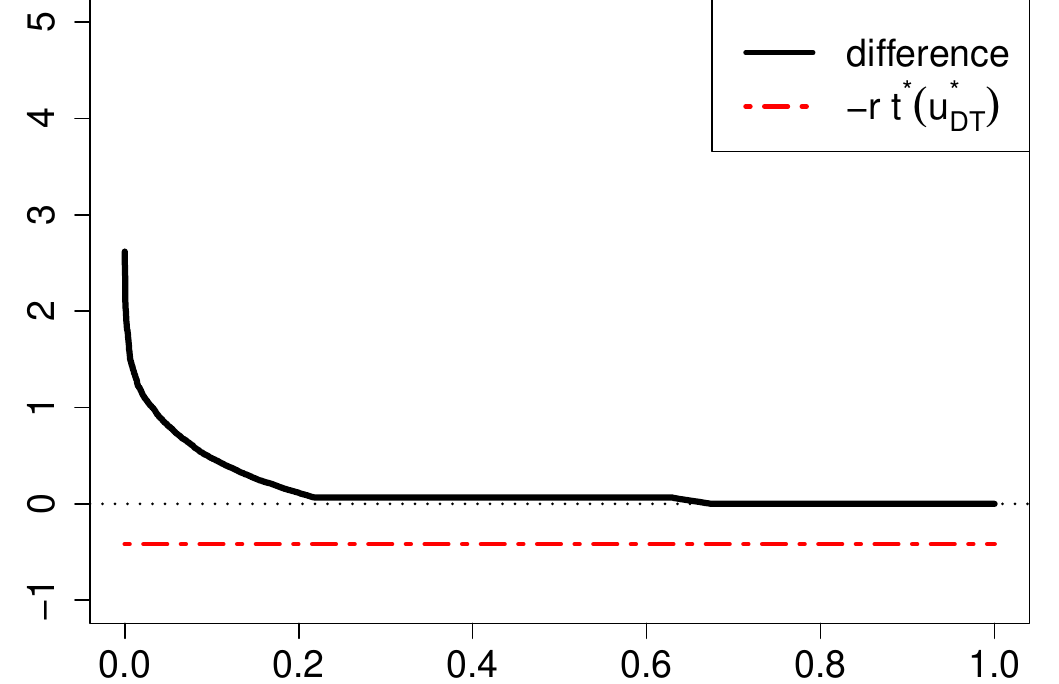}}
		\caption{Fixing $r$ and varying $w$ within some range remains the sparsity of the SLOPE estimator (around 0.6) but forces the SLOPE proximal operator to approach the soft-thresholding. In other words, as $w$ increases, the flat (averaged) quantile of $\mathlarger{\eta}_{ \pi+Z,\A}(\pi+Z)$ takes a magnitude converging to zero.  This implies that $\mathlarger{\eta}_{Z,\A^*}(Z)\to \eta_{\text{soft}}(Z;t^\star(\tppmax))$. Here $\pi=0, \A=\A_{t^\star(\tppmax),r t^\star(\tppmax),w}, \delta=0.3, \epsilon=0.2, \epsilon^\star=0.087, t^\star(\tppmax)=1.1924, u=0.8176, r(u)=0.3500, w(u)=0.4819$.}
		\label{fig:approach proof}
	\end{figure}
	
	Moreover, as $w$ increases to this bound, we observe that $\h{\pi+Z,A}(Z)$ becomes more similar to $\eta_{\text{soft}}(Z;t^\star(\tppmax))$, as demonstrated in \Cref{fig:approach proof}, and hence $\E\h{\pi+Z,A}(Z)^2$ becomes more similar to $\E\eta_{\text{soft}}(Z;t^\star(\tppmax))^2$. 
	To observe this similarity property rigorously, notice that the quantile function of $\pi+Z$ has a sharp drop at $\epsilon^\star$ since $\PP(\pi=\infty)=\epsilon^\star$, which splits the quantiles corresponding to $\infty+Z$ and to $Z$. Accordingly, \Cref{fact:hat lam decrease} says that for the input value corresponding to the `infinity' part of the signal, $\infty+Z$,  since $w>\epsilon^\star$, the SLOPE assigns a penalty given by the upper $\epsilon^\star$ quantiles of $\A$, namely $\left(\A|\A\geq \A_{(\epsilon^\star)}\right)=\A_{(\epsilon^\star)}=t^\star(\tppmax)$  and for the inputs corresponding to the `nothing' part of the signal, $Z$, SLOPE assigns a penalty given by the lower $1-\epsilon^\star$ quantiles of $\A$, denoted by 
	$$\A^*(w):=\left(\A|\A\leq \A_{(\epsilon^\star)}\right)=\A_{t^\star(\tppmax),r t^\star(\tppmax),\frac{w-\epsilon^\star}{1-\epsilon^\star}}.$$
	Notice that what the above says is that a fraction, $\frac{w-\epsilon^\star}{1-\epsilon^\star}$, of the `nothing' signals $Z$ match with the large penalty $t^\star(\tppmax)$ and the remaining fraction match with the smaller penalty $rt^\star(\tppmax)$.
	In this way, when considering just the `nothing' part of the signal, we can write $\h{\pi+Z,\A}(Z)\overset{\mathcal{D}}{=}\h{Z,\A^*(w)}(Z)$\footnote{Notice that, because no averaging takes place at the $wp$-th position, $\prox_{J}(\bm\pi+\bm Z;\bfalph)=\left[\prox_{J}(\bm\pi^\star+\bm Z;(\alpha_1,\cdots,\alpha_{wp})),\prox_{J}(\bm Z;(\alpha_{wp+1},\cdots,\alpha_p))\right]$, in which $[\cdot]$ means concatenation.}.
	
	We now determine the exact $w(u)$ such that $\h{Z,\A^*(w)}(Z)\overset{\mathcal{D}}{=}\eta_{\text{soft}}(Z;t^\star(\tppmax))$, so as to satisfy condition (b).  
	Our strategy is to select a $w(u)$ such that we are able to divide the output of $\h{Z,\A^*(w)}(Z)$ into clearly non-zero, arbitrarily close to zero and zero parts, so as to look like the soft-thresholding function as desired. That we can do this is visualized in Figure~\ref{fig:approach proof} and follows from the fact that with the two-level penalty, there will be only one flat averaged region in the output of $\h{Z,\A^*(w)}(Z)$, which we want to suppress to almost zero. Denoting 
	\begin{equation}
		\label{eq:Pdefs}
		P_1:=\PP(|Z|>t^\star(\tppmax)), \quad  P_2:=\frac{w-\epsilon^\star}{1-\epsilon^\star}, \quad  \text{ and } \quad P_3:=\PP(|Z|>rt^\star(\tppmax)),
	\end{equation}
	we quantitatively define these three parts (clearly non-zero, close to zero, and zero) as the quantiles of $\h{Z,\A^*(w)}(Z)$ on the probability intervals $(0,P_1)$, $(P_1,P_3)$ and $(P_3,1)$ respectively: i.e. we want $w(u)$ such that
	\begin{align*}
		|\h{Z,\A^*(w)}(Z)|\overset{\mathcal{D}}{=}
		\begin{cases}
			\eta_{\textnormal{soft}}(|Z|\big| |Z|>t^\star(\tppmax);t^\star(\tppmax)) &\text{ w.p. $P_1$}
			\\
			0.0001&\text{ w.p. $P_3-P_1$}
			\\
			0&\text{ w.p. $1-P_3$}
		\end{cases}
	\end{align*}
	Here 0.0001 can be an arbitrarily small positive constant, which tends to 0 as $w\nearrow w(u)$. By such a construction, we have met our goal: we have determined $w(u)$ such that $\h{Z,\A^*(w)}(Z)\overset{\mathcal{D}}{=}\eta_{\text{soft}}(Z;t^\star(\tppmax))$. For example, in Figure \ref{fig:approach proof}(d), $P_1\approx 0.23$ and $P_3\approx 0.68$. Given that the averaged sub-interval between $P_1$ and $P_3$ is arbitrarily close to zero, we can write the scaled conditional expectation of $|\h{Z,\A^*(w)}(Z)|$ being on the flat region as an integral of the quantile function:
	\begin{align*}
		&\int_{P_1}^{P_2}(|Z|_{(x)}-t^\star(\tppmax))dx+\int_{P_2}^{P_3}(|Z|_{(x)}-rt^\star(\tppmax)) dx
		\\
		=&\int_{P_1}^{P_3}\h{Z,\A^*(w)}(|Z|_{(x)})dx=0.0001(P_3-P_1).
	\end{align*}
	Setting $w=w(u)$ and thus the right hand side to 0, and rearranging the equation,
	$$\int_{P_1}^{P_3}|Z|_{(x)}dx=t^\star(\tppmax)(P_2-P_1)+rt^\star(\tppmax)(P_3-P_2)=t^\star(\tppmax)[(P_2-P_1)+r(P_3-P_2)],$$
	where the left hand side is the scaled conditional expectation of the random variable $|Z|$ given $rt^\star(\tppmax)<|Z|<t^\star(\tppmax)$, with an explicit form as 
	\begin{align*}
		\int_{P_1}^{P_3}|Z|_{(x)}dx&=\E\left(|Z|\Big|rt^\star(\tppmax)<|Z|<t^\star(\tppmax)\right)\PP\left(rt^\star(\tppmax)<|Z|<t^\star(\tppmax)\right)
		\\
		&=\E\left(Z\Big|rt^\star(\tppmax)<Z<t^\star(\tppmax)\right)\PP\left(rt^\star(\tppmax)<|Z|<t^\star(\tppmax)\right)
		\\
		&=2\phi(rt^\star(\tppmax))-2\phi(t^\star(\tppmax)),
	\end{align*}
	in which the last equality holds from a direct calculation of the expection of a two-sided truncated normal distribution. Hence, we have,
	$$2\phi(rt^\star(\tppmax))-2\phi(t^\star(\tppmax))=t^\star(\tppmax)[(P_2-P_1)+r(P_3-P_2)],$$
	which, upon rearrangement, gives
	$$
	P_2 = \frac{P_1 - r P_3}{1-r} - \frac{2}{(1-r)}\left[\frac{\phi(t^\star(\tppmax)) - \phi(rt^\star(\tppmax))}{t^\star(\tppmax)}\right].
	$$
	Then, plugging in the values in \eqref{eq:Pdefs}, the above simplifies to
	$$
	w(u)=\epsilon^\star+\frac{2(1-\epsilon^\star)}{1-r}\left[\Phi(-t^\star(\tppmax))-r\Phi(-rt^\star(\tppmax))-\frac{\phi(-t^\star(\tppmax))-\phi(-rt^\star(\tppmax))}{t^\star(\tppmax)}\right].
	$$
	
	We claim $w(u)$ can be uniquely determined by $r(u)$, and $w(u)$ is clearly larger than $\epsilon^\star$ as the second term is positive. To see this, we study the term in the bracket and claim that its derivative over $t^\star$ is $(\phi(-t^\star)-\phi(-rt^\star))/(t^\star)^2$, which is negative and hence the term is larger than when $t^\star=\infty$, i.e. 0. 
	
	Combining \eqref{u and r} for designing $r(u)$ and \eqref{r and w} for designing $w(u)$, we can design the two-level SLOPE penalty that, together with infinity-or-nothing prior $\pi_{\infty}(\epsilon^\star/\epsilon)$, approaches $(u,q^\star(u))$ arbitrarily close.
\end{proof}

On a side note, if the $w$ is larger than the specific choice in \eqref{r and w}, i.e. when the flat quantile in \Cref{fig:approach proof} drops below zero, the SLOPE proximal operator has the same effect as soft-thresholding and the analysis for the Lasso case follows. Consequently, $\tppinf$ reduces to the interval $[0,\tppmax)$. Graphically speaking, when one fixes $r$ and increases $w$ from 0 to 1 (similar to \Cref{sharp_pic}), the SLOPE $\tppinf$ will increase from $\tppmax$ to above, until $(\tppinf,\fdpinf)$ touches the \Mobius curve. Then $\tppinf$ will suddenly jump below $\tppmax$, once $w$ is larger than \eqref{r and w}, and then remain constant afterwards.

\subsection{Achievable TPP--FDP region by SLOPE}
The trade-off boundary curves $q_\star$ and $q^\star$ only provide information that splits the entire TPP--FDP region into two parts: the possibly achievable $(\tppinf,\fdpinf)$ and the unachievable ones. See the red and non-red regions in \Cref{fig:lasso_phase}. Although we have shown the achievability of the upper boundary $q^\star$ via \Cref{prop:whole_sharp}, such achievability of the curve does not directly distinguish the achievability of the regions, until the recent work on Lasso by \cite{wang2020complete} which gives a complete Lasso TPP--FDP diagram. 

\begin{figure}[!htbp]
	\centering
	\includegraphics[width=7cm,height=5cm]{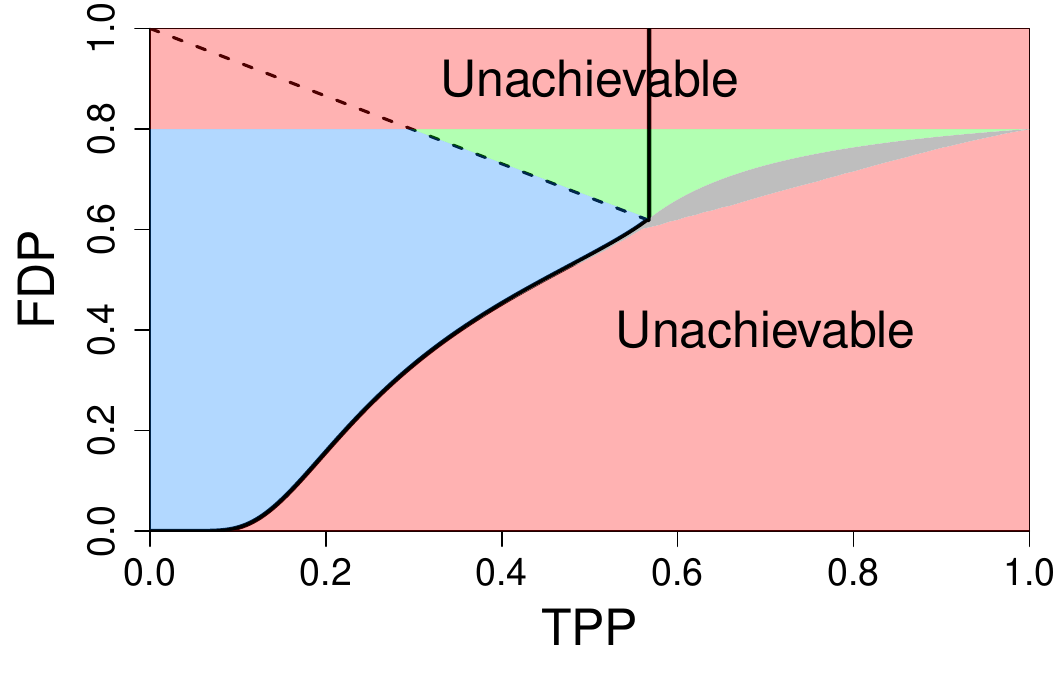}
	\includegraphics[width=7cm,height=5cm]{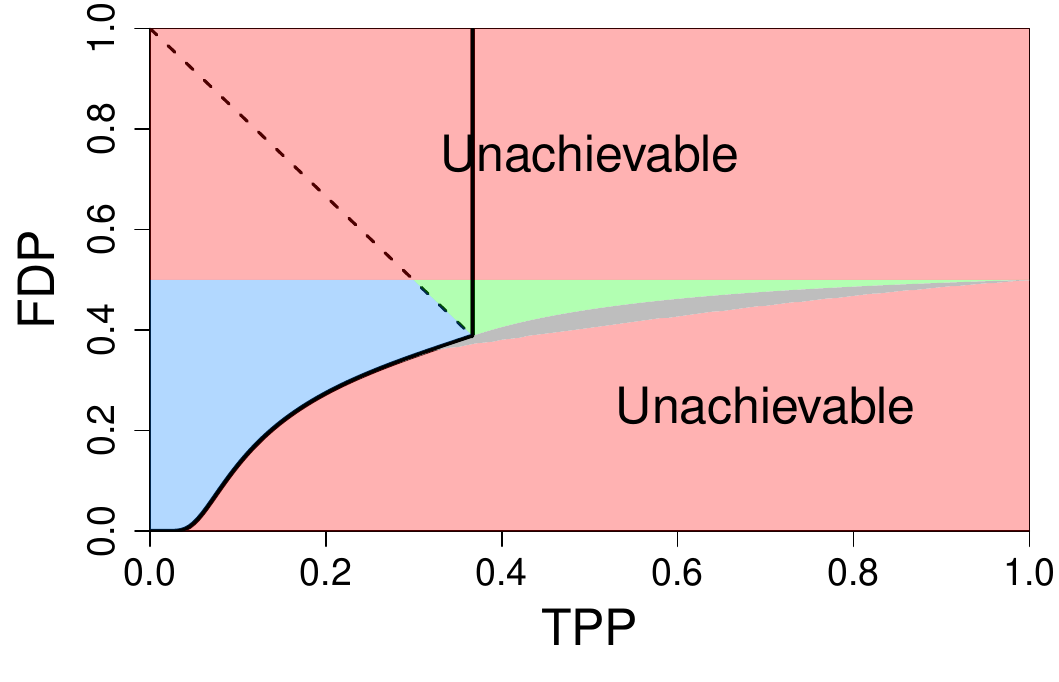}
	\caption{SLOPE TPP--FDP diagram by \Cref{thm:complete SLOPE diagram}. Left: $(\delta,\epsilon)=(0.3,0.2)$. Right: $(\delta,\epsilon)=(0.3,0.5)$. The red regions are $(\tppinf,\fdpinf)$ pairs not achievable by SLOPE nor by the Lasso, regardless of the prior distribution or the penalty tuning. The blue regions are achievable by both the SLOPE and by the Lasso. The green region is achievable only by the SLOPE but not by the Lasso. We note that the boundary between the blue region and the green one is a line segment connection (0,1) and $(u_\textnormal{DT}^\star,q^\star(u_\textnormal{DT}^\star))$, same as given by \cite{wang2020complete} for the Lasso case. The gray region is where the SLOPE trade-off lies in, and thus is possibly achievable by the SLOPE but not by the Lasso.}
	\label{fig:complete SLOPE diagram}
\end{figure}

Here we leverage the homotopy result in \cite[Lemma 3.8]{wang2020complete} to bridge from the achievability of the boundary curve to the achievability of the region. Thus we establish the actually achievable region by SLOPE.

The idea of homotopy is quite intuitive: suppose there are two curves, Curve $ A $ (our upper boundary curve $q^\star$) and Curve $ B $ (the horizontal line $\fdpinf=1-\epsilon$), and a continuous transformation $ f $ moving from Curve $ A $ to $ B$. During the movement, $f$ sweeps out a region whose boundaries include Curve $A$ and $B$, where every single point in this region is passed by the transforming curve during the transformation. Formally, we have a homotopy lemma below.
\begin{lemma}[Lemma 3.7, \cite{wang2020complete}]\label{lem:wanghua}
	If a continuous curve is parameterized by $ f:[0,1] \times[0,1] \rightarrow \mathbb{R}^{2} $ and if the four curves
	\begin{itemize}
		\item$\mathcal{C}_{1}=\{f(u, 0): 0 \leq u \leq 1\}$,
		\item$\mathcal{C}_{2}=\{f(u, 1): 0 \leq u\leq 1\}$,
		\item$\mathcal{C}_{3}=\{f(0, s): 0 \leq s \leq 1\}$,
		\item$\mathcal{C}_{4}=\{f(1, s): 0 \leq s \leq 1\}$,  
	\end{itemize}
	join together as a simple closed curve, $\mathcal{C}:=\mathcal{C}_{1} \cup \mathcal{C}_{2} \cup \mathcal{C}_{3} \cup \mathcal{C}_{4}, $ then $ \mathcal{C} $ encloses an interior area $ \mathcal{D}, $ and $ \forall(x, y) \in \mathcal{D}, \exists(u, s) \in [0,1] \times[0,1] $ such that $ f(u, s)=(x, y) . $ In other words, every point inside the region $\mathcal{D}$ enclosed by curve $ \mathcal{C}$ is realizable by some $ f(u, s)$.
\end{lemma}

Now we can show a region $\mathcal{D}_{\epsilon,\delta}$ defined below is indeed asymptotically achievable. This directly give the SLOPE TPP--FDP diagram in \Cref{fig:complete SLOPE diagram}.

\begin{proposition}\label{thm:complete SLOPE diagram}
	Any $(\tppinf,\fdpinf)$ in $\mathcal{D}_{\epsilon,\delta}$ is asymptotically achievable by the SLOPE. Here $\delta<1, \epsilon>\epsilon^\star(\delta)$ and $\mathcal{D}_{\epsilon,\delta}$ is enclosed by the four curves: $\fdpinf=1-\epsilon, \fdpinf=q^\star(\tppinf)$, $\tppinf=0$ and $\tppinf=1$.
\end{proposition}
\begin{proof}[Proof of \Cref{thm:complete SLOPE diagram}]
	Note that $ (\tppinf,\fdpinf) $ is a function of $ \Lambda, \sigma $ and $ \Pi$ and hence we can denote every TPP--FDP point in $[0,1]\times[0,1]$ as 
	$$ g=\left(\tppinf(\Lambda, \sigma, \Pi), \fdpinf(\Lambda, \sigma, \Pi)\right).$$ 
	To characterize the boundary of the achievable region, i.e. $\mathcal{C}_1$, we parameterize the (two-level) penalty distribution $\Lambda_*(u)$ and the (infinity-or-nothing) prior distribution $\Pi_*(u)$, empowered by the achievability result \Cref{approach all q upper} (which holds for finite noise, including the noiseless case), such that \begin{align*}
		\tppinf(\Pi_*(u),\Lambda_*(u))&=u,
		\\
		\fdpinf(\Pi_*(u),\Lambda_*(u))&=q^\star(u).
	\end{align*}
	
	Leveraging this parameterization, we define the transformation $$f(u,s)=g(\Lambda_*(u),\tan(\frac{\pi s}{2}),\Pi_*(u)),$$ 
	that is employed in \Cref{lem:wanghua}. Therefore, $\mathcal{C}_1$ is the curve described by $q^\star$.
	
	When the noise $\sigma=\infty$, clearly $\fdpinf(\Lambda,\infty,\Pi)=1-\epsilon$. It follows that $\mathcal{C}_2$ is $\fdpinf=1-\epsilon$. When $\tppinf=u=0$, this is the case that the penalty $\Lambda=\infty$ and we get $\mathcal{C}_3$ is $\tppinf=0$. When $u=1$, we have $f(1,s)=(1,1-\epsilon)$. This is the case that the penalty $\Lambda=0$ and $\mathcal{C}_4$ is $\tppinf=1$.
	
	We notice that $\{\mathcal{C}_1,\mathcal{C}_2,\mathcal{C}_3,\mathcal{C}_4\}$ indeed composes a closed curve. Therefore, $f$ sweeps from $\mathcal{C}_1$ to $\mathcal{C}_2$ and each point in $\mathcal{D}_{\epsilon,\delta}$ is achievable by some $f(u,s)$ by the homotopy lemma in \Cref{lem:wanghua}.
\end{proof}


	\section{Lower bound not equal to upper bound}
\label{app:upper lower different}
To complement \Cref{sec:difference upper lower}, we give concrete examples that the upper bound $q^\star$ does not equal the lower bound $q_\star$. Visually, in \Cref{fig:intro_result}, it is not difficult to distinguish the two bounds when $\tppinf\geq\tppmax$. However, when $\tppinf<\tppmax$, the difference can be rather small (see \Cref{fig:t star}), but we assert that, at least for some $\tppinf<\tppmax$, the difference indeed exists and is not a result of numerical errors.

\begin{figure}[!htb]
	\centering
	\includegraphics[width=8cm,height=5.5cm]{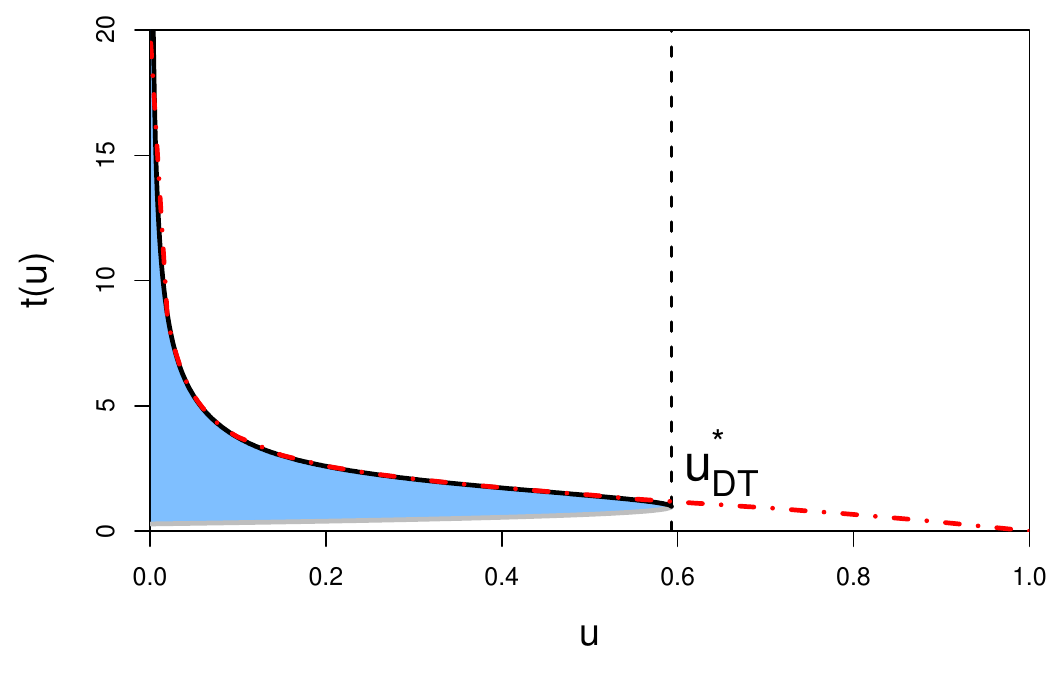}
	\caption{Demonstration of $t^\star(u)$ and $t_\star(u)$ with $\delta=0.3, \epsilon=0.2$. The blue area is valid $x$ if \eqref{eq:t_star_eq} is an inequality with the left hand side being smaller than the right one; the black solid line is the larger root of \eqref{eq:t_star_eq}, i.e. $t^\star(u)$, with the support $[0,\tppmax)$; the gray line is the smaller root. The red dotted line is $t_\star(u)$, with support $[0,1]$. Note that if $t^\star(u)>t_\star(u)$ at some $u$, then $q^\star(u)>q_\star(u)$.}
	\label{fig:t star}
\end{figure}

\subsection{Characterizing the analytic SLOPE penalty}
\label{app:Euler-Lagrange for three point prior}
In order to characterize the optimal SLOPE penalty analytically, we discuss the complementary slackness condition on the monotonicity constraint in problem \eqref{eq:quadratic programming functional}. We start with the case when the monotonicity constraint is \textbf{not binding} (i.e. when $\AA'(z)>0$ for all $z\geq\alpha$). We apply the Euler-Lagrange Multilplier Theorem to derive the following Euler-Lagrange equation on $L(z,\AA)$ defined in \eqref{eq:functional split}:
$$\frac{\partial L}{\partial \AA}-\frac{d}{dz}\frac{\partial L}{\partial \AA'}=0.$$
This is a necessary condition of the optimal SLOPE penalty function. Since $L$ does not explicitly depend on $\AA'$, the Euler-Lagrange equation gives, for the optimal SLOPE penalty function $\AA^*$ of problem \eqref{eq:quadratic programming functional},
\begin{align*}
	\frac{\partial L}{\partial\AA}&=4(1-\epsilon)(\AA^*(z)-z)\phi(z)
	\\
	&+2\epsilon \sum_{j=1,2} p_j\left[\big(\AA^*(z)-(z-t_j)\big)\phi(z-t_j)+\big(\AA^*(z)-(z+t_j)\big)\phi(z+t_j)\right]=0.
\end{align*}

This equation can be significantly simplified: if we denote a function
$$H(z):=4(1-\epsilon)\phi(z)+2\epsilon \sum_{j=1,2} p_j\left[\phi(z-t_j)+\phi(z+t_j)\right],$$
then the Euler-Lagrange equation above claims that
$$H(z)\AA^*(z)+H'(z)=0,$$
which is equivalent to
$$\AA^*(z)=-H'(z)/H(z).$$

On the other hand, when the monotonicity constraint is \textbf{binding} (i.e. when $\AA'(z)=0$ for all $z\geq\alpha$), clearly the penalty function $\AA^*$ is a constant. In short, the optimal penalty function $\AA^*$ coincides with the function $-H'/H$ in the interval $(\alpha,\infty)$ when $\AA^*$ is strictly increasing and stays (piecewise) constant elsewhere; in particular, $\AA^*(z)=\alpha$ on $[0,\alpha]$. 

Unfortunately, the conclusion so far only gives the necessary but not sufficient condition for any penalty function to be optimal. Putting differently, the condition is not specific enough to uniquely determine $\AA^*$ and thus we have to rely on the numerical approach to find $\AA^*$, as shown in \Cref{sec:lower bound}. Nevertheless, the condition we derived above will serve as an essential tool to build up analytic SLOPE penalty in the following sections.

\subsection{Analytic SLOPE penalty for two-point prior}
\label{app:two point analytic}
Here we derive the optimal SLOPE penalty analytically for a special two-point prior, which can be used to prove $q_\star(u)<q^\star(u)$ for some $u$, including those below the DT power limit. We review what is known for the Lasso trade-off: fixing $\delta, \epsilon$ and $\tppinf=u$, the maximum Lasso zero-threshold $t^\star(u)$ satisfies \eqref{eq:t_star_eq} and the minimum Lasso $\fdpinf$ is achieved at such threshold (see \eqref{cor:whole tradeoff}). If for SLOPE we can find a larger zero-threshold $\alpha$ than $t^\star(u)$, then by the definition in \eqref{eq:tpp fdp zero-threshold}:
$$\fdpinf(\Pi,\Lambda)
=\dfrac{2(1-\epsilon)\Phi(-\alpha)}{2(1-\epsilon)\Phi(-\alpha)+\epsilon u}$$ 
for the SLOPE must be smaller than the minimum Lasso $\fdpinf$.

We first determine the prior that we want to study. We focus on a zero-or-constant prior
\begin{align*}
	\pi=\begin{cases}
		t_1& \text{ w.p. } \epsilon
		\\
		0 & \text{ w.p. } 1-\epsilon,
	\end{cases}
\end{align*}
whose probability density function is $p_{\pi}(t)=(1-\epsilon)\delta(t)+\epsilon\delta(t-t_1)$ and clearly $\pi^\star=t_1$. For the Lasso, from \eqref{eq:tpp fdp zero-threshold}, we see that $\tppinf=u$ defines a unique $t_1$ by
$$\PP(|t_1+Z|>t^\star(u))=\Phi(t_1-t^\star)+\Phi(-t_1-t^\star)=u,$$
i.e. $t_1(u,\delta,\epsilon)$ only depends on $u, \delta$ and $\epsilon$. 

Now that we have determined the prior, we seek a feasible SLOPE penalty function $\A_S$ which allows a larger zero-threshold $\alpha$ with this prior: let
\begin{align}
	\A_S(z)=
	\begin{cases}
		\alpha & \text{ if } |z|<\alpha
		\\
		\max(\alpha,-H_{t_1}'(z)/H_{t_1}(z)) & \text{ if } |z|\geq\alpha,
	\end{cases}
	\label{eq:A_S form}
\end{align}
where $H_{t_1}(z)$ is the function $H(z)$ in the \Cref{app:Euler-Lagrange for three point prior} but specific to our new prior, i.e. $p_1=1, p_2=0, t_1=t_2$ in \eqref{eq:dirac delta prior}: we get
\begin{align*}
	H_{t_1}(z)=4(1-\epsilon)\phi(z)+2\epsilon[\phi(z-{t_1})+\phi(z+{t_1})].
\end{align*}

We remark that the SLOPE penalty $\A_S$ is clearly feasible for problems \eqref{eq:variational optimization} and \eqref{eq:quadratic programming functional} if it is monotonically increasing. Furthermore, this monotonicity condition indeed holds true for some ${t_1}$ and $\alpha$ (such that $\A_S$ is increasing in $z$; we will give examples shortly), for which we can show $q_\star(u)<q^\star(u)$.

In summary, fixing $(u,\delta,\epsilon)$, we can uniquely determine ${t_1}\in\R$ for the two-point zero-or-constant prior and the maximum Lasso penalty $t^\star\in\R_+$. Looking at ${t_1}$, we can construct $\A_S$ using \eqref{eq:A_S form} on the interval $(\alpha=t^\star,\infty)$. If furthermore $\A_S$ is increasing, then this non-constant penalty $\A_S$ is feasible and must outperform the constant penalty of the Lasso (which is $t^\star$), based on the Euler-Lagrange equation discussed in \Cref{app:Euler-Lagrange for three point prior}. In consequence, the SLOPE allows strictly larger zero-threshold $\alpha$ than the maximum Lasso zero-threshold $t^\star$, until for some $\alpha$ we saturate the state evolution condition \eqref{eq:SE contraint_quote} by having $F_\alpha[\A_S,p_{\pi^\star}]=\delta$.

We give an example as follows for the framework described above.

\subsection{An example of SLOPE FDP below the Lasso trade-off}\label{app:concrete example}
As a concrete example of SLOPE $\fdpinf$ being smaller than the minimum Lasso $\fdpinf$, i.e. $q^\star_\text{Lasso}$, we use $\delta=0.3, \epsilon=0.2, \sigma=1, u=\tppmax(\delta,\epsilon)= 0.56760$ by \eqref{eq:DT_power_limit}. Then the maximum Lasso zero-threshold $t^\star(u)$ (or equivalently the Lasso penalty $\A_L$) equals $1.19241$ by \eqref{eq:t_star_eq}. In this case, the Lasso $\fdpinf=0.62160$ by \eqref{eq:tpp fdp zero-threshold}. We can compute $t_1(u,\delta,\epsilon)= 1.34864$ by \eqref{eq:tpp fdp zero-threshold}.

One can check that the function $-H_{t_1}'/H_{t_1}$ as well as the penalty function $\A_S$ in \eqref{eq:A_S form} (with $\alpha$ set as $t^\star$) are indeed increasing. Hence $\A_S$ is the unique optimal SLOPE penalty that satisfies the Euler-Lagrange equation. We can analytically compute the state evolution condition in problem \eqref{eq:variational optimization} (see also \eqref{eq:F_alpha definition} for the formula):
\begin{align*}
	&F_{t^\star}[\A_S,p_{\pi^\star}]=2(1-\epsilon)\int_{t^\star}^{\infty}(z-\A_S(z))^2\phi(z)dz+\epsilon\Big[t^2(\Phi(t^\star-{t_1})-\Phi(-t^\star-{t_1}))
	\\
	&+\int_{t^\star}^{\infty}\left(\big(z-{t_1}-\A_S(z)\big)^2\phi(z-{t_1})+\big(-z-{t_1}+\A_S(z)\big)^2\phi(-z-{t_1})\right)dz\Big]
	\\
	=&\int_{t^\star}^{\infty}\left[\frac{1}{2}H_t(z)\A_S(z)^2+H_t'(z)\A_S(z)\right]dz+\epsilon {t_1}^2(\Phi(t^\star-{t_1})-\Phi(-t^\star-{t_1}))
	\\
	&+2(1-\epsilon)\int_{t^\star}^\infty z^2\phi(z)dz+\epsilon\int_{t^\star}^\infty (z-{t_1})^2\phi(z-{t_1})dz+\epsilon\int_{t^\star}^\infty (z+{t_1})^2\phi(z+{t_1})dz.
\end{align*}

Using the facts that $\ddot\phi(z)=(z^2-1)\phi(z)$ and $\dot\phi(z)=-z\phi(z)$, we get 
$\frac{d}{dz}\left(\Phi(z)-z\phi(z)\right)=z^2\phi(z),$
which can be used to simplify the last three integrals:
\begin{align*}
	&F_{t^\star}[\A_S,p_{\pi^\star}]=\int_{t^\star}^{\infty}\left[\frac{1}{2}H_t(z)\A_S(z)^2+H_t'(z)\A_S(z)\right]dz
	\\
	&+\epsilon {t_1}^2(\Phi(t^\star-{t_1})-\Phi(-t^\star-t))
	+2(1-\epsilon)\left[t^\star\phi(t^\star)+\Phi(-t^\star)\right]
	\\
	&+\epsilon\left[(t^\star-{t_1})\phi(t^\star-{t_1})+\Phi(-t^\star+{t_1})\right]
	+\epsilon\left[(t^\star+{t_1})\phi(t^\star+{t_1})+\Phi(-t^\star-{t_1})\right].
\end{align*}

Together with $\A_S(z)=\max\big(t^\star,-H_{t_1}'(z)/H_{t_1}(z)\big)$ on $(t^\star,\infty)$, this analytic quantity can be calculated by numerical integration to arbitrary precision, and it gives $E(\pi,\A_S)=0.27727<\delta$. In words, at the Lasso maximum zero-threshold $t^\star$, the SLOPE and the Lasso have the same $\fdpinf$, but the SLOPE has a smaller normalized estimation error $E$ in the state evolution condition \eqref{eq:SE contraint_quote}. Hence, this leaves a margin to further reduce the $\fdpinf$ before we use up the margin. 

\begin{figure}[!htb]
	\centering
	\includegraphics[width=7cm,height=5cm]{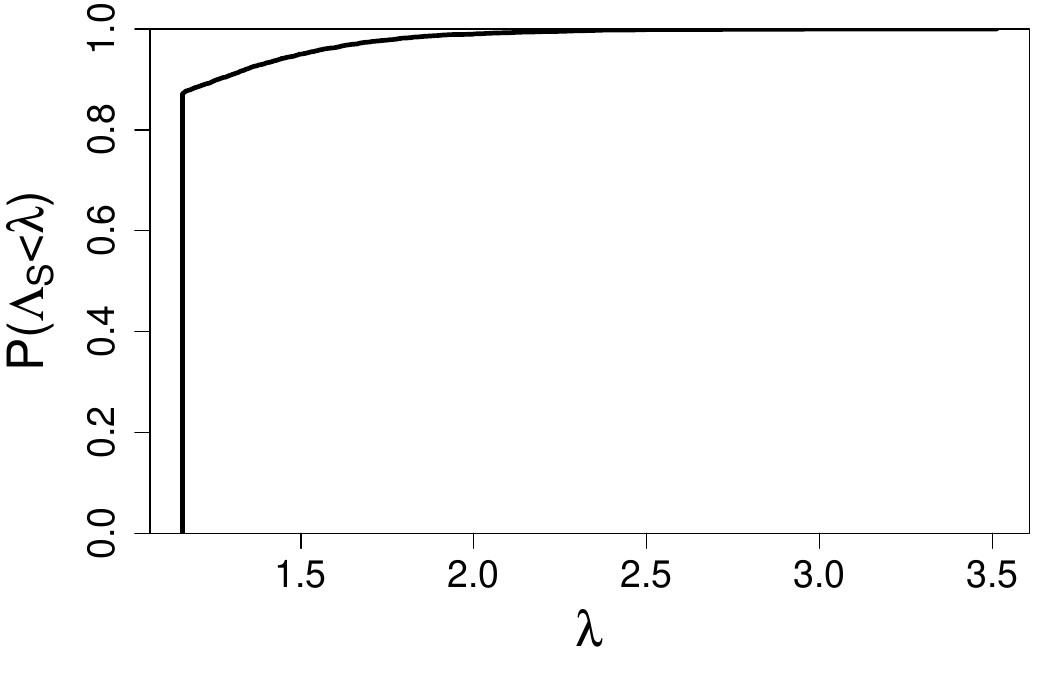}
	\includegraphics[width=7cm,height=5cm]{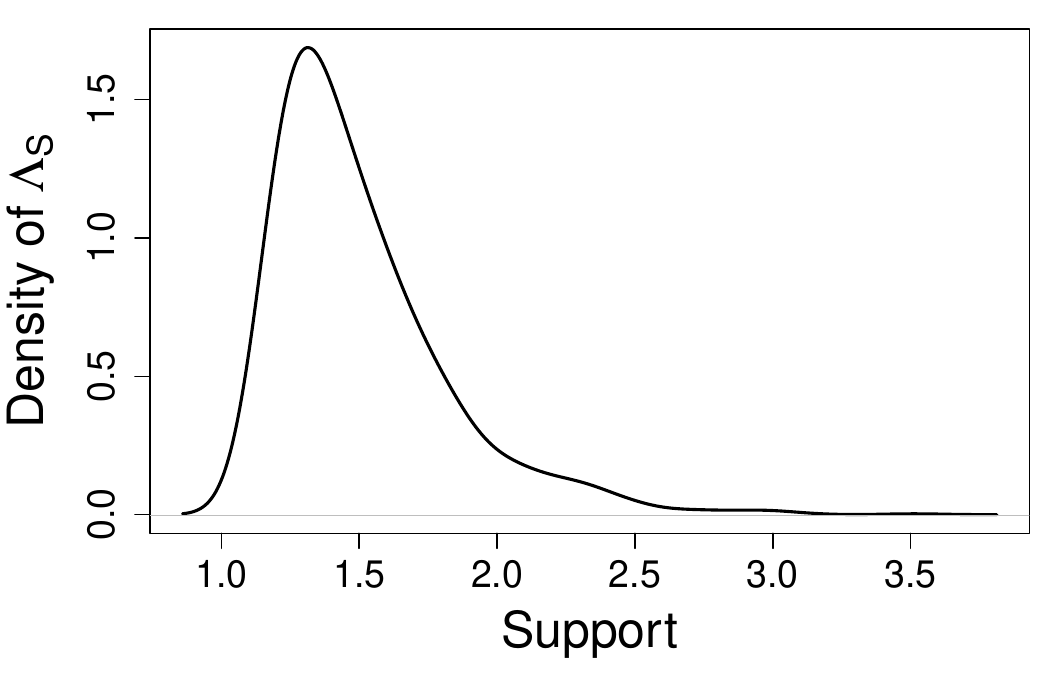}
	\caption{Cumulative distribution function of the optimal SLOPE penalty $\Lambda_S$ and probability density of its non-constant component, at $(\tppinf,\fdpinf)=(0.5676,0.6216)$. Here $\delta=0.3,\epsilon=0.2,\sigma=1,\Pi^\star= 4.9006$.}
	\label{fig:Lambda_S}
\end{figure}

Up until now, we are working in the normalized regime on $(\pi,\A)$ and we want to determine the original prior and penalty $(\Pi,\Lambda_S)$. To do so, we use the state evolution \eqref{SE_formula} to compute $\tau=\sqrt{\frac{\sigma^2}{1-E(\pi,\A_S)/\delta}}=3.6337$, which uniquely defines the two-point prior via $\Pi^\star={t_1}\cdot\tau= 4.9006$. We then apply the calibration \eqref{asym calibration} to derive $\Lambda_S$ and visualize the distribution in \Cref{fig:Lambda_S}.

To saturate the state evolution condition \eqref{eq:SE contraint_quote} so that $E(\pi,\A_S)=\delta$, while still fixing $\tppinf=0.5676$, we can increase the zero-threshold $\alpha$ from $t^\star$ (which is 1.19241) to $1.25672$ and derive ${t_1}=1.41748$ via \eqref{eq:tpp fdp zero-threshold}:
$$\PP(|{t_1}+Z|>\alpha)=\Phi({t_1}-\alpha)+\Phi(-{t_1}-\alpha)=u.$$

Again, $\A_S$ constructed by \eqref{eq:A_S form} is increasing and optimal. This new SLOPE zero-threshold $\alpha$ implies $\fdpinf=0.5954$ which is strictly smaller than the Lasso minimum $\fdpinf=0.6216$.

\subsection{A new TPP threshold $u^\dagger$}
In this section, we find the minimum $\tppinf$ such that we can leverage \Cref{app:two point analytic} to construct SLOPE $\fdpinf$ below the Lasso trade-off $q^\star_\text{Lasso}(\tppinf)$.

When $\tppinf=u$ and the zero-threshold $\alpha$ equals $t^\star(u)$ defined in \eqref{eq:t_star_eq}, the SLOPE penalty may have a normalized estimation error $E(\pi,\A_S)<\delta$. In the above example, we increase the zero-threshold $\alpha$ until the state evolution constraint is binding: $E(\pi,\A_S)=\delta$, thus obtaining smaller $\fdpinf$. From a different angle, we can decrease $u$ (and change $q^\star(u)$ and $t^\star(u)$ consequently) until $E(\pi,\A_S)=\delta$.

To be specific, we test a general $\tppinf=u$ and set the zero-threshold $\alpha$ at $t^\star(u)$. Then the single point $\pi^\star={t_1}$ can be computed via $\PP(|{t_1}+Z|>t^\star(u))=u$ and the SLOPE penalty function $\A_S$ is determined via \eqref{eq:A_S form}. Lastly, we compute the normalized estimation error $E(\pi,\A_S)$ if $\A_S$ is increasing. 

We define the smallest $u$ such that $E\leq\delta$ as our new TPP threshold $u^\dagger$:
\begin{align*}
	u^\dagger(\delta,\epsilon):=\inf\{u \text{ s.t. } &F_{t^\star(u)}[\max(t^\star(u),-H_{t_1}'/H_{t_1}),\rho_{t_1}]\leq\delta
	\\
	\text{ and }&\max(t^\star(u),-H_{t_1}'/H_{t_1}) \text{ is increasing}\}.
\end{align*}
Here the function $\rho_{t_1}(t)=\delta(t-{t_1})$ is the probability density function of $\pi^\star={t_1}$ and the functional $F$ is defined in \eqref{eq:F_alpha definition}.

\begin{figure}[!h]
	\centering
	\includegraphics[width=8cm,height=5.5cm]{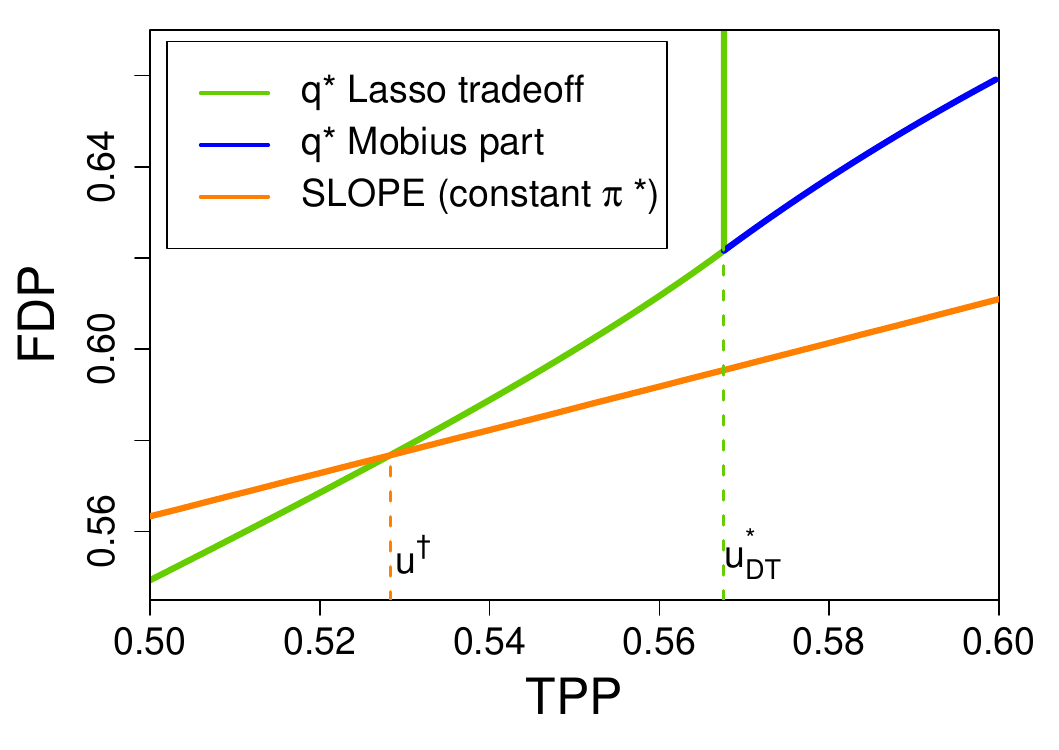}
	\caption{SLOPE $\tppinf$--$\fdpinf$ of constant $\pi^\star$ and the new TPP threshold $u^\dagger$ when $(\delta,\epsilon)=(0.3,0.2)$. The green line is the Lasso trade-off $q^\star$ and the blue line is the \Mobius part of $q^\star$. The orange line is the $\tppinf$--$\fdpinf$ realized by constant $\pi^\star$ and $\AA$ for the maximum zero-threshold $\alpha$. In this example, $u^\dagger=0.5283<\tppmax=0.5676$.}
	\label{fig:u dagger visualization}
\end{figure}

Under $\delta=0.3, \epsilon=0.2$, we find that $u^\dagger=0.5283<\tppmax=0.5676$ (visualized in \Cref{fig:u dagger visualization}). This indicates that we can show $q_\star<q^\star$ for a range of $u$ smaller than the DT power limit. We observe that below $u^\dagger$, the Lasso penalty and the infinity-or-nothing prior achieve smaller $\fdpinf$ than our SLOPE penalty and constant prior, and vice versa. We further offer graphical demonstration of the difference between $u^\dagger$ and $\tppmax$ in \Cref{fig:more u dagger}.

\begin{figure}[!h]
	\centering
	\includegraphics[width=7cm,height=5cm]{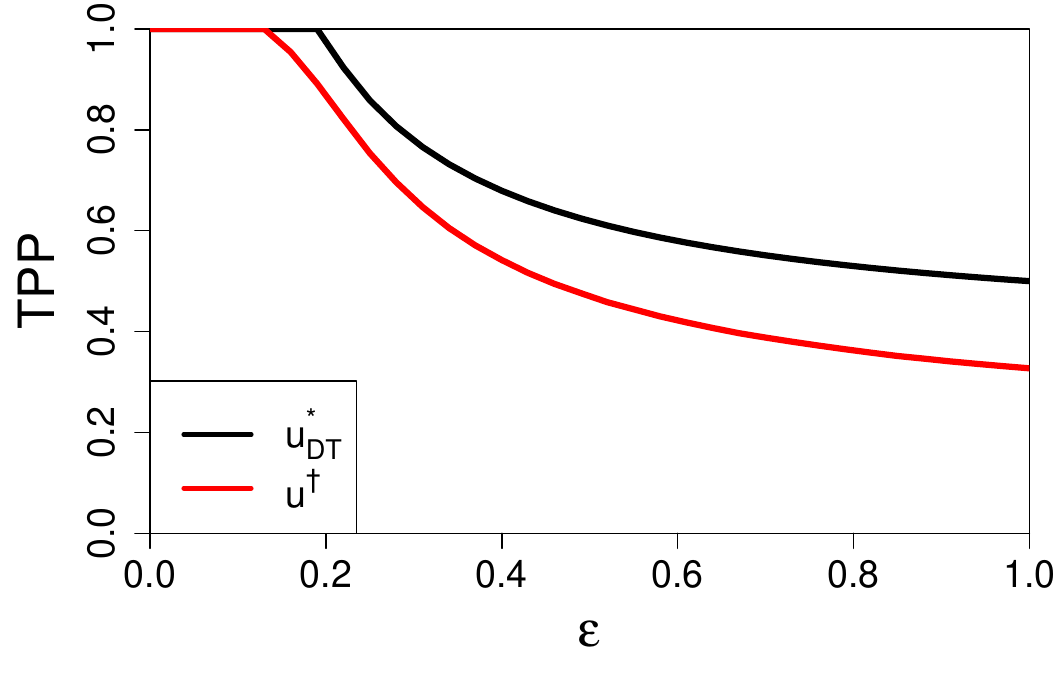}
	\includegraphics[width=7cm,height=5cm]{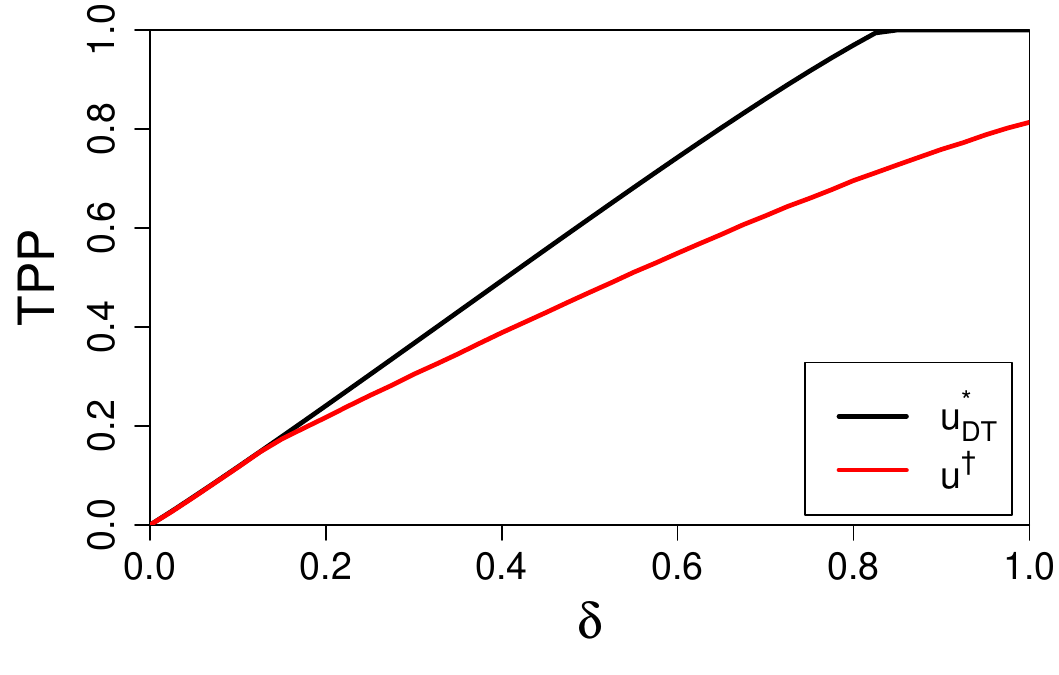}
	\caption{Comparison of $u^\dagger$ and $\tppmax$, fixing $\delta=0.5$ (left) or $\epsilon=0.5$ (right). The difference can be as large as 0.173 on the left plot and 0.282 on the right one.}
	\label{fig:more u dagger}
\end{figure}

	\section{Proving SLOPE outperforms the Lasso for fixed prior}\label{proof_instance_better}
\begin{proof}[Proof of \Cref{thm:instance better}]
	The proof is broken down into three pieces: we start with the MSE, then the asymptotic TPP and lastly the asymptotic FDP \footnote{\Cref{thm:instance better} can be generalized to further include certain unbounded signal prior $\Pi$ as long as \Cref{eq:EZsign} is satisfied. For example, for any Gaussian or Exponential $\Pi$, the SLOPE can outperform the Lasso.}.
	\\\\
	\textbf{SLOPE has smaller MSE}\qquad
	Fixing any bounded prior $\Pi$ and any scalar Lasso penalty $\lambda$, we can derive the corresponding $(\tau_L,\alpha_L)$ from the calibration \eqref{asym calibration} and the state evolution \eqref{SE_formula}, so as to work in the normalized regime $(\pi_L,\A_L)$. The quantity $\tau$ relates to the MSE by \citet[Corollary 3.4]{ourAMP}:
	\begin{align}
	\plim\|\widehat\bet-\bet\|^2/\p=\delta(\tau^2-\sigma^2).
	\label{eq:my corollary34}
	\end{align}

	Obviously, the SLOPE estimator has a smaller MSE than the Lasso one if and only if $\tau_S<\tau_L$, where $(\tau_S,\A_S)$ is the solution to the SLOPE calibration \eqref{asym calibration} and the state evolution \eqref{SE_formula}. 
	
	We now illustrate that this is always feasible by carefully designing the SLOPE penalty vector $\bfalph_S(p)$, or in the asymptotic sense, the SLOPE penalty distribution $\A_S$. 
	We directly work with the SLOPE AMP state evolution \eqref{SE_formula} instead of the Lasso AMP, since SLOPE covers the Lasso as a sub-case. In particular, we consider the two-level SLOPE of the form $\A_{\ell,\alpha_L,w}$ defined in \eqref{eq:theta_func}.
	
Our goal is to show that for any Lasso penalty $\A_L=\alpha_L=\A_{\alpha_L,\alpha_L,w}$, we can find a SLOPE penalty $\A_S=\A_{\ell,\alpha_L,w}$ for some sufficiently small $w$ and $\ell>\alpha_l$ such that $\tau_S<\tau_L$. In other words, among all the two-level SLOPE penalties $\A_{\ell,\alpha_L,w}$ with a zero-threshold $\alpha_L$, we show the optimal MSE is not achieved at $\ell=\alpha_L$.
	
To present a clear proof, we simplify the notation of $\prox_{J}(\bm a;{\bm b})$ by using $\bm\eta(\bm a,\bm b)$, or simply $\bm\eta$, where $\bm a:=\bm\Pi+\tau\Z$ and $\bm b:=\bfalph\tau$, where $\bfalph=\bfalph_{\ell,\alpha_L,w}$ (defined in \eqref{eq:theta_func}) with $\ell\geq\alpha_L$. On convergence of the state evolution \eqref{SE_formula}, we can differentiate both sides of
	$$\tau^2=\sigma^2+\frac{1}{\delta}\lim_{p\to\infty}\E\langle[\bm\eta(\bm\Pi+\tau \Z,\bfalph\tau)-\bm\Pi]^2\rangle=\sigma^2+\frac{1}{\delta}\lim_{p\to\infty}\E\langle[\bm\eta(\bm a,\bm b)-\bm\Pi]^2\rangle,$$
	with respect to $\ell\in\mathbb{R}$. Denoting $\tau'=\frac{\partial \tau}{\partial \ell}$, we obtain
	\begin{align*}
	2\tau\tau'&=\frac{\partial}{\partial \ell}\left(\sigma^2+\frac{1}{\delta}\lim_{\p\to\infty}\E\langle[\bm\eta(\bm a,\bm b)-\bm\Pi]^2\rangle\right)
	=\frac{1}{\delta}\lim_{\p\to\infty}\E\frac{\partial}{\partial \ell}\langle[\bm\eta(\bm a,\bm b)-\bm\Pi]^2\rangle.
	\end{align*}
	Then the chain rule leads to
	\begin{align}
	\tau\tau'=\lim_{\p\to\infty}\frac{1}{\n}\sum_j\E(\eta_j-\Pi_j)\frac{\partial\eta_j}{\partial \ell}
	=\lim_{\p\to\infty}\frac{1}{\n}\sum_j\E(\eta_j-\Pi_j)\sum_k\left[ \frac{d \eta_j}{d a_k} Z_k\frac{\partial \tau}{\partial \ell}+\frac{d \eta_j}{d b_k}\frac{\partial b_k}{\partial \ell}\right].
	\label{eq:tau'}
	\end{align}
To investigate the derivative terms, we copy some important facts in \citet[Appendix A]{ourAMP} here for reader's convenience:
\begin{align*}
	\frac{d}{da_k}  [\bm\eta(\mathbf{a}, \mathbf{b})]_j &= \mathbb{I} \{ |\bm\eta(\mathbf{a}, \mathbf{b})|_j = | \bm\eta(\mathbf{a}, \mathbf{b})|_k\}\sgn([ \bm\eta(\mathbf{a}, \mathbf{b})]_j [ \bm\eta(\mathbf{a}, \mathbf{b})]_k)  [\partial_1 \eta(\mathbf{a}, \mathbf{b})]_j,
	\\
	\frac{d}{d b_k}  [\bm\eta(\mathbf{a}, \mathbf{b})]_j  &= - \mathbb{I}\big\{ |\bm\eta(\mathbf{a}, \mathbf{b})|_j =|\bm\eta(\mathbf{a}, \mathbf{b})|_{\o(k)}\big\}\sgn \big([ \bm\eta(\mathbf{a}, \mathbf{b})]_j \big)    \big[\partial_1 \eta(\mathbf{a}, \mathbf{b})\big]_j.
\end{align*}
where the permutation $\o: i\to j$ finds the index of the $i$-th largest magnitude, i.e. $|\eta|_{\o(i)}:=|\eta|_{(i)}=|\eta|_j$, and its inverse function is the rank of the magnitudes. 
In the above notation, we have used (again from \citet[Appendix A]{ourAMP})
	$$[\partial_1\eta]_j=\frac{\partial\eta_j}{\partial a_j}=\frac{1}{\text{\Big|\{$1\leq k\leq \p: | \eta_k  |=| \eta_j  |$\}\Big|}},$$
	which converges to 1 as $\p\goto\infty$ if $\eta_j$ is unique in the magnitudes of $\bm\eta$ and to 0 otherwise. In the Lasso case (i.e. $\ell=\alpha_L$), each non-zero entry in $|\bm\eta|$ is indeed unique, and hence we can simplify $[\partial_1\eta]_j$ to $\mathbb{I}\{|\eta|_j\neq 0\}$. We now rewrite \eqref{eq:tau'} as
	\begin{align*}
	\tau\tau'&=\lim_{\p\to\infty}\frac{1}{\n}\sum_j\E(\eta_j-\Pi_j)\left[[\partial_1\eta]_j Z_j \tau'-\sgn(\eta_j)[\partial_1\eta]_j\frac{\partial b_{o^{-1}(j)}}{\partial \ell}\right]
	\\
	&=\lim_{\p\to\infty}\frac{1}{\n}\sum_{j:\eta_j\neq 0}\E(\eta_j-\Pi_j)\left[Z_j \tau'-\sgn(\eta_j)\frac{\partial b_{o^{-1}(j)}}{\partial \ell}\right]
	\\
	&=\lim_p\frac{1}{\n}\sum_{j:\eta_j\neq 0}\E(\eta_j-\Pi_j)\Big[Z_j\tau'-\sgn(\eta_j)(\alpha_{\o^{-1}(j)}\tau'+\mathbb{I}\{\o^{-1}(j)\leq wp\}\tau)\Big]
	\\
	&=\lim_p\frac{1}{\n}\sum_{j:\eta_j\neq 0}\E(\eta_j-\Pi_j)\left[\left(Z_j-\sgn(\eta_j)\alpha_L\right)\tau'-\sgn(\eta_j)\mathbb{I}\{\o^{-1}(j)\leq wp\}\tau\right].
	\end{align*}
	

	To summarize, we derive that
	\begin{equation}
	\begin{split}
	\frac{\partial\tau}{\partial\ell}&=\frac{\lim\limits_p\frac{1}{n} \sum\limits_{j:\o^{-1}(j)\leq wp}\E(\eta_j-\Pi_j)\sgn(\eta_j)\tau}{\lim\limits_p\frac{1}{\n}\sum\limits_{j:\eta_j\neq 0}\E(\eta_j-\Pi_j)\left(Z_j-\sgn(\eta_j)\alpha_L\right)-\tau}.
	\end{split}
	\label{eq:d tau d alpha}
	\end{equation}

	The rest of the proof contains two statements: (1) We will show that the numerator term is positive for sufficiently small $w$; (2) We also show that the denominator term is always negative for any Lasso penalty $\alpha_L$.
	
	To show that the numerator in \eqref{eq:d tau d alpha} is positive for small $w$, we write
	\begin{align*}
	\lim_p\frac{\tau}{n}\sum_{j:\o^{-1}(j)\leq wp}\E(\eta_j-\Pi_j)\sgn(\eta_j)
	=&\frac{w\tau}{\delta}\E\left[(\eta_j-\Pi_j)\sgn(\eta_j)\big|\o^{-1}(j)\leq wp\right]
	\\
	=&\frac{w\tau^2}{\delta}\E\left[(\eta_{\text{soft}}-\pi)\sgn(\eta_{\text{soft}})\big||\eta_{\text{soft}}|\geq q_w\right],
	\end{align*}
in which we slightly abuse the notation for the distribution $\eta\overset{\mathcal{D}}{:=}\eta_{\text{soft}}(\pi+ Z;\alpha_L)$ and define $q_w$ as the $w$-quantile of $|\eta_{\text{soft}}(\pi+ Z;\alpha_L)|$ such that $\PP\left(|\eta|\geq q_w\right)=w$.

Next, simple substitution gives that it is equivalent to show
\begin{align}
\E\left(Z\sgn(\eta)\Big| |\eta|\geq q_w\right)>\alpha_L.
\label{eq:EZsign}
\end{align}
We notice that as $w\to 0$, $q_w\to\infty$. Hence we can always consider $w$ small enough that the desired inequality above holds. The full proof of this fact is referred to \Cref{all_other_proofs}.

The next step is to show that the denominator in \eqref{eq:d tau d alpha} is negative, similar to the proof in \cite{zhang2021efficient}. By multiplying with the positive $\tau$, the denominator becomes
\begin{align*}
&\lim_p\frac{1}{\n}\sum_{j:\eta_j\neq 0}\E(\eta_j-\Pi_j)\left[Z_j-\sgn(\eta_j)\alpha_L\right]-\tau
\\
\propto&\lim_p\frac{1}{\n}\sum_{j:\eta_j\neq 0}\E(\eta_j-\Pi_j)\left[\tau Z_j-\tau\sgn(\eta_j)\alpha_L\right]-\tau^2
\\
=&\lim_p\frac{1}{\n}\sum_{j:\eta_j\neq 0}\E(\eta_j-\Pi_j)^2-\tau^2
\\
=&\lim_p\frac{1}{\n}\sum_{j:\eta_j}\E(\eta_j-\Pi_j)^2-\lim_p\frac{1}{n}\sum_{j:\eta_j=0}\E(\eta_j-\Pi_j)^2-\tau^2
\\
=&-\lim_p\frac{1}{\n}\sum_{j:\eta_j= 0}\E(\eta_j-\Pi_j)^2-\sigma^2<0,
\end{align*}
where the last equality follows from \eqref{eq:my corollary34}.

All in all, we finish the proof that $\frac{\partial\tau}{\partial\ell}$ in \eqref{eq:d tau d alpha} is negative for $\A_{\ell,\alpha_L,w}$ at $\ell=\alpha_L$ and small $w$. Along this negative gradient $\frac{\partial\tau}{\partial\ell}$, increasing the first argument of $\A_{\ell,\alpha_L,w}$ from $\ell=\alpha_L$ (the Lasso case) leads to a SLOPE penalty $\A_S$ and reduces $\tau_L$ to a smaller $\tau_S$. Equivalently, the SLOPE MSE is strictly smaller than the Lasso MSE.
\\\\
\textbf{SLOPE has higher TPP}\qquad
To prove the TPP result, we need the SLOPE to have smaller MSE (as shown previsouly) as well as the same zero-threshold as the Lasso. To achieve this, we claim that, for sufficiently small $w$ and some $\ell>\alpha_L$, the SLOPE zero-threshold $\alpha(\Pi,\Lambda_S)$ is the same as the Lasso zero-threshold $\alpha(\Pi,\Lambda_L)=\alpha_L$. 

In fact, the two-level SLOPE $\A_{\ell,\alpha_L,w}$ by its levels must have the zero-threshold as either $\ell$ or $\alpha_L$ (see \Cref{lambda meets}), and the zero-threshold will be $\alpha_L$ if and only if the sparsity $\kappa(\Pi,\Lambda_S)>w$ (see \Cref{fact:lambda meet}). Therefore it suffices to guarantee $\kappa(\Pi,\Lambda_S)>w$. From $\PP(|\Pi/\tau_S+Z|>\ell)\leq\kappa(\Pi,\Lambda_S)\leq\PP(|\Pi/\tau_S+Z|>\alpha_L)$, it is not hard to obtain that the sparsity $\kappa$ is continuous in $\ell$. Hence for any $w<\kappa(\Pi,\Lambda_L)$, there exists some $\ell>\alpha_L$ but close to $\alpha_L$ so that the SLOPE sparsity $\kappa(\Pi,\Lambda_S)>w$.

Now that we have $\tau_S<\tau_L$ and $\alpha(\Pi,\Lambda_S)=\alpha(\Pi,\Lambda_L)$, we can finish the proof by the definition of $\tppinf$: intuitively, $\pi_S:=\Pi/\tau_S>\pi_L:=\Pi/\tau_L$ and SLOPE $\tppinf=\PP(|\pi_S^\star+Z|>\alpha_L)>\PP(|\pi_L^\star+Z|>\alpha_L)=$ the Lasso $\tppinf$; formally, we show by \Cref{eq:tpp fdp zero-threshold} that
\begin{align*}
&\tppinf(\Pi,\Lambda_S)=\PP(|\pi^\star_S+Z|>\alpha(\Pi,\Lambda_S))=\int_{-\infty}^\infty\PP(|t/\tau_S+Z|>\alpha_L)p_{\Pi^\star}(t)dt
\\
>&\int_{-\infty}^\infty\PP(|t/\tau_L+Z|>\alpha_L)p_{\Pi^\star}(t)dt=\PP(|\pi^\star_L+Z|>\alpha(\Pi,\Lambda_L))=\tppinf(\Pi,\Lambda_L).
\end{align*}
\\\\
\textbf{SLOPE has lower FDP}\qquad
To prove the FDP result, we again use the fact that the SLOPE shares the same zero-threshold as the Lasso but has larger TPP. By \Cref{eq:tpp fdp zero-threshold}, 
\begin{align*}
\fdpinf(\Pi,\Lambda_S)
=&\dfrac{2(1-\epsilon)\Phi(-\alpha(\Pi,\Lambda_S))}{2(1-\epsilon)\Phi(-\alpha(\Pi,\Lambda_S))+\epsilon\tppinf(\Pi,\Lambda_S)}
\\
<&\dfrac{2(1-\epsilon)\Phi(-\alpha(\Pi,\Lambda_L))}{2(1-\epsilon)\Phi(-\alpha(\Pi,\Lambda_L))+\epsilon\tppinf(\Pi,\Lambda_L)}
=\fdpinf(\Pi,\Lambda_L).
\end{align*}
\end{proof}


	\section{Auxiliary proofs}
\label{all_other_proofs}
Here we give some technical proofs that have been used in this work.

\subsection{Derivation of $F_\alpha$ in \Cref{sec:lower bound}}
We are ready to give the explicit form of the functional $F_\alpha$ given that the zero-threshold is $\alpha\in\R_+$. Expanding the term in the integral form, we get
\begin{align*}
	F_{\alpha}[\AA,p_{\pi^\star}]
	&:=\E\left(\eta_{\textnormal{soft}}(\pi+Z;\AA(\pi+Z))-\pi\right)^2
	\\
	=&\int_{0}^{\infty}\int_{-\infty}^{\infty}\Big(\eta_{\text{soft}}\big(t+z;\AA(t+z)\big)-t\Big)^2\phi(z)dz p_\pi(t)dt,
\end{align*}
where $p_\pi$ is the probability density function of the normalized prior $\pi$, which is uniquely determined by $p_{\pi^\star}$ in a way to be explained shortly. 

We further expand the quadratic term in the inner integral, by using the fact that $\h{\pi+Z,\A}(t+z)=\eta_{\text{soft}}(t+z;\AA(t+z))=0$ for $-\alpha\leq t+z\leq \alpha$, since $\alpha$ is the zero-threshold in \Cref{zero threshold}. We obtain
\begin{align*}
	F_{\alpha}[\AA,p_{\pi^\star}]
	=&\int_{0}^{\infty}\Big[\int_{-\alpha-t}^{\alpha-t}t^2\phi(z)dz+\int_{\alpha-t}^{\infty}\big(z-\AA(t+z)\big)^2\phi(z)dz
	\\
	&+\int_{-\infty}^{-\alpha-t}\big(z+\AA(t+z)\big)^2\phi(z)dz\Big]p_\pi(t)dt
	\\	=&\int_{0}^{\infty}\Big[t^2(\Phi(\alpha-t)-\Phi(-\alpha-t))+\int_{\alpha}^{\infty}\big(z-t-\AA(z)\big)^2\phi(z-t)dz
	\\
	&+\int_{\alpha}^{\infty}\big(-z-t+\AA(z)\big)^2\phi(-z-t)dz\Big]p_\pi(t)dt.
\end{align*}

By the definition of $\pi^\star$ in \Cref{lem:A1}, we use $p_\pi(t)=(1-\epsilon)\delta(t)+\epsilon p_{\pi^\star}(t)$, in which $p_{\pi^\star}(t)$ is the probability density function of $\pi^\star$, to write
\begin{align}
	\begin{split}
		&F_{\alpha}[\AA,p_{\pi^\star}]
		\\
		=&2(1-\epsilon)\int_{\alpha}^{\infty}(z-\AA(z))^2\phi(z)dz
		+\epsilon\int_{0}^{\infty}\Big[t^2(\Phi(\alpha-t)-\Phi(-\alpha-t))
		\\
		&+\int_{\alpha}^{\infty}\big(z-t-\AA(z)\big)^2\phi(z-t)dz
		+\int_{\alpha}^{\infty}\big(-z-t+\AA(z)\big)^2\phi(-z-t)dz\Big]p_{\pi^\star}(t)dt	
		\\
		=&2(1-\epsilon)\int_{\alpha}^{\infty}(z-\AA(z))^2\phi(z)dz
		+\epsilon\int_{0}^{\infty}\Big[t^2(\Phi(\alpha-t)-\Phi(-\alpha-t))
		\\
		&+\int_{\alpha}^{\infty}\left(\big(z-t-\AA(z)\big)^2\phi(z-t)
		+\big(-z-t+\AA(z)\big)^2\phi(-z-t)\right)dz\Big]p_{\pi^\star}(t)dt.
	\end{split}
	\label{eq:F_alpha definition}
\end{align}

Since \Cref{lem:two point vertex} states that the optimal $p_{\pi^\star}(t)$ takes the form $\rho(t;t_1,t_2)=p_1\delta(t-t_1)+p_2\delta(t-t_2)$ in \eqref{eq:dirac delta prior}, the above functional turns into
\begin{align*}
	&F_\alpha[\AA,\rho(\cdot;t_1,t_2)]
	\\
	=&\int_{\alpha}^{\infty}\Big[2(1-\epsilon)(z-\AA(z))^2\phi(z)
	\\
	&+\epsilon p_1\left(\big(z-t_1-\AA(z)\big)^2\phi(z-t_1)+\big(-z-t_1+\AA(z)\big)^2\phi(-z-t_1)\right)
	\\
	&+\epsilon p_2\left(\big(z-t_2-\AA(z)\big)^2\phi(z-t_2)+\big(-z-t_2+\AA(z)\big)^2\phi(-z-t_2)\right)\Big]dz
	\\&+\epsilon p_1t_1^2\Big[\Phi(\alpha-t_1)-\Phi(-\alpha-t_1)\Big]
	+\epsilon p_2t_2^2\Big[\Phi(\alpha-t_2)-\Phi(-\alpha-t_2)\Big].
\end{align*}

To construct the quadratic programming in problem \eqref{eq:quadratic programming functional}, we can apply the left endpoint rule and approximate $F_\alpha[\AA,\rho(\cdot;t_1,t_2)]$ by
\begin{align}
	\begin{split}
		\bar{F}_{\alpha}(\bm\A;t_1,t_2)&=2(1-\epsilon)\sum_{i=1}^m(z_i-\A_i)^2\phi(z_i)\Delta z
		\\
		&+\epsilon p_1\sum_{i=1}^m\left(\big(z_i-t_1-\A_i\big)^2\phi(z_i-t_1)+\big(-z_i-t_1+\A_i\big)^2\phi(-z_i-t_1)\right)\Delta z
		\\
		&+\epsilon p_2\sum_{i=1}^m\left(\big(z_i-t_2-\A_i\big)^2\phi(z_i-t_2)+\big(-z_i-t_2+\A_i\big)^2\phi(-z_i-t_2)\right)\Delta z
		\\
		&+\epsilon p_1 t_1^2\Big[\Phi(\alpha-t_1)-\Phi(-\alpha-t_1)\Big]
		+\epsilon p_2 t_2^2\Big[\Phi(\alpha-t_2)-\Phi(-\alpha-t_2)\Big].
	\end{split}
	\label{eq:F bar}
\end{align}

\subsection{Proof of \Cref{lem:two point vertex}}
\begin{proof}
	In general, $\rho^*$ can always be approximated by a sum of Dirac delta functions, $\rho^*(t)=\sum_{i=1}^m p_i\delta(t-t_i)$. In particular, since $\rho^*$ is a probability density function, we require $0<p_i<1$: otherwise if for some $i$, $p_i=1$, then $m=1$ and we are done.
	
	We now show $m<3$ by contradiction. The vertex principle of linear programming states that the minimum value of the linear objective function occurs at the vertices of the feasible region. Hence it suffices to show that all vertices are two-point Dirac delta functions. 
	
	The constraints in problem \eqref{eq:linear programming} lead to
	\begin{align*}
		\sum_i p_i=1,\quad \sum_i p_i[\Phi(t_i-\alpha)+\Phi(-t_i-\alpha)]=u.
	\end{align*}
	
	Suppose $m\geq 3$, then there always exists $\rho'(t)=\sum_{i=1}^m p_i'\delta(t-t_i)$ such that
	\begin{itemize}
		\item $p_i=p_i'$ for $i>3$;
		\item $p_1+p_2+p_3=p_1'+p_2'+p_3'$;
		\item $p_1 h_\alpha(t_1)+p_2 h_\alpha(t_2)+p_3 h_\alpha(t_3)=p_1'h_\alpha(t_1)+p_2'h_\alpha(t_2)+p_3'h_\alpha(t_3)$;
	\end{itemize}
	where we denote $h_\alpha(t_i):=\Phi(t_i-\alpha)+\Phi(-t_i-\alpha)$. In other words, we can find $\rho'$ such that
	\begin{align*}
		\mathbf{0}=\begin{pmatrix}
			1&1&1
			\\
			h_\alpha(t_1)&h_\alpha(t_2)&h_\alpha(t_3)
		\end{pmatrix}
		\left[\begin{pmatrix}
			p_1\\p_2\\p_3
		\end{pmatrix}
		-
		\begin{pmatrix}
			p_1'\\p_2'\\p_3'
		\end{pmatrix}
		\right]
		.
	\end{align*}
	Since there are only two equations involving the three unknown variables $p'_1, p'_2$ and $p'_3$, in the generic case, we can represent the infinitely many $\rho'$ with one degree of freedom,
	\begin{align*}
		\begin{pmatrix}
			p_1'\\p_2'\\p_3'
		\end{pmatrix}
		=\begin{pmatrix}
			p_1\\p_2\\p_3
		\end{pmatrix}
		+
		s\begin{pmatrix}
			c_1\\c_2\\c_3
		\end{pmatrix},
	\end{align*}
	using the null vector $(c_1,c_2,c_3)^\top$ of the above matrix.
	
	As $0<p_i'<1$ for $i=1,2,3$, we claim that for all $s\in(-s_0,s_0)$ with some $s_0>0$, $(p_1',p_2',p_3')^\top$ defined above is feasible for problem \eqref{eq:linear programming}. In other words, suppose we explicitly define
	$$\rho_s(t)=\sum_{i=1}^3(p_i+c_i s)\delta(t-t_i)+\sum_{i>3}p_i\delta(t-t_i),$$
	then there exists a range of $s\in\R$ such that $\rho_s$ is feasible. However, one can easily check that
	$$\rho^*=\frac{1}{2}\left(\rho_s+\rho_{-s}\right)$$
	is also a feasible solution. Hence $\rho^*$ is not a vertex. Contradiction.
\end{proof}

\subsection{Proof of \Cref{eq:EZsign}}
\begin{proof}
	To see that for large enough $q_w$, $\E\left(Z\sgn(\eta)\Big| |\eta|>q_w\right)\geq \alpha_L$, we write $\tilde{q}_w:=q_w+\alpha_L$ and study
	$$\E\left(Z\sgn(\eta)\Big| |\eta|\geq q_w\right)=\E\left(Z\sgn(\eta)\Big| |\pi+Z|\geq \tilde{q}_w\right).$$
	
	We have
	\begin{align*}
		&\E\left(Z\sgn(\eta)\Big| |\pi+Z|\geq \tilde{q}_w\right)
		\\
		=&\E\left(Z\Big| \pi+Z\geq \tilde{q}_w\right)\frac{\PP(\pi+Z\geq \tilde{q}_w)}{\PP(|\pi+Z|\geq \tilde{q}_w)}
		-\E\left(Z\Big| \pi+Z\leq -\tilde{q}_w\right)\frac{\PP(\pi+Z\leq -\tilde{q}_w)}{\PP(|\pi+Z|\geq \tilde{q}_w)}
		\\
		=&\frac{\int_{-\infty}^\infty \left[\phi\left(t-\tilde q_w\right)+\phi\left(-\tilde q_w-t\right)\right]p_{\pi}(t)dt}{\int_{-\infty}^\infty \left[\Phi\left(t-\tilde q_w\right)+\Phi\left(-\tilde q_w-t\right)\right]p_{\pi}(t)dt},
	\end{align*}
	in which $p_{\pi}(t)$ is the unknown but fixed probability density function of $\pi$.
	
	Now we show cases where the above ratio of integrals goes to $\infty$ as $q_w\to\infty$ or equivalently $\tilde{q}_w\to\infty$. For bounded $\pi$ (in fact for priors with bounded essential infimum and essential supreme), denoting the minimum and maximum as $\pi_\text{min}$ and $\pi_\text{max}$, then the ratio is
	\begin{align*}
		\frac{\int_{\pi_\text{min}}^{\pi_\text{max}} \left[\phi\left(t-\tilde q_w\right)+\phi\left(-\tilde q_w-t\right)\right]p_{\pi}(t)dt}{\int_{\pi_\text{min}}^{\pi_\text{max}} \left[\Phi\left(t-\tilde q_w\right)+\Phi\left(-\tilde q_w-t\right)\right]p_{\pi}(t)dt}.
	\end{align*}
	
	Using the fact that $\phi(x)+x\Phi(x)>0$, we have
	\begin{align*}
		&\E\left(Z\sgn(\eta)\Big| |\pi+Z|\geq \tilde{q}_w\right)
		\\
		>&\frac{\int_{\pi_\text{min}}^{\pi_\text{max}} \left[(\tilde q_w-t)\Phi\left(t-\tilde q_w\right)+(\tilde q_w+t)\Phi\left(-\tilde q_w-t\right)\right]p_{\pi}(t)dt}{\int_{\pi_\text{min}}^{\pi_\text{max}} \left[\Phi\left(t-\tilde q_w\right)+\Phi\left(-\tilde q_w-t\right)\right]p_{\pi}(t)dt}
		\\
		>&\frac{\int_{\pi_\text{min}}^{\pi_\text{max}} \left[(\tilde q_w-\pi_\text{max})\Phi\left(t-\tilde q_w\right)+(\tilde q_w+\pi_\text{min})\Phi\left(-\tilde q_w-t\right)\right]p_{\pi}(t)dt}{\int_{\pi_\text{min}}^{\pi_\text{max}} \left[\Phi\left(t-\tilde q_w\right)+\Phi\left(-\tilde q_w-t\right)\right]p_{\pi}(t)dt}
		\\
		>&\min\{\tilde q_w-\pi_\text{max},\tilde q_w+\pi_\text{min}\}.
	\end{align*}
	
	In summary, when $w\to 0$, all $q_w,\tilde q_w$ and $\E(Z\sgn(\eta) \big| |\eta|\geq q_w)\to\infty$. Therefore we have $\E(Z\sgn(\eta) \big| |\eta|\geq q_w)>\alpha_L$ for sufficiently small $w$.
\end{proof}


	\section{Computation of SLOPE AMP quantities}

In order to compute $\bfalph$ and $\tau$, e.g. for the AMP calibration or for computing the estimation error, we need to estimate the SLOPE proximal operator in the state evolution \eqref{finite_SE}. Despite that the Monte Carlo method is easy to implement, it is often unstable nor efficient for high-dimensional SLOPE problems, say when $\p$ is in the order of thousands. Here we demonstrate how to approximate the normalized estimation error $E(\Pi,\Lambda)$ in a way that matches the truth asymptotically and has satisfactory approximation error in the finite dimension (see bottom-right plot in \Cref{fig:decompose SE}).

Notice in this section, the prior distribution $\Pi$ is general and does not necessarily satisfy the sparsity assumption $\PP(\Pi\neq 0)=\epsilon$.

\subsection{Approximating state evolution and calibration with quantiles}
In state evolution \eqref{finite_SE}, the expectation term can be difficult to evaluate because of the ordering and the non-separability of the sorted norm. In addition, the convolution between $\Pi$ and $Z$ also makes the estimation difficult.

We propose the following method for estimation: denote $q_D$ as discretized quantile function of distribution $D$ at $\{\frac{1}{p},\frac{2}{p},...,\frac{p-1}{p}\}$. Denote $\Pi$ as the true distribution of $\bet$, $\bm\Pi_p$ as p-variate $\Pi$ with i.i.d. entries; $\pi$ as $\Pi/\tau$ and $\bm\pi_p$ as p-variate $\pi$ with i.i.d. entries accordingly. Similarly denote standard normals $Z$ and $\bm Z_p$. Assume $\bfalph\in \mathbb{R}^p$ is $p$-variate $\A$ with i.i.d. entries (in decreasing order). We can decompose
\begin{align}
\begin{split}
&\E\langle [\prox_{J}(\bm\Pi_p+\tau \bm Z_p;{\bfalph\tau})-\bm\Pi_p]^2\rangle
\\
=&\E\langle[\prox_{J}(\bm\Pi_p+\tau \bm Z_p;{\bfalph\tau})]^2\rangle+\E\langle \bm\Pi_p^2\rangle-\frac{2}{p}\E[\prox_{J}(\bm\Pi_p+\tau \bm Z_p;{\bfalph\tau})^\top \bm\Pi_p]
\\
=&\E\langle[\prox_{J}(\bm\Pi_p+\tau \bm Z_p;{\bfalph\tau})]^2\rangle+\E\Pi^2-\frac{2}{p}\E[\prox_{J}(\bm\Pi_p+\tau\bm Z_p;{\bfalph\tau})^\top \bm\Pi_p]
\\
\approx&\langle[\prox_{J}(q_{\Pi+\tau Z};{q_{\A\tau}})]^2\rangle+\frac{1}{p}q_{\Pi}^\top q_{\Pi}-\frac{2}{p} \prox_{J}(q_{\Pi+\tau Z};{q_{\A\tau}})^\top \E[\bm\Pi_p|\bm\Pi_p+\tau\bm Z_p=q_{\Pi+\tau Z}].
\end{split}
\label{decompose}
\end{align}

Such approximation for \eqref{decompose} is consistent and can be visualized in \Cref{fig:decompose SE}. This is due to the fact that the ordering and the signs do not affect the sum of squares and the property of Riemann Stieltjes integral (see \cite{johnsonbaugh2012foundations,kolmogorov1975introductory,rudin1976principles}).
\begin{fact}[Existence of Riemann Stieltjes integral]
Suppose $f$ is continuous and $g$ is of bounded variation. For every $\epsilon>0$, there exists $\delta>0$ such that for every partition $P:=(a=x_0<x_1<\cdots<x_p=b)$ with mesh$(P)<\delta$, and for every choice of points $c_i$ in $[x_i, x_{i+1}]$, we have
\begin{align*}
\left|S(P,f,g)-\int_{a}^{b}f(x)dg(x)\right|<\epsilon,
\end{align*}
where $S(P,f,g):=\sum_{{i=1}}^{{p-1}}f(c_{i})(g(x_{{i+1}})-g(x_{i}))$.
\end{fact}
We start with the simplest second term in \eqref{decompose}. Set $f(x)=x^2$ and $g(x)$ to be cumulative distribution function of $\Pi$. Denote $q_i$ as $\frac{i}{p}$-th quantile. Setting $c_i=q_i$ and $P=(q_1,...,q_{p-1})$, then we have the approximate sum as
\begin{align*}
\frac{1}{p}q_{\Pi}^\top q_{\Pi}&=S(P,f,g)=q_1^2\PP(q_1<\Pi<q_2)+\cdots+q_{p-1}^2\PP(q_{p-1}<\Pi<q_{p})\\
&=\sum_i\int_{q_i}^{q_{i+1}} q_i^2 dg(x)\rightarrow \int_{-\infty}^\infty x^2 dg(x)=\E\Pi^2.
\end{align*}
Similarly, for the first term in \eqref{decompose}, denote the distribution to which the empirical distribution of $\prox_{J}(\bm\Pi+\tau\bm Z;\bfalph\tau)$ converges as $\widehat{\Pi}$. By approximating the Riemann Stieltjes integral twice, we have 
$$
\langle[\prox_{J}(q_{\Pi+\tau Z};q_{\A\tau})]^2\rangle
\to\E \widehat{\Pi}^2
=\plim\E\langle[\prox_{J}(\bm\Pi+\tau\bm Z;\bfalph\tau)]^2\rangle.
$$
For the last term in \eqref{decompose}, we transform the term via the law of total expectation,
\begin{align*}
&\plim\frac{1}{\p}\E[\prox_{J}(\bm\Pi+\tau\bm Z;\bfalph\tau)^\top \bm\Pi]
\\
=&\lim\frac{1}{\p}\E[\prox_{J}(\bm\Pi+\tau\bm Z;{q_{\A\tau}})^\top \bm\Pi]
\\
=&\lim\frac{1}{\p}\E_{\Pi+\tau Z}[\E_{\Pi,Z}[\prox_{J}(\bm\Pi+\tau\bm Z;{q_{\A\tau}})^\top \bm\Pi\big|\bm\Pi+\tau\bm Z]]
\\
=&\lim\frac{1}{\p}\E_{\Pi,Z}[\prox_{J}(q_{\Pi+\tau Z};{q_{\A\tau}})^\top \bm\Pi\big|\bm\Pi+\tau\bm Z=q_{\Pi+\tau Z}]
\\
=&\lim\frac{1}{\p}\left(\prox_{J}(q_{\Pi+\tau Z};{q_{\A\tau}})^\top \E_{\Pi,Z}[ \bm\Pi\big|\bm\Pi+\tau\bm Z=q_{\Pi+\tau Z}]\right).
\end{align*}


Before we move on to the next section where we look at the conditional expectation term above, we pause to remark that the approximation via quantiles can be also used for computing the calibration \eqref{finite_calibration}: the calculation of $\E\|\prox_{J}(\bm\Pi+\tau\bm Z;{\bm A\tau})\|_0^*$ can be approximated by the number of unique values in $\left|\prox_{J}(q_{\Pi+\tau Z};{q_{\A\tau}})\right|$.

\subsection{Closed-form of conditional expectation}
The challenge remains on computing the vector  $\E[\bm\Pi|\bm\Pi+\tau\bm Z=q_{\Pi+\tau Z}]$. We will derive its closed-form by applying the inverse transform sampling on each entry. The effect of approximation using our explicit form is demonstrated in \Cref{fig:decompose SE different priors}. 

\begin{figure}[!htp]
	\centering
	\hspace{-0.5cm}
	\includegraphics[width=0.33\linewidth]{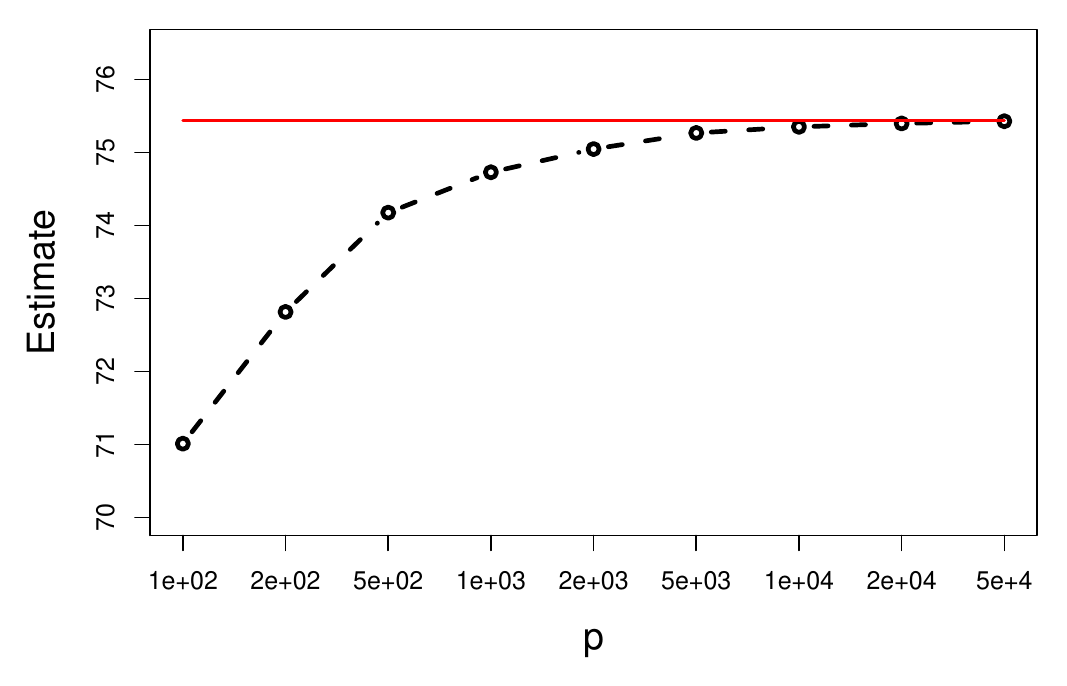}
	\includegraphics[width=0.33\linewidth]{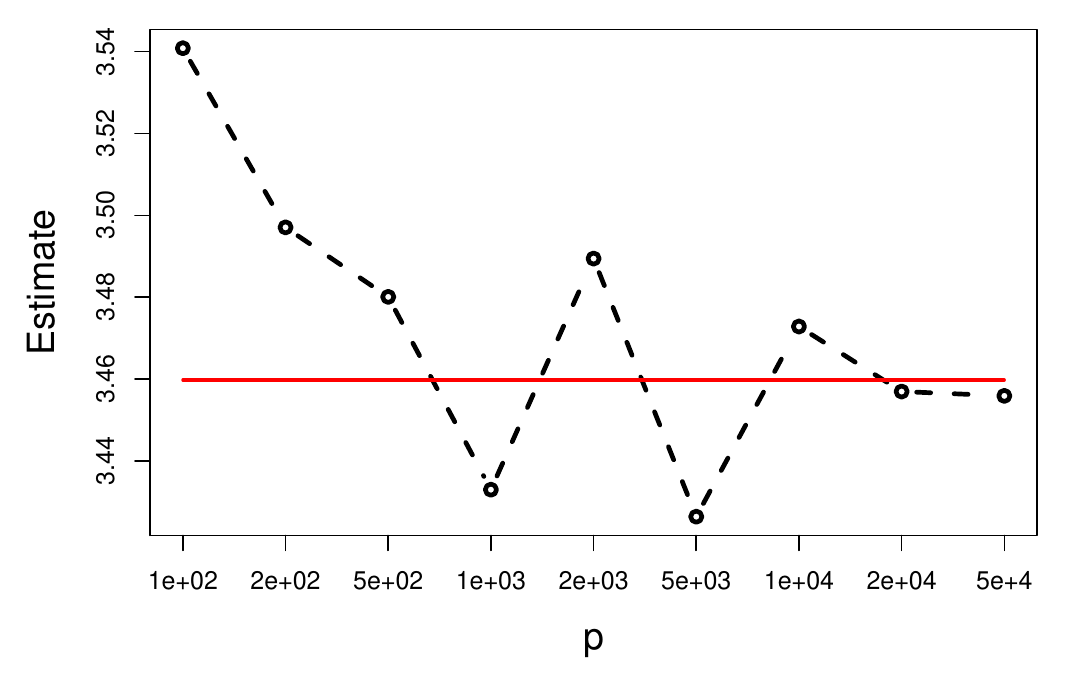}
	\includegraphics[width=0.33\linewidth]{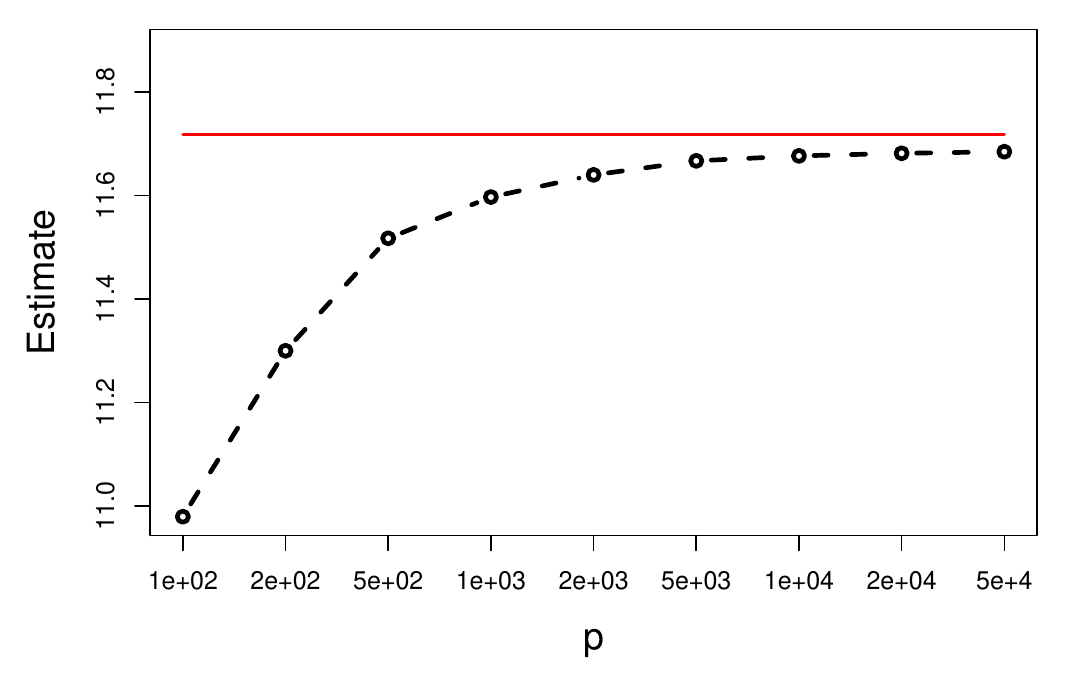}
	\caption{$\E\langle\prox_{J}(\bm\Pi+\tau \bm Z;{\bm A\tau})^\top \bm\Pi\rangle$ estimated by the quantiles and the closed-form conditional expectation (black dotted) against true expectation (red solid). Here $\delta=0.3$ and $\A\sim \text{Exp}(0.2)/10$. Left: $\Pi\sim\text{Exp(0.1)}$. Middle: $\Pi\sim10\cdot\text{Bernoulli(0.1)}$. Right: $\Pi\sim \mathcal{N}(2,25)$.}
	\label{fig:decompose SE different priors}
\end{figure}
For any $q\in \mathbb{R}$, denoting the support of $\Pi$ as $\supp$ and the probability density function as $p_{\Pi}$, we get
\begin{align*}
\E[\Pi|\Pi+\tau Z=q]=&\frac{\int_{\supp(\Pi)}{x\cdot\frac{1}{\sqrt{2\pi}\tau}\exp(-\frac{(q-x)^2}{2\tau^2})}p_{\Pi}(x)dx}{\int_{\supp(\Pi)}{ \frac{1}{\sqrt{2\pi}\tau}\exp(-\frac{(q-x)^2}{2\tau^2})}p_{\Pi}(x)dx}.
\end{align*}

Substitute $u=q-x$,
\begin{align*}
\E[\Pi|\Pi+\tau Z=q]=&q-\frac{\int_{\supp(q-\Pi)}{u\exp(-\frac{u^2}{2\tau^2})}p_{\Pi}(q-u)du}{\int_{\supp(q-\Pi)}{\exp(-\frac{u^2}{2\tau^2})}p_{\Pi}(q-u)du}.
\end{align*}

Denote $s_q(u):=p_{\Pi}(q-u)\exp(-\frac{u^2}{2\tau^2}), S_q:=\int_{-\infty}^\infty{s_q(u)du}$. Then $h_q(u)=s_q(u)/S_q$ is a normalized density of $s_q$. It can be viewed as a posterior density $h(u|q)$ with a Gaussian prior of $u$ and $p_\Pi(q-u)$ as evidence. Then
\begin{align*}
\E[\Pi|\Pi+\tau Z=q]=&q-\frac{\int_{\supp(q-\Pi)}{u\cdot s_q(u)du}}{\int_{\supp(q-\Pi)}{s_q(u)du}}
\\
=&q-\frac{\int_{\supp(q-\Pi)}{u\cdot h_q(u)du}}{\int_{\supp(q-\Pi)}{h_q(u)du}}=q-\E[U_q|U_q\in{\supp(q-\Pi)}],
\end{align*}
where $U_q$ is a univariate random variable with the density $h_q$.

%
Here we derive explicit formulae when the priors are Gaussian, exponential and Bernoulli. Two types of special generalization are worth mentioning: (1) when the support of $\Pi$ is $(-\infty,\infty)$, the conditional expectation $\E[U_q|U_q\in{\supp(q-\Pi)}]$ is indeed unconditional and the computation is simplified; (2) the cases of discrete priors can be easily derived in general besides the Bernoulli case.

\textbf{Gaussian distribution\quad}
When $\Pi=\N(\mu,\sigma^2)$, we have $\Pi+\tau Z=\N(\mu,\sigma^2+\tau^2)$, and $h_q(u)$ is the density of $\N(\frac{(q-\mu)\tau^2}{\sigma^2+\tau^2},\frac{\sigma^2\tau^2}{\sigma^2+\tau^2})$. Hence
\begin{align*}
\E[\Pi|\Pi+\tau Z=q]=&q-\E[U_q]=q-\frac{(q-\mu)\tau^2}{\sigma^2+\tau^2}=\frac{q\sigma^2+\mu\tau^2}{\sigma^2+\tau^2}.
\end{align*}

\textbf{Exponential distribution\quad}
When $\Pi=\text{Exp}(c)$, we have $\Pi+\tau Z$ being an exponentially modified Gaussian (EMG) distribution. Then $$s_q(u)=c\exp\left(-cq+\frac{c^2\tau^2}{2}\right)\left[\frac{1}{\sqrt{2\pi}\tau}\exp(-\frac{(u-c\tau^2)^2}{2\tau^2})\right],$$
and $h_q$ is the density of $N(c\tau^2,\tau^2)$. And, denoting $\xi=\frac{q-c\tau^2}{\tau}$, we get
\begin{align*}
\E[\Pi|\Pi+\tau Z=q]=&q-\E[U_q|U_q\in(-\infty,q)]=q-\E[U_q|U_q<q]=\tau\left(\xi+\frac{\phi(\xi)}{\Phi(\xi)}\right).
\end{align*}

\textbf{Discrete distribution\quad}
We begin with $\Pi=k\cdot \text{Bernoulli}(\epsilon)$ and then generalize to any discrete priors. Writing the density as $(1-\epsilon)\delta(x)+\epsilon\delta(x-k)$,we have
\begin{align*}
h_q(u)&\propto [\epsilon\delta(q-u-k)+(1-\epsilon)\delta(q-u)]\exp(-\frac{u^2}{2\tau^2})
\\
&=[\epsilon\delta(u-(q-k))+(1-\epsilon)\delta(u-q)]\exp(-\frac{u^2}{2\tau^2}).
\end{align*}
The last equality is true since the Dirac delta function is an even function. Hence
\begin{align*}
U_q&=
\begin{cases}
q&\text{ w.p. $(1-\epsilon)\exp(-\frac{q^2}{2\tau^2})/C$}
\\
q-k&\text{ w.p. $\epsilon\exp(-\frac{(q-k)^2}{2\tau^2})/C$}
\end{cases}
\end{align*}
where $C=(1-\epsilon)\exp(-\frac{q^2}{2\tau^2})+\epsilon\exp(-\frac{(q-k)^2}{2\tau^2})$. We get
\begin{align*}
\E[\Pi|\Pi+\tau Z=q]=&q-\E[U_q|U_q\in\{q,q-k\}]=q-\E[U_q]
\\
=&q-q\PP[U_q=q]-(q-k)\PP[U_q=q-k].
\end{align*}
where both probabilities are given above.

\begin{rem} 
It is easy to derive the conditional expectation for any discrete priors by writing the probability mass function as a sum of Dirac delta functions. In general, suppose the prior takes values in $a_1,...a_n$ with probability $p_1,...p_n$, then $U_q$ takes values $(q-a_i)$ with probability $\PP_i=p_i\exp(-\frac{(q-a_i)^2}{2\tau^2})/C$ where $C$ is normalizing constant and the conditional expectation is $q-\sum_i(q-a_i)\PP_i$.
\end{rem}

\subsection{Algorithm for state evolution term}
Taking the quantile method and closed-form conditional expectation described in the previous sections, we present the following algorithm to compute the state evolution term efficiently.

\begin{algorithm}[H]
	\caption{Calculating $\E\langle[\prox_{J}(\bm\Pi+\tau\bm Z;{\bm A\tau})-\bm\Pi]^2\rangle$ given $\A,\Pi,\tau$}
	\begin{algorithmic}
		\STATE{1. Derive quantiles $q_{\A}, q_{\Pi}$ }
		\STATE{2. Derive $D=\Pi+\tau Z$ by convolution and compute $q_D$}
		\STATE{3. Compute $G=\prox_{J}(q_{D};q_{\A\tau})\in \R^\p$, (notice that $q_{\A\tau}=\tau q_{\A}$)}
		\STATE{4. Compute $H=\E[\bm\Pi|\bm\Pi+\tau\bm Z=q_{D}] \in \R^\p$}
		\STATE{5. Return $\frac{1}{p}\left[G^\top G+q_{\Pi}^\top q_{\Pi}^{}-2G^\top H\right]$}
	\end{algorithmic}
\end{algorithm}

We give some simulation results on different dimensions. Since all three priors show similar patterns, only the exponential prior case is plotted. 
\begin{figure}[!htb]
	\centering
	\includegraphics[width=0.45\linewidth]{ComplexExp}
	\includegraphics[width=0.45\linewidth]{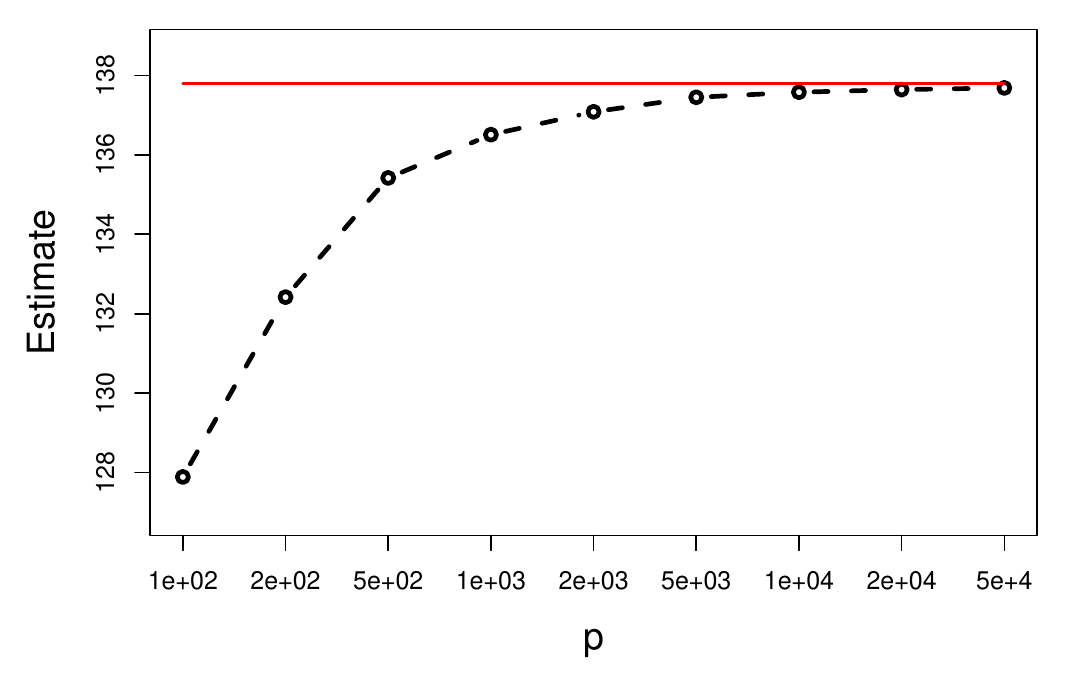}
	\includegraphics[width=0.45\linewidth]{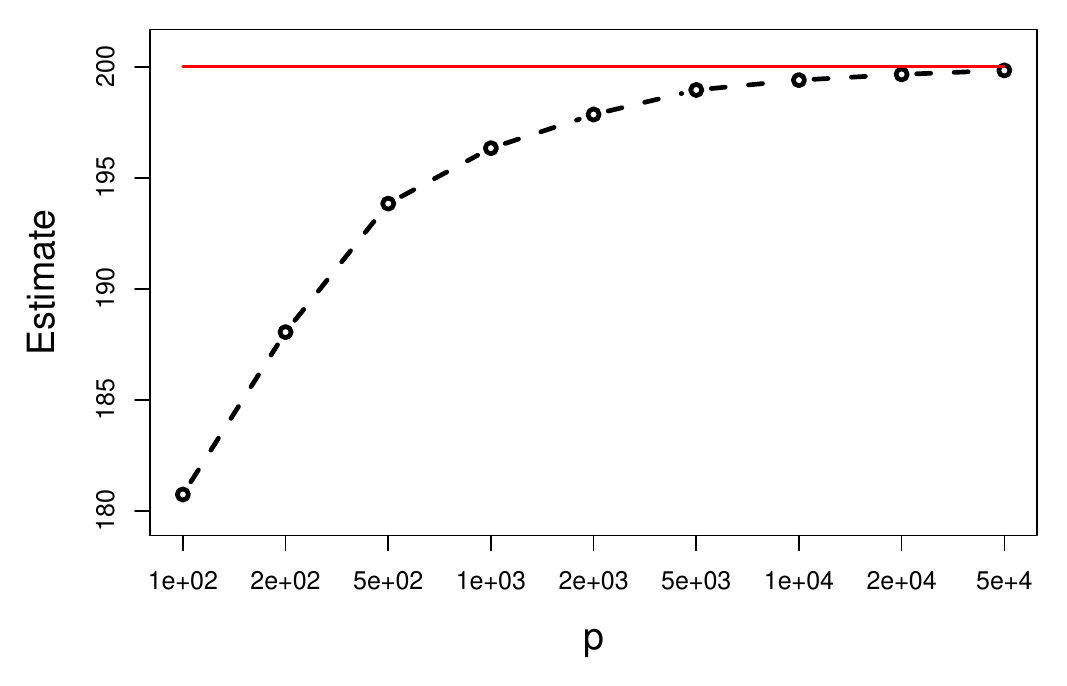}
	\includegraphics[width=0.45\linewidth]{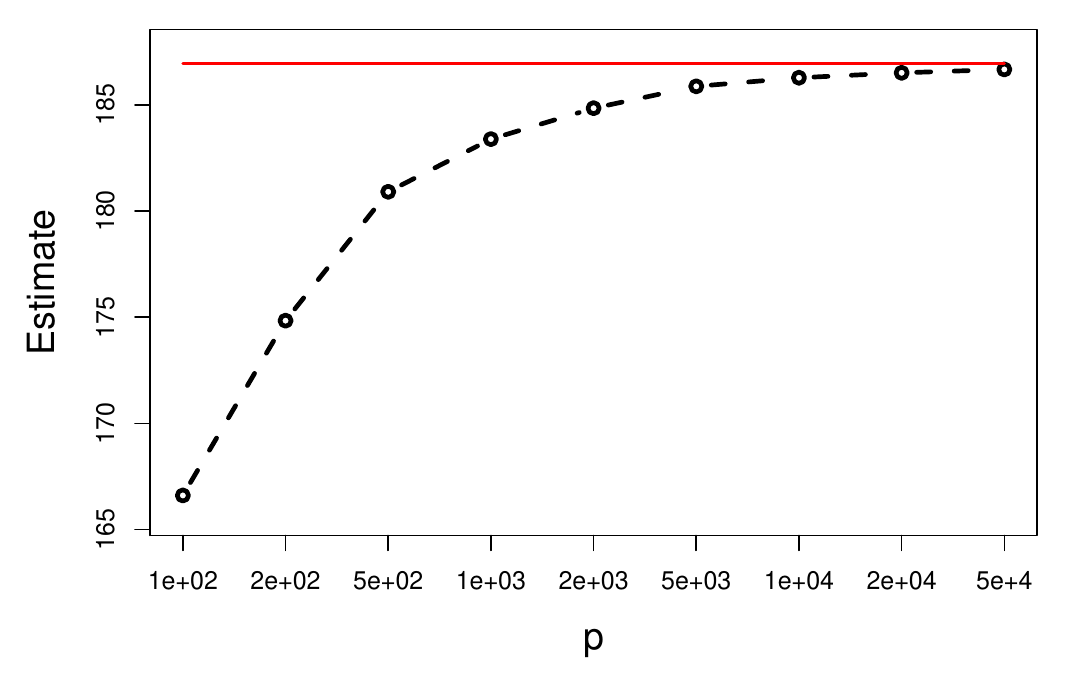}
	\caption{$\E\langle\prox_{J}(\bm\Pi+\tau \bm Z;{\bm A\tau})^\top \bm\Pi\rangle$ (top-left),
		$\E\langle[\prox_{J}(\Pi+\tau Z;{\alpha\tau})]^2\rangle$ (top-right),
		$\E[\Pi^2]$ (bottom-left) and $\E\langle[\prox_{J}(\bm\Pi+\tau\bm Z;{\bm A\tau})-\bm\Pi]^2\rangle$ (bottom-right) estimated by the quantiles (black dotted) against true expectation (red solid). Here $\sigma=0$ (no noise), $\X\sim \mathcal{N}(0,1/n), \n/\p=\delta=0.3, \A\sim \text{Exp}(0.2)/10, \Pi\sim \text{Exp}(0.1)$.}
	\label{fig:decompose SE}
\end{figure}

\clearpage
	\section{Design of SLOPE penalty under fixed prior}
	\label{app:quadratic programming}
	This section studies the problem of minimizing $\fdpinf$ at fixed $\tppinf=u$ over all possible SLOPE penalties, when the prior $\Pi$ is fixed. This problem has been investigated extensively in \cite{SLOPEasymptotic} but not via the quadratic programming approach that we proposed in \Cref{sec:lower bound}. We give the detailed procedure to find the SLOPE trade-off below.
	\begin{enumerate}
		\item Given $\Pi$, we try different $\tau$ from the small to the large, which defines $\pi:=\Pi/\tau$ and $\pi^\star:=\Pi^\star/\tau$. Denote the corresponding probability density function as $p_{\pi^\star}$.
		\item Since $\tppinf=u$, we require $\int_0^\infty [\Phi(t-\alpha)+\Phi(-t-\alpha)]p_{\pi^\star}(t)dt=u$ by the definition of the zero-threshold $\alpha$. Note that $\alpha(\tau)$ is a unique scalar for a given $\pi$, or equivalently $\tau$.
		\item By \eqref{eq:F_alpha definition}, we have the formula of $F_\alpha[\mathrm{A}_\textup{eff},p_\pi]$ which we want to minimize over all $\mathrm{A}_\textup{eff}$ under the same constraints as in \eqref{eq:quadratic programming functional}. The minimization can again be achieved by discretization and via quadratic programming in \eqref{eq:quadratic programming standard}, though the forms of $\mathrm{Q},\mathrm{d}$ are different due to the more generalized form of the prior. I.e.
		\begin{align}
			\begin{split}
				&\min_{\bm\A}\quad\frac{1}{2}\bm \A^\top\bm{\mathrm{Q}}\bm \A-\bm \A^\top\bm{\mathrm{d}}
				\\
				&\text{s.t.}\quad
				\begin{pmatrix}
					1&0&0&\cdots&0
					\\
					-1&1&0&\cdots&0
					\\
					0&-1&1&\cdots&0
					\\
					\cdots&\cdots&\cdots&\cdots&\cdots
					\\
					0&\cdots&0&-1&1
				\end{pmatrix}\bm\A\geq
				\begin{pmatrix}
					\alpha(\tau)
					\\
					0
					\\
					\vdots
					\\
					0
				\end{pmatrix}.
			\end{split}
		\end{align}
		where
		\begin{align*}
			\bm{\mathrm{Q}}&=\textup{diag}\left(2(1-\epsilon)\phi(\bm z)+\epsilon\int_0^\infty\Big[\phi(\bm z-t)+\phi(-\bm z-t)\Big]p_{\pi^\star}(t)dt\right),
			\\
			\bm{\mathrm{d}}&=2(1-\epsilon)\bm z\phi(\bm z)+\epsilon\int_0^\infty\Big[(\bm z-t)\phi(\bm z-t)+(\bm z+t)\phi(\bm z+t)\Big]p_{\pi^\star}dt,
		\end{align*}
		\item The smallest $\tau$ that is valid, i.e. $F_\alpha[\mathrm{A}_\textup{eff}^*,p_{\pi^\star}]\leq\delta$ with $\mathrm{A}_\textup{eff}^*$ being the optimal penalty from the above quadratic programming, can be shown to correspond to the largest zero-threshold $\alpha$ by \Cref{fact:alpha tau}.
		\item The largest zero-threshold $\alpha$ gives the minimum FDP via (3.8):
		$$\frac{2(1-\epsilon)\Phi(-\alpha(u))}{2(1-\epsilon)\Phi(-\alpha(u))+\epsilon u}.$$ 
	\end{enumerate}
	
	\begin{fact}
		\label{fact:alpha tau}
		Fixing the prior $\Pi=\pi\tau$ and under the condition $\int_0^\infty [\Phi(t-\alpha)+\Phi(-t-\alpha)]p_{\pi^\star}(t)dt=u$, we have $\frac{d\alpha}{d\tau}<0$.
	\end{fact}
	\begin{proof}[Proof of \Cref{fact:alpha tau}]
		\begin{align*}
			&\mathbb{E}[\Phi(\Pi/\tau-\alpha)+\Phi(-\Pi/\tau-\alpha)]=u
			\\
			\Longrightarrow\quad
			&\mathbb{E}[-(\frac{\Pi}{\tau^2}+\frac{d\alpha}{d\tau})\phi(\Pi/\tau-\alpha)+(\frac{\Pi}{\tau^2}-\frac{d\alpha}{d\tau})\phi(-\Pi/\tau-\alpha)]=0
			\\
			\Longrightarrow\quad
			&\frac{d\alpha}{d\tau}=\frac{\mathbb{E}\left[-(\frac{\Pi}{\tau^2})[\phi(\Pi/\tau-\alpha)-\phi(-\Pi/\tau-\alpha)]\right]}{\mathbb{E}[\phi(\Pi/\tau-\alpha)+\phi(-\Pi/\tau-\alpha)]}<0.
		\end{align*}

\vspace{-0.5cm}	
\end{proof}


\end{document}